\documentclass[letterpaper]{article}

\RequirePackage{amsthm,amsmath,amsfonts,amssymb}
\RequirePackage[numbers]{natbib}
\usepackage{bbm}
\usepackage{hyperref}
\usepackage{latexsym}
\usepackage{enumitem}
\usepackage{tikz}
\usepackage{anyfontsize}
\usepackage[T1]{fontenc}
\usepackage[margin=1in]{geometry}
\usepackage{setspace}
\usepackage{pgfplots}

\pgfplotsset{compat=1.18}

\theoremstyle{plain}
\newtheorem{theorem}{Theorem}[section]
\newtheorem{lemma}[theorem]{Lemma}

\newtheorem{proposition}[theorem]{Proposition}
\theoremstyle{remark}
\newtheorem{remark}{Remark}
\newtheorem{definition}{Definition}

\def\R{\mathbb{R}}
\def\Z{\mathbb{Z}}

\DeclareMathOperator{\Var}{Var}
\DeclareMathOperator{\Cov}{Cov}
\DeclareMathOperator{\Uniform}{Uniform}
\DeclareMathOperator{\Rademacher}{Rademacher}
\newcommand{\indep}{\perp \!\!\! \perp}

\hypersetup{
    colorlinks=true,
    linkcolor=blue,
    filecolor=blue,      
    urlcolor=cyan,
    citecolor=red
}

\title{Optimal heteroskedasticity testing in nonparametric regression}
\author{Subhodh Kotekal\(^1\) and Soumyabrata Kundu\(^2\) \\ \textit{University of Chicago}}
\date{}

\begin{document}
\maketitle

\footnotetext[1]{Email: \texttt{skotekal@uchicago.edu}. The research of SK is supported in part by NSF
Grant ECCS-2216912. }
	\footnotetext[2]{Email: \texttt{soumyabratakundu@uchicago.edu}.}

\begin{abstract}
	Heteroskedasticity testing in nonparametric regression is a classic statistical problem with important practical applications, yet fundamental limits are unknown. Adopting a minimax perspective, this article considers the testing problem in the context of an \(\alpha\)-H\"{o}lder mean and a \(\beta\)-H\"{o}lder variance function. For \(\alpha > 0\) and \(\beta \in (0, 1/2)\), the sharp minimax separation rate \(n^{-4\alpha} + n^{-4\beta/(4\beta+1)} + n^{-2\beta}\) is established. To achieve the minimax separation rate, a kernel-based statistic using first-order squared differences is developed. Notably, the statistic estimates a proxy rather than a natural quadratic functional (the squared distance between the variance function and its best \(L^2\) approximation by a constant) suggested in previous work.

    The setting where no smoothness is assumed on the variance function is also studied; the variance profile across the design points can be arbitrary. Despite the lack of structure, consistent testing turns out to still be possible by using the Gaussian character of the noise, and the minimax rate is shown to be \(n^{-4\alpha} + n^{-1/2}\). Exploiting noise information happens to be a fundamental necessity, as consistent testing is impossible if nothing more than zero mean and unit variance is known about the noise distribution. Furthermore, in the setting where the variance function is \(\beta\)-H\"{o}lder but heteroskedasticity is measured only with respect to the design points, the minimax separation rate is shown to be \(n^{-4\alpha} + n^{-\left((1/2) \vee (4\beta/(4\beta+1))\right)}\) when the noise is Gaussian and \(n^{-4\alpha} + n^{-4\beta/(4\beta+1)} + n^{-2\beta}\) when the noise distribution is unknown.
\end{abstract}

\section{Introduction}\label{section:intro}    
Consider the heteroskedastic nonparametric regression model
\begin{equation}\label{model}
    Y_i = f(x_i) + V^{1/2}(x_i) \xi_i
\end{equation}
for \(i=0,...,n\) where \(x_i = \frac{i}{n}\) are fixed design points in the unit interval and the noise variables \(\xi_i\) are independent and identically distributed with zero mean, unit variance, and finite fourth moment. The functions \(f\) and \(V\) are unknown and denote the mean and variance respectively. 

Nonparametric regression is, by now, a classical topic in statistics. Much of the existing work and theoretical results concern the homoskedastic setting where the variance function \(V\) is constant. Though the heteroskedastic case is more practically relevant, few rigorous results are available in the literature. It is well known that optimal estimation requires accounting for the heteroskedasticity of the noise and simpler estimators employed in the homoskedastic case should not be used as the statistical cost can be quite large. As has been widely acknowledged, the problem of detecting heteroskedasticity is an important one.

Heteroskedasticity testing has a long history in statistics. For example, it is relevant in the basic statistical problem of testing the equality of means of two groups. Specifically, it is preferable to use a pooled \(t\)-test if the variances were equal, and similarly when testing equality of means of multiple groups in ANOVA. In other contexts, the equality of the group variances may itself be the scientific question of interest. Levene's test \cite{levene_robust_1960} is a celebrated method used to test for heteroskedasticity in this setting. The Brown-Forsythe test \cite{brown_robust_1974} was developed as an improvement to Levene's test for robustness to nonnormality. We point the interested reader to \cite{gastwirth_impact_2009} for a detailed discussion of the history of heteroskedasticity testing in the context of testing equality of group means (and other considerations in ANOVA). 

Beyond ANOVA, heteroskedasticity testing has seen much development in the older statistical literature, most notably in the context of linear regression. The problem has attracted wide interest from a variety of communities, including statisticians, biostatisticians, and econometricians. Though the usual ordinary least squares (OLS) estimator is still consistent in the heteroskedastic setting, it is no longer efficient and the usual reported confidence intervals require modification; thus, it is important to test for heteroskedasticity. From this literature, the Breusch-Pagan test \cite{breusch_simple_1979} has emerged as a popular choice for the practitioner. 

The Breusch-Pagan test assumes normality of the error and assumes the variance for the \(i\)th observation can be written as \(\sigma^2_i = h(\langle Z_i, \alpha\rangle)\) where \(h\) is a known twice-differentiable function, \(\alpha \in \R^d\) is an unknown coefficient vector, and \(Z_i \in \R^d\) are auxiliary covariates (with the first coordinate equal to \(1\) to incorporate an intercept) satisfying some regularity conditions. In this setup, the null hypothesis of homoskedasticity is equivalent to the hypothesis \(\alpha_2 = ... = \alpha_d = 0\). The test statistic is obtained by taking the squared residuals from OLS, regressing them on \(\{Z_i\}_{i=1}^{n}\), and computing the sum of the resulting squared residuals. Breusch and Pagan \cite{breusch_simple_1979} show the statistic (after normalization) is asymptotically distributed according to the \(\chi^2_{d-1}\) distribution under the null hypothesis. Though the conceptual elegance and computational simplicity of the test is attractive, it suffers from some serious drawbacks. In particular, \cite{breusch_simple_1979} assumes a particular specification of the heteroskedasticity. Furthermore, normality of the error term and some regularity conditions on \(\{Z_i\}_{i=1}^{n}\) are assumed. One cannot broadly consider situations where the variances generically depend on the main covariates \(\{X_i\}_{i=1}^{n}\). These drawbacks are shared in various degrees by most procedures suggested in this early literature concerning linear regression. For example, the widely used Goldfeld-Quandt test \cite{goldfeld_tests_1965} is only suited to testing the presence of a specific form of heteroskedasticity, namely that which is inflated only by a pre-specified covariate. Though the original proposal requires normality of the error term, Goldfeld and Quandt \cite{goldfeld_tests_1965} also provide a test to handle nonnormal error distributions. However, more general heteroskedasticity behavior is not addressed.

The celebrated work of White \cite{white_heteroskedasticity-consistent_1980} proposes, in addition to a renowned heteroskedasticity correction for standard errors, a test for heteroskedasticity. Notably, no particular form of the heteroskedasticity is assumed. However, as White notes, the test he proposes requires that the linear model is correctly specified. In other words, if the statistician, who is agnostic about model specification, rejects the null hypothesis, then it can only be concluded that either there is heteroskedasticity present or the model is misspecified. If, however, the statistician is willing to assume the correctness of the linear model, then the presence of heteroskedasticity can be reasonably concluded. Thus, the impressive results of White \cite{white_heteroskedasticity-consistent_1980} are tightly tied to the linear model in the context of heteroskedasticity testing. 

Since this early work, the literature has matured beyond the linear model. In a broad regression setting, an old and intuitive idea is to examine the residuals after fitting an estimator developed for homoskedastic models. The early literature considered special cases for the mean function and the variance function (e.g. parametric family or tied to \(f\) by a link function) \cite{bickel_using_1978, white_heteroskedasticity-consistent_1980,harrison_test_1979,koenker_robust_1982}. When parametric forms are imposed, score-based tests are also natural \cite{cook_diagnostics_1983,breusch_simple_1979}. In the fully nonparametric case, a residual-based diagnostic was first proposed by Eubank and Thomas \cite{eubank_detecting_1993}, though their statistic requires choosing weights based on possible forms for the unknown variance function. After estimating a putative variance from residuals as if homoskedasticity was in force, Zheng \cite{zheng_testing_2009} constructed a test statistic by applying kernel-smoothing to the difference between the squared residuals and the estimated putative variance. A different method based on an empirical process from residuals has also been proposed \cite{chown_detecting_2018}. 

An alternative to residual-based methods are those based on difference sequences (or pseudoresiduals) \cite{gasser_residual_1986,hall_asymptotically_1990}. Such methods use differences between neighboring responses to examine the variance function, relying on the smoothness of \(f\) to assert the corresponding mean differences are small and do not pose too big a nuisance. Employing these differences, empirical process methods have been developed in special cases \cite{dette_new_2007,dette_simple_2009}. In the fully nonparametric case, the most developed residual-based results are due to \cite{dette_testing_1998,dette_consistent_2002}. 

To the best of our knowledge, fundamental limits for heteroskedasticity detection in the fully nonparametric model (\ref{model}) have not been established in the literature. We adopt a minimax testing framework \cite{ingster_nonparametric_2003,baraud_non-asymptotic_2002} and consider the detection problem in the setting where \(f\) is an \(\alpha\)-H\"{o}lder function. Specifically, define the H\"{o}lder class of functions for \(\alpha, M > 0\)
\begin{align}
    \mathcal{H}_\alpha(M) = &\left\{ g : [0, 1] \to \R : \left|g^{(\lfloor \alpha \rfloor)}(x) - g^{(\lfloor \alpha \rfloor)}(y)\right| \leq M|x-y|^{\alpha - \lfloor \alpha \rfloor} \text{ for all } x, y \in [0, 1] \right. \nonumber \\
    &\left. \;\;\;\;\;\text{and } ||g^{(k)}||_\infty \leq M \text{ for all } k \in \{0,...,\lfloor \alpha \rfloor\} \right\}. \label{space:Holder} 
\end{align}
Throughout, we will assume \(M\) is a sufficiently large constant and will frequently write \(\mathcal{H}_\alpha\) for notational ease. 

As noted in our earlier discussion and evidenced by the literature (e.g. \cite{white_heteroskedasticity-consistent_1980}), it is desirable to be completely flexible and not assume any particular form for \(V\). In Section \ref{section:profile}, we formulate and investigate the problem of heteroskedasticity testing where the variance profile may be completely arbitrary. To the best of our knowledge, no article in the literature accommodates an arbitrary variance profile for heteroskedasticity testing in nonparametric regression. Previous work (e.g. \cite{dette_testing_1998,dette_consistent_2002}) all require smoothness assumptions on \(V\). In fact, the proposed tests \cite{dette_testing_1998,dette_consistent_2002} are only shown to work under the assumption \(V\) has high enough smoothness above some level. We aim to fill this gap in the literature. 

Of course, it is likely heteroskedasticity testing becomes fundamentally easier when smoothness assumptions are imposed on \(V\). Consequently, it is also desirable to pin down statistical limits in this structured, but still nonparametric, setting and provide a test which exploits the available smoothness. For presentation purposes, we will first address the heteroskedasticity testing with smooth \(V\) before moving onto an arbitrary variance profile in Section \ref{section:profile}.

Following \cite{dette_testing_1998,dette_consistent_2002}, the case where \(V\) is a \(\beta\)-H\"{o}lder function will be considered. To formulate a minimax theory for heteroskedasticity testing, define the parameter spaces 
\begin{align}
    \mathcal{V}_0 &= \left\{V : [0, 1] \to [0, \infty) : V \equiv \sigma^2 \text{ for some } 0 \leq \sigma^2 \leq M\right\}, \label{space:V0} \\
    \mathcal{V}_{1, \beta}(\varepsilon) &= \left\{V \in \mathcal{H}_\beta : V \geq 0 \text{ and } ||V - \bar{V}\mathbf{1}||_2 \geq \varepsilon\right\} \label{space:V1} 
\end{align}
for \(\beta, \varepsilon > 0\). Here, \(\bar{V} = \int_{0}^{1} V(x) \,dx\) and \(\mathbf{1}\) is the constant function on the unit interval with value equal to one. The quantity \(||V - \bar{V}\mathbf{1}||_2\) is used to measure the deviation of \(V\) from homoskedasticity following \cite{dette_testing_1998}, and so \(\varepsilon\) determines the separation between the two spaces (\ref{space:V0}) and (\ref{space:V1}). Formally, the problem of testing for heteroskedasticity is 
\begin{align}
    H_0 &: V \in \mathcal{V}_0 \text{ and } f \in \mathcal{H}_\alpha, \label{problem:var0}\\
    H_1 &: V \in \mathcal{V}_{1, \beta}(\varepsilon) \text{ and } f \in \mathcal{H}_\alpha. \label{problem:var1}
\end{align}
The goal, following the minimax testing literature \cite{ingster_nonparametric_2003,baraud_non-asymptotic_2002}, is to characterize the minimax separation rate, that is, the necessary and sufficient separation between (\ref{space:V0}) and (\ref{space:V1}) such that successful detection is possible. To define the minimax separation rate, first define the testing risk 
\begin{equation}\label{def:testing_risk}
    \mathcal{R}(\varepsilon) = \inf_{\varphi}\left\{ \sup_{\substack{f \in \mathcal{H}_\alpha, \\ V \in \mathcal{V}_0}} P_{f,V}\left\{ \varphi = 1\right\} + \sup_{\substack{f \in \mathcal{H}_\alpha, \\ V \in \mathcal{V}_{1, \beta}(\varepsilon)}} P_{f,V}\left\{ \varphi = 0 \right\} \right\}
\end{equation}
where the infimum runs over all tests \(\varphi\) (i.e. measurable, binary-valued functions with the data \(\{Y_i\}_{i=0}^{n}\) as input). With the testing risk in hand, the minimax separation rate can be defined. 

\begin{definition}
    Fix \(\alpha, \beta > 0\). We say \(\varepsilon^* = \varepsilon^*(\alpha, \beta)\) is the minimax separation rate for (\ref{problem:var0})-(\ref{problem:var1}) if for all \(\eta \in (0, 1)\),
    \begin{enumerate}[noitemsep, label=(\roman*)]
        \item there exists \(C_\eta > 0\) not depending on \(n\) such that for all \(C > C_\eta\) we have \(\mathcal{R}(C\varepsilon^*) \leq \eta\),
        \item there exists \(c_\eta > 0\) not depending on \(n\) such that for all \(0 < c < c_\eta\) we have \(\mathcal{R}(c\varepsilon^*) \geq 1-\eta\),
    \end{enumerate}
    where \(\mathcal{R}\) is given by (\ref{def:testing_risk}). 
\end{definition}
The minimax separation rate \(\varepsilon^*(\alpha, \beta)\) characterizes the fundamental limit of the heteroskedasticity detection problem (\ref{problem:var0})-(\ref{problem:var1}). Importantly, a composite set of alternatives is considered rather than a shrinking, local alternative in a fixed direction considered by earlier work \cite{dette_consistent_2002}. In our view, the minimax perspective is more meaningful given the nonparametric setting. For the entirety of what follows, including the proofs, we will assume $\alpha$ and $\beta$ are fixed, and any constants introduced implicitly or explicitly can potentially depend on them.

\subsection{Related work}\label{section:related_work}

Some existing works have studied heteroskedasticity detection in (\ref{model}) under the H\"{o}lder setting and are worthy of further discussion. As noted earlier, the most developed results in the fully nonparametric case are due to Dette and Munk \cite{dette_testing_1998} and later to Dette \cite{dette_consistent_2002}. The earlier article by Dette and Munk \cite{dette_testing_1998} consider the case \(\alpha, \beta > \frac{1}{2}\). Taking the first-order differences \(R_i = Y_{i+1} - Y_i\) for \(i=0,...,n-1\), they define the test statistic 
\begin{equation*}
    \hat{T} = \frac{1}{4(n-2)}\sum_{i=0}^{n-3}R_i^2R_{i+2}^2 - \left(\frac{1}{2n} \sum_{i=0}^{n-1} R_i^2 \right)^2
\end{equation*}
and establish the convergence in distribution as \(n \to \infty\),
\begin{equation*}
    4\sqrt{n}\left(\hat{T} - ||V-\bar{V}\mathbf{1}||_2^2\right) \overset{d}{\longrightarrow} N(0, \tau^2_V),
\end{equation*}
for some asymptotic variance \(\tau^2_V\) for which they obtain an explicit formula. Dette and Munk \cite{dette_testing_1998} also provide an estimator for \(\tau^2_V\) and thus furnish a consistent test for heteroskedasticity. Though their test is appealing in the statistic's simplicity, a number of questions remain unresolved. First, their test is shown only to work under smoothness assumptions on \(V\), and even so nothing is known when either \(\alpha\) or \(\beta\) is below \(\frac{1}{2}\). Second, there is no guarantee the test is optimal. Finally, the convergence rate does not depend on either H\"{o}lder exponent. Heteroskedasticity testing is expected to become easier as \(V\) becomes smoother. Likewise, the effect of the nuisance mean function is not reflected in their result; for example, the problem ought to be easier if \(f\) is known to be constant than if \(f\) were quite rough. The result of Dette and Munk \cite{dette_testing_1998} only guarantees that alternatives converging to the null at \(n^{-1/2}\) rate (in squared norm) can be detected. 

Motivated by these drawbacks, Dette constructed, in the later article \cite{dette_consistent_2002}, a kernel-based test statistic (which is a modified version of that proposed by Zheng \cite{zheng_consistent_1996})
\begin{equation*}
    \hat{T} = \frac{1}{4n(n-1)h} \sum_{|i-j| \geq 2} K\left(\frac{x_i - x_j}{h}\right) \left(R_{i}^2 - \overline{R^2}\right)\left(R_j^2 - \overline{R^2}\right)
\end{equation*}
where \(K\) is a kernel satisfying the typical assumptions placed for density estimation and \(\overline{R^2} = \frac{1}{n} \sum_{i=0}^{n-1} R_i^2\). Under the assumption \(\alpha, \beta \geq \frac{1}{2}\), \(h \to 0\), and \(nh^2 \to \infty\), Dette showed under the null hypothesis
\begin{equation*}
    n\sqrt{h} \hat{T} \overset{d}{\longrightarrow} N(0, \tau^2_0)
\end{equation*}
and under the alternative hypothesis 
\begin{equation*}
    \sqrt{n}\left(\hat{T} - \frac{1}{h} \iint K\left(\frac{x-y}{h}\right) \left(V(x) - \bar{V}\right)\left(V(y) - \bar{V}\right) \,dx\,dy\right) \overset{d}{\longrightarrow} N(0, \tau^2_V)
\end{equation*}
as \(n \to \infty\). As Dette notes, \(\frac{1}{h} \iint K\left(\frac{x-y}{h}\right) \left(V(x) - \bar{V}\right)\left(V(y) - \bar{V}\right) \,dx\,dy = ||V - \bar{V}\mathbf{1}||_2^2 + o(1)\) as \(n \to \infty\), and so a consistent test is obtained by rejecting the null when \(n \sqrt{h} \hat{T}\) is large. The convergence rates are different under the null and alternative hypotheses. However, Dette \cite{dette_consistent_2002} also shows asymptotic normality still holds under local alternatives which shrink along a fixed direction at a rate \((n\sqrt{h})^{-1/2}\). In this sense, with an appropriate choice of bandwidth, the test can detect local alternatives shrinking at a faster rate than that obtained earlier by Dette and Munk \cite{dette_testing_1998}.

Though Dette \cite{dette_consistent_2002} improves upon the earlier work of Dette and Munk \cite{dette_testing_1998}, many of the unresolved questions remain unresolved. There is still no optimality guarantee and still nothing is known when no smoothness is assumed or even when \(\alpha\) or \(\beta\) is below \(\frac{1}{2}\). Furthermore, we find it difficult to motivate the choice of examining local alternatives along a fixed direction in the current nonparametric setup. Worst case guarantees from a minimax perspective are more appropriate when considering function space. It is typically the case that there is no strong prior information for the statistician to commit to a specific direction in the infinite-dimensional parameter space. Furthermore, a narrow focus on directional alternatives can lead to misleading conclusions, as has been the case for multivariate goodness-of-fit testing in which the prototypical curse of dimensionality apparently was absent in earlier work (see \cite{arias-castro_remember_2018} for detailed discussion). Hence, we find it more meaningful to work under the minimax framework. All in all, the articles \cite{dette_testing_1998,dette_consistent_2002} develop interesting methodology but some foundational questions remain open.

Estimation of the variance function is a related problem. Wang et al. \cite{wang_effect_2008} showed that difference-based estimators can achieve better performance than residual-based estimators. It turns out to be more beneficial to use an estimator of \(f\) with minimal bias rather than plug in a traditional estimator. Typically, traditional nonparametric regression estimators are constructed to balance the bias and stochastic error. In the model (\ref{model}), Wang et al. \cite{wang_effect_2008} established the minimax rate for estimating \(V\) is
\begin{equation}\label{rate:estimation}
    \inf_{\hat{V}} \sup_{\substack{f \in \mathcal{H}_\alpha, \\ V \in \mathcal{H}_\beta, \\ V \geq 0}} E_{f,V}\left(||\hat{V} - V||_2^2\right) \asymp n^{-4\alpha} + n^{-\frac{2\beta}{2\beta+1}}.
\end{equation}
Notably, if the mean function is sufficiently smooth, i.e. \(4\alpha > \frac{2\beta}{2\beta+1}\), then the difficulty of estimation is driven entirely by \(V\) as if \(f\) was known completely. Likewise, the rate is dominated by the effect of the mean if \(f\) is very rough. The estimator constructed by Wang et al. \cite{wang_effect_2008} is a kernel estimator based on the squared first-order differences of the responses. 

Returning to the testing problem (\ref{problem:var0})-(\ref{problem:var1}) with the optimal estimator \(\hat{V}\) from \cite{wang_effect_2008} in hand, the quantity \(||V - \bar{V}\mathbf{1}||_2\) can be estimated via plug-in. This plug-in estimator can be used as a testing statistic, yielding the upper bound \(\varepsilon^*(\alpha, \beta)^2 \lesssim n^{-4\alpha} + n^{-\frac{2\beta}{2\beta+1}}\). However, this is not expected to be the sharp rate as a frequent theme in nonparametric statistics is the capacity to test at faster rates than to estimate, meaning plug-in tests are usually suboptimal \cite{ingster_nonparametric_2003,tsybakov_introduction_2009,gine_mathematical_2016}.

We will focus on three main questions which are currently unaddressed in the literature. First, a sharp characterization of the minimax separation rate \(\varepsilon^*(\alpha, \beta)\) has not been established. None of the aforementioned articles provide any optimality guarantees in any smoothness regimes for the tests they propose. Pinning down the fundamental statistical limits is an important and basic problem to resolve. Second, to the best of our knowledge, there are no procedures in the literature which are able to test heteroskedasticity in (\ref{model}) assuming \emph{no} smoothness on the variance function, despite existing results for the linear model (e.g. \cite{white_heteroskedasticity-consistent_1980}). Third, we will address adaptation to the H\"{o}lder exponent of the mean function. In some testing and functional estimation problems, it is widely known a logarithmic cost in the statistical rate is unavoidable \cite{spokoiny_adaptive_1996,lepski_problem_1990}. It remains to be seen whether adaptation is possible for heteroskedasticity testing.

\subsection{Main contributions}\label{section:main_contribution}
The main contribution of this work is to establish the minimax rates in Table \ref{table:rates}. A variety of settings are considered and are discussed in turn. 

\begin{table}[ht]
    \renewcommand{\arraystretch}{2.1}
    \centering
    \begin{tabular}{|c|c|c|}
        \hline
        Heteroskedasticity & Gaussian noise & Unknown noise \\
        \hline 
        \(L^2\) & \(n^{-4\alpha} + n^{-\frac{4\beta}{4\beta+1}} + n^{-2\beta}\) & \(n^{-4\alpha} + n^{-\frac{4\beta}{4\beta+1}} + n^{-2\beta}\) \\
        \hline 
        Design & \(n^{-4\alpha} + n^{-\left(\frac{1}{2} \vee \frac{4\beta}{4\beta+1}\right)}\) & \(n^{-4\alpha} + n^{-\frac{4\beta}{4\beta+1}} + n^{-2\beta}\) \\
        \hline
    \end{tabular}
    \caption{A summary of minimax rates for \(\alpha > 0\) and \(\beta \in (0, \frac{1}{2})\). The \(L^2\) norm is used to measure heteroskedasticity in (\ref{problem:var0})-(\ref{problem:var1}). In Sections \ref{section:profile} and \ref{section:discrete_loss}, the heteroskedasticity is measured only across the design points. }
    \label{table:rates}
\end{table}

For the heteroskedasticity testing problem (\ref{problem:var0})-(\ref{problem:var1}), where the heteroskedasticity is measured with respect to the \(L^2\) norm, under the model (\ref{model}), in which the precise distribution of the noise is unknown, we establish the sharp rate
\begin{equation}\label{rate:testing}
    \varepsilon^*(\alpha, \beta)^2 \asymp n^{-4\alpha} + n^{-\frac{4\beta}{4\beta+1}} + n^{-2\beta}
\end{equation}
for \(\alpha > 0\) and \(\beta \in \left(0, \frac{1}{2}\right)\). The rate is the same if the noise is known to be standard Gaussian. Notably, the testing rate (\ref{rate:testing}) shares the term \(n^{-4\alpha}\) with the estimation rate (\ref{rate:estimation}). In particular, when \(4\alpha \leq \frac{4\beta}{4\beta+1} \wedge 2\beta\) the problem of detecting heteroskedasticity is equally as hard (from the perspective of rates) as estimating the entire function \(V\). In other words, there is no statistical cost to using the plug-in test when the mean function is sufficiently rough. 

Previous work \cite{dette_testing_1998,dette_consistent_2002} has essentially viewed the heteroskedasticity detection as a problem of estimating the quadratic functional \(||V - \bar{V}\mathbf{1}||_2^2\). Instead, we aim at a different target \(\frac{1}{n} \sum_{i=0}^{n-1}\left(W_i + \delta_i^2 - \bar{W}_n - \overline{\delta^2_n}\right)^2\) where \(W_i = V\left(\frac{i+1}{n}\right) + V\left(\frac{i}{n}\right)\) and \(\delta_i = f\left(\frac{i+1}{n}\right) - f\left(\frac{i}{n}\right)\) for \(i=0,...,n-1\). Further, \(\bar{W}_n = \frac{1}{n} \sum_{i=0}^{n-1} W_i\) and \(\overline{\delta^2_n} = \frac{1}{n} \sum_{i=0}^{n-1} \delta_i^2\). This target acts as a proxy for the quadratic functional in that it is large when \(||V - \bar{V}\mathbf{1}||_2^2\) is large, thus containing signal to enable heteroskedasticity detection. The statistic we construct from the first-order squared differences estimate this proxy target.

In the regime \(\frac{1}{4} < \beta < \frac{1}{2}\), the rate (\ref{rate:testing}) simplifies to \(n^{-4\alpha} + n^{-\frac{4\beta}{4\beta+1}}\) and suggests smoothing is necessary. We construct a kernel-based test statistic in the spirit of \cite{dette_consistent_2002} (and in turn \cite{zheng_consistent_1996}). However, our construction departs from earlier work by carefully modifying the kernel function to further reduce bias. Moreover, the bias calculation departs from \cite{dette_consistent_2002} and adapts an idea of Nickl and Gin\'{e} \cite{gine_simple_2008} for estimation of a density quadratic functional in order to sharply capture the \(n^{-\frac{4\beta}{4\beta+1}}\) scaling. However, the fixed design setting of (\ref{model}) poses a substantial technical complication in successfully adapting this idea from \cite{gine_simple_2008}. When \(\beta < \frac{1}{4}\), the rate (\ref{rate:testing}) simplifies to \(n^{-4\alpha} + n^{-2\beta}\). The \(n^{-2\beta}\) term can be attributed to the discretization of the unit interval at the fixed design points in (\ref{model}). The familiar nonparametric term \(n^{-\frac{4\beta}{4\beta+1}}\) disappears, suggesting smoothing is not beneficial. Indeed, we use the same kernel-based statistic as in the case \(\frac{1}{4} < \beta < \frac{1}{2}\), but now take the bandwidth to be of order \(n^{-1}\).

Heteroskedasticity testing under no smoothness assumptions on \(V\) is a particularly important problem, in which the variance profile across the design points can be arbitrary. Specifically, we consider the problem (\ref{problem:var0_profile})-(\ref{problem:var1_profile}) in the model (\ref{model:profile}). In Section \ref{section:profile_triviality}, we show that consistent testing is impossible if the noise distribution is unknown as in (\ref{model}). However, if the noise is known to be standard Gaussian as in (\ref{model:profile}), we establish in Section \ref{section:profile} the minimax rate
\begin{equation*}
    \varepsilon^*(\alpha)^2 \asymp n^{-4\alpha} + n^{-1/2}.
\end{equation*}
Notably, consistent heteroskedasticity detection is still possible. The result identifies that the rate \(n^{-1/2}\) obtained by Dette and Munk \cite{dette_consistent_2002} actually holds without the smoothness assumptions they impose (provided \(\alpha\) is large enough). Conceptually, the case of an arbitrary variance profile corresponds to \(\beta = 0\) in the design heteroskedasticity row of Table \ref{table:rates}.

As there is no smoothness to exploit, a simpler statistic reminiscent of that proposed by Dette and Munk \cite{dette_testing_1998} (though they require \(\alpha, \beta > \frac{1}{2}\)) is used instead of a kernel-based one. The proposed statistic critically uses the Gaussianity of the noise, which is in contrast to the case where \(V\) is smooth and precise noise information becomes unnecessary. Additionally, the proxy used in place of the quadratic functional is more complicated due to the lack of smoothness. In the smooth setting, the target proxy is a function of the local aggregates \(W_i = V_i + V_{i+1}\), which is close (up to a factor of two) to \(V_i\) and \(V_{i+1}\) by the H\"{o}lder assumption. In essence, the variance profile \((V_1,...,V_n)\) can be approximately reconstructed from \(\{W_i\}_i\). However, this is no longer true when no smoothness is assumed; additional information is needed. For example, the profile for \(1 \leq i \leq n-2\) can be recovered from the collection \(\{V_i + V_{i+1}, V_{i} + V_{i+2}\}_i\) since \(V_i = \frac{1}{2}\left((V_i + V_{i+1}) + (V_i + V_{i+2}) - (V_{i+1} + V_{i+2})\right)\). The proxy we aim at incorporates additional information to overcome the lack of smoothness. We use not only the first-order differences \(\left\{Y_{i+1} - Y_i\right\}_i\), but also the differences \(\left\{Y_{i+2} - Y_i\right\}_i\) and \(\left\{Y_{i+3} - Y_i\right\}_i\).

The test statistic we propose requires knowledge of \(\alpha\) in setting the cutoff value defining the rejection criterion. The H\"{o}lder exponent is typically unknown in practice, so we consider adaptation to \(\alpha\) in Section \ref{section:adaptation}. By adaptation, we mean the existence of a testing procedure which does not require knowledge of \(\alpha\) but achieves the minimax rate as if it were known. It turns out adaptation to \(\alpha\) is completely impossible. Roughly speaking, if the statistician is faced with two possible smoothness classes for the mean function with H\"{o}lder exponents \(\alpha_1 < \alpha_2\), then the best possible testing procedure can only achieve the slower rate associated to \(\alpha_1\) even if the true mean function is \(\alpha_2\)-H\"{o}lder. This impossibility result is surprising given that adaptation is well-known to be possible (albeit with unavoidable logarithmic costs) in other hypothesis testing and quadratic functional estimation problems \cite{spokoiny_adaptive_1996,efromovich_optimal_1996,cai_optimal_2006}.

It is interesting to see noise information makes the heteroskedasticity testing problem possible in the setting of an arbitrary variance profile. It turns out noise information generically enables faster testing rates. In Section \ref{section:discrete_loss}, we study the setting of a \(\beta\)-H\"{o}lder variance function but where the heteroskedasticity is measured only with respect to the design points. It is shown that if the noise distribution is unknown (beyond mean zero, unit variance, and finite fourth moment), the minimax separation rate is \(\varepsilon^*(\alpha, \beta)^2 \asymp n^{-4\alpha} + n^{-\frac{4\beta}{4\beta+1}} + n^{-2\beta}\) for \(\alpha > 0\) and \(\beta \in (0, \frac{1}{2})\). On the other hand, the faster minimax rate \(\varepsilon^*(\alpha,\beta)^2 \asymp n^{-4\alpha} + n^{-\left(\frac{1}{2} \vee \frac{4\beta}{4\beta+1}\right)}\) is available if the noise distribution is known to be Gaussian.

\subsection*{Notation}
The following notation will be used throughout the paper. For \(n \in \mathbb{N}\), let \([n] := \{1,...,n\}\). Denote the set of positive integers by \(\Z_{+}\). For \(a, b \in \R\), denote \(a \vee b := \max\{a, b\}\) and \(a \wedge b = \min\{a , b\}\). Denote \(a \lesssim b\) to mean there exists a universal constant \(C > 0\) such that \(a \leq C b\). The expression \(a \gtrsim b\) means \(b \lesssim a\). Further, \(a \asymp b\) means \(a \lesssim b\) and \(b \lesssim a\). For functions $f,g:\Z\to \R$ define the discrete convolution $\ast_D$ as $(f\ast_Dg)(z) = \sum_{k\in \Z}f(k)g(z-k)$. The total variation distance between two probability measures \(P\) and \(Q\) on a measurable space \((\mathcal{X}, \mathcal{A})\) is defined as \(d_{TV}(P, Q) := \sup_{A \in \mathcal{A}} |P(A) - Q(A)|\). The product measure on the product space is denoted as \(P \otimes Q\). If \(P\) and \(Q\) are probability measures on \(\R\), then \(P * Q\) denotes the convolution of \(P\) and \(Q\), which is the distribution of \(X + Y\) where \(X \sim P\) and \(Y \sim Q\) are independent. If \(Q\) is absolutely continuous with respect to \(P\), the \(\chi^2\) divergence is defined as \(\chi^2(Q, P) := \int_{\mathcal{X}} \left(\frac{dQ}{dP} - 1\right)^2 \, dP\). For \(x \in \R\), the probability measure with full mass at \(x\) is denoted by \(\delta_x\). For a sequence of random variables \(\{X_n\}_{n \geq 1}\) and a random variable \(X\), the notation \(X_n \overset{d}{\longrightarrow} X\) as \(n \to \infty\) is used to denote \(X_n\) converges in distribution to \(X\). For a \(k\)-times differentiable function \(f : \R \to \R\), the function \(f^{(j)}\) denotes the \(j\)th derivative of \(f\) for \(1 \leq j \leq k\). For \(x \in \R\), the quantity \(\lfloor x \rfloor\) denotes the largest integer less than or equal to \(x\). The quantity \(\lceil x \rceil\) denotes the smallest integer greater than or equal to \(x\). For a square-integrable function \(f : [0, 1] \to \R\), the \(L^2\) norm is \(||f||_2 = \sqrt{\int_{0}^{1} f^2(x)\,dx}\). For a bounded function \(f : [0, 1] \to \R\), the \(L^\infty\) norm is \(||f||_\infty = \sup_{0 \leq x \leq 1} |f(x)|\). For a vector \(x \in \R^d\), the \(\ell^2\) norm is \(||x||_2 = \sqrt{\sum_{i=1}^{d} x_i^2}\) and the \(\ell^\infty\) norm is \(||x||_\infty = \max_{1 \leq i \leq d} |x_i|\).

\section{Smooth variance function}
In this section, we focus on testing (\ref{problem:var0})-(\ref{problem:var1}) under the model (\ref{model}). Namely, this section will establish the results in the \(L^2\) heteroskedasticity row of Table \ref{table:rates}.

\subsection{Upper bound}\label{section: upper bound}

In this section, we focus on establishing an upper bound for the minimax separation rate. It is crucial to note that unlike previous work \cite{dette_testing_1998,dette_consistent_2002}, our testing statistic is not intended to estimate the quadratic functional \(||V - \bar{V}\mathbf{1}||^2_2\). Instead, our strategy revolves around devising a test statistic which estimates the proxy
\begin{equation}\label{def: T}
    T = \frac{1}{n} \sum_{i=0}^{n-1}\left(W_i + \delta_i^2 - \bar{W}_n - \overline{\delta^2_n} \right)^2.
\end{equation}

\noindent Here, \(W_i = V\left(\frac{i}{n}\right)+V\left(\frac{i+1}{n}\right)\), \(\delta_i = f\left(\frac{i+1}{n}\right)-f\left(\frac{i}{n}\right)\) for \(0\leq i\leq n-1\), and \(\bar{W}_n\) and \(\overline{\delta^2_n}\) are the respective averages. The proxy $T$ and the quadratic functional $||V - \bar{V}\mathbf{1}||_2^2 $ obey the following inequality.
\begin{proposition}\label{prop: V-V bar less than T}
    Let \(f \in \mathcal{H}_\alpha\) and \(V \in \mathcal{H}_\beta\) with \(\alpha, \beta > 0\). Then we have
    \[
        ||V - \bar{V}\mathbf{1}||_2^2 \lesssim T + n^{-2(\beta \wedge 1)} + n^{-4(\alpha \wedge 1)}
    \]
    where \(T\) is defined as in (\ref{def: T}).
\end{proposition}

\begin{proof}
  When \(\alpha \leq 1\), since \(f \in \mathcal{H}_\alpha\) we can establish that \(|\delta_i| \lesssim n^{-\alpha}\). When \(\alpha > 1\), we have \(|\delta_i|\lesssim n^{-1}\) since \(\mathcal{H}_\alpha\) is a subset of \(\mathcal{H}_1\). Consequently, we can deduce \(\delta_i^2 \lesssim n^{-2(\alpha \wedge 1)}\) for any $\alpha>0$. Similarly, since $V\in \mathcal{H}_\beta$, for any fixed $0\leq i\leq n-1$ and $\frac{i}{n} \leq x\leq \frac{i+1}{n}$ we have $|2V(x) - W_i| \lesssim n^{-(\beta \wedge 1)}$ for any $\beta>0$. Using these bounds and the useful inequality $(a+b)^2 \lesssim a^2 + b^2$ for reals $a$ and $b$, we can deduce that

  \begin{align*}
     ||V - \bar{V}\mathbf{1}||_2^2 
     &= \int_{0}^{1} (V(x) - \bar{V})^2 \, dx \\
     &= \frac{1}{4}\sum_{i=0}^{n-1} \int_{\frac{i}{n}}^{\frac{i+1}{n}} (2V(x) - 2\bar{V})^2\, dx \\
     &= \frac{1}{4}\sum_{i=0}^{n-1} \int_{\frac{i}{n}}^{\frac{i+1}{n}} \left(2V(x) - W_i + W_i + \delta_i^2 - \bar{W}_n - \overline{\delta^2_n} + \bar{W}_n - 2\bar{V} - \delta_i^2 + \overline{\delta^2_n}\right)^2 \, dx \\
     &\lesssim \sum_{i=0}^{n-1} \int_{\frac{i}{n}}^{\frac{i+1}{n}} \left(2V(x) - W_i\right)^2 + \left(W_i + \delta_i^2 - \bar{W}_n - \overline{\delta^2_n}\right)^2 + \left(\bar{W}_n - 2\bar{V}\right)^2 + \left(\delta_i^2 - \overline{\delta^2_n}\right)^2 \, dx \\
     &\lesssim n^{-2(\beta \wedge 1)} + n^{-4(\alpha \wedge 1)} + \frac{1}{n} \sum_{i=0}^{n-1} \left(W_i + \delta_i^2 - \bar{W}_n - \overline{\delta^2_n}\right)^2 \\
     &=  n^{-2(\beta \wedge 1)} + n^{-4(\alpha \wedge 1)} + T.
  \end{align*}
\end{proof}
    Under the null hypothesis, it is clear \(T \lesssim n^{-4(\alpha \wedge 1)}\) and so if $\hat{T}$ estimates $T$ well, then $\hat{T}$ will also be small with high probability. Conversely, under the alternative hypothesis, it follows from Proposition \ref{prop: V-V bar less than T} that for $\alpha, \beta>0$,
    \begin{equation*}
       ||V - \bar{V}\mathbf{1}||_2^2  \lesssim n^{-4(\alpha \wedge 1)} + n^{-2(\beta\wedge1)} + T.
    \end{equation*}
    Therefore, if \( ||V - \bar{V}\mathbf{1}||_2^2\) is large, and if $\hat{T}$ estimates $T$ well, then it must be the case that \( \hat{T}\) is also large. Hence, establishing an appropriate cutoff for \(\hat{T}\) will serve as a test for detecting heteroskedasticity.
    
    This motivates us to construct a statistic \(\hat{T}\) that can estimate \(T\) at an appropriate rate. To accomplish this, we will use a kernel estimator based on first-order differences
    \begin{equation}\label{def: R_i}
        R_i = Y_{i+1} - Y_{i}
    \end{equation}
    for \(i = 0,...,n-1\). Importantly, \(E(R_i^2) = \delta_i^2 +W_i\). First-order differences have been employed in prior works, such as \cite{wang_effect_2008} for estimating the variance function and \cite{dette_consistent_2002} for testing heteroskedasticity.
    
	Let $K$ be a kernel function, supported on $[-1,1]$, symmetric about $0$, bounded above and below by a universal constant, and $\int_{-1}^1K(x) \,dx = 1$. As the design in (\ref{model}) is fixed, we will construct a kernel similar to that constructed in \cite{wang_effect_2008}. For $h>0$ and $t\in \Z$ define,
    \begin{equation}\label{def: kernel}
         K_n^h(t) = {\displaystyle \frac{\displaystyle \int_{|t|/n}^{(|t|+1)/n} \frac{1}{h}K\left(\frac{u}{h}\right) du }{ \displaystyle 1 - \int_{-2/n}^{2/n} \frac{1}{h}K\left(\frac{u}{h}\right) du }} \cdot \mathbbm{1}_{\{|t| \geq 2\}}.
    \end{equation}
The construction via integration between the fixed design points follows the construction in \cite{wang_effect_2008}. However, for bias reasons (which is discussed further in Remark \ref{remark:low_debias}), we restrict to only considering \(|t| \geq 2\) in (\ref{def: kernel}). Consequently, the normalization term is needed in (\ref{def: kernel}) to account for the deletion of some mass. Using this kernel (\ref{def: kernel}) and the first-order differences (\ref{def: R_i}), define the test statistic
\begin{equation}\label{def: testing_statistic beta>1/4}
    \hat{T} = \frac{1}{n}\sum_{|i-j|\geq 2} K_n^h(i-j)R_i^2R_j^2 -  \frac{1}{n^2}\sum_{|i-j|\geq 2}R_i^2R_j^2.
\end{equation}

\noindent The sum in (\ref{def: testing_statistic beta>1/4}) is taken over \(|i-j|\geq 2\) to ensure independence between $R_i$ and $R_j$. This testing statistic resembles a kernel-based statistic similar to the one used in \cite{dette_consistent_2002}.

 \begin{proposition}\label{prop: MSE of T_hat beta>1/4}
    Suppose \(f \in \mathcal{H}_\alpha\) and \(V \in \mathcal{H}_\beta\) where \(\alpha> 0\) and $0<\beta<\frac{1}{2}$. Define $T$ and \(\hat{T}\) as in (\ref{def: T}) and (\ref{def: testing_statistic beta>1/4}) respectively and let $0<h<1$. If
        \begin{equation}\label{eqn: kernel integral condition}
            \left| 1 - \int_{-2/n}^{2/n} \frac{1}{h}K\left(\frac{u}{h}\right) du \right| \geq c,
        \end{equation}
    for some universal constant $c>0$, then we have
      \begin{align}
         E_{f,V}\left(\left|\hat{T} - T\right|^2\right) \lesssim n^{-8(\alpha \wedge 1)}+ h^{4\beta} + h^2 + \frac{1}{n^2h} + \frac{||W-\bar{W}_n\mathbf{1}_n||^2_2}{n^2}. \label{eqn: MSE expression T_hat<1/2}
      \end{align}
    Here, $W = (W_0, \dots, W_{n-1}) \in \R^n$ and $\mathbf{1}_n$ is the vector of ones with length $n$.
\end{proposition}

\begin{remark}\label{remark: convolution}
    The key step in attaining the rate outlined in Proposition \ref{prop: MSE of T_hat beta>1/4} is to derive the term $h^{4\beta}$, which is a component of the bias calculation. In \cite{wang_effect_2008}, Wang et al. obtain $h^{2\beta}$ in the squared bias, but this is suboptimal for our scenario. The additional level of smoothness arises because we are estimating the quadratic functional $||V-\bar{V}\mathbf{1}||_2^2$ rather than estimating the entire variance function $V$ as in \cite{wang_effect_2008}.
    To achieve this rate, we employ an argument similar to \cite{gine_simple_2008}, where we analyze the smoothness of the discrete convolution of the variance function with itself. However, it is important to note that this argument cannot be directly applied due to the fixed design nature of our model (\ref{model}),{ whereas \cite{gine_simple_2008} deals with an i.i.d. model which does not have the fixed design constraints}.
\end{remark}

\begin{remark}\label{remark:low_debias}
	The main distinction between the kernels of \cite{dette_consistent_2002, wang_effect_2008} and (\ref{def: kernel}) lies in the exclusion of weights for $|i-j|\leq 1$ and the subsequent normalization. This construction is done to further reduce bias since the summation in (\ref{def: testing_statistic beta>1/4}) only involves terms for $|i-j|\geq 2$ to ensure the independence of \(R_i^2\) and \(R_j^2\). To illustrate why we exclude weights for \(|i-j| \leq 1\), consider the estimator using a typical kernel which does not have this modification, 
		\begin{equation}\label{def:That_non_delete}
		\hat{T} = \frac{1}{n^2} \sum_{|i-j| \geq 2} \frac{1}{h}K\left(\frac{x_i - x_j}{h}\right) R_i^2R_j^2 - \frac{1}{n^2} \sum_{|i-j| \geq 2} R_i^2R_j^2. 
	\end{equation}
	Recall the proxy \(T\) given by (\ref{def: T}) and note we can write \(T = \left(\frac{1}{n} \sum_{i=0}^{n-1} U_i^2\right) - \bar{U}_n^2\) where \(U_i = W_i + \delta_i^2\) and \(\bar{U}_n\) is the average. For sake of illustrating the need for bias reduction, let us focus only on the bias associated to estimating the first term \(T_1 = \frac{1}{n} \sum_{i=0}^{n-1} U_i^2\) with the first term in (\ref{def:That_non_delete}) which we denote as \(\hat{T}_1\). 
	Noting \(E(R_i^2) = U_i\), the squared bias is 
	\begin{align}
		\left|E\left(\hat{T}_1\right) - T_1\right|^2 &= \left|\frac{1}{n^2h} \sum_{|i-j| \geq 2} K\left(\frac{x_i - x_j}{h}\right) U_i U_j - \frac{1}{n} \sum_{i=0}^{n-1} U_i^2\right|^2 \nonumber \\
		&\lesssim  \left|\frac{1}{n^2h} \sum_{i, j} K\left(\frac{x_i-x_j}{h}\right) U_iU_j - \frac{1}{n} \sum_{i=0}^{n-1} U_i^2\right|^2 + \left|\frac{1}{n^2h} \sum_{|i-j| \leq 1} K\left(\frac{x_i - x_j}{h}\right) U_iU_j\right|^2 \label{eqn:deletion_debias}\\
		&\lesssim h^{4\beta} + \frac{1}{n^2h^2}. \nonumber
	\end{align}
	The first term in (\ref{eqn:deletion_debias}) is obtained via a convolution argument adapted from \cite{gine_simple_2008} as described in Remark \ref{remark: convolution}. The second term follows from boundedness of $K$, $f$ and $V$. However, the latter term $\frac{1}{n^2h^2}$ turns out to be problematic, as will be seen shortly. Calculating the variance of \(\hat{T}\) in a manner similar to that in Proposition \ref{prop: MSE of T_hat beta>1/4}, the mean squared error is bounded as
	\begin{equation*}
		E\left(\left|\hat{T} - T\right|^2\right) \lesssim n^{-8(\alpha \wedge 1)} + h^{4\beta} + h^2 + \frac{1}{n^2h^2} + \frac{1}{n^2h} + \frac{||W - \bar{W}_n\mathbf{1}_n||_2^2}{n^2}.
	\end{equation*}
	Clearly, the term \(\frac{1}{n^2 h}\) is dominated by \(\frac{1}{n^2h^2}\) since \(h < 1\). Therefore, the optimal choice of \(h\) must balance \(h^{4\beta} \asymp \frac{1}{n^2h^2}\), which yields the choice \(h \asymp n^{-\frac{1}{2\beta+1}}\). Consequently, the test furnished from \(\hat{T}\) would only yield the upper bound \(\varepsilon^*(\alpha, \beta)^2 \lesssim n^{-4\alpha} + n^{-\frac{2\beta}{2\beta+1}}\), which no better than the plug-in test using a variance function estimator \cite{wang_effect_2008}. In fact, the statistic proposed by Dette \cite{dette_consistent_2002} is very close to (\ref{def:That_non_delete}) and likewise suffers from the excess squared bias \(\frac{1}{n^2h^2}\) (for example, (A.7) in \cite{dette_consistent_2002}). 

	To eliminate the problematic \(\frac{1}{n^2h^2}\) term, we modify the kernel to debias the statistic. Namely, if it could be guaranteed that \(\frac{1}{h}K\left(\frac{x_i - x_j}{h}\right) = 0\) whenever \(|i-j| \leq 1\), then the second term in (\ref{eqn:deletion_debias}) would vanish and the bad term would be removed. The construction in (\ref{def: kernel}) is engineered with this purpose in mind.
\end{remark}

By an appropriate selection of the bandwidth in Proposition \ref{prop: MSE of T_hat beta>1/4}, we can establish the following theorem.

\begin{theorem}\label{theorem: MSE of T_hat}
    Suppose \(f \in \mathcal{H}_\alpha\) and \(V \in \mathcal{H}_\beta\) where \(\alpha > 0\) and \(0 < \beta < \frac{1}{2}\). Define \(T\) as in (\ref{def: T}) and $\hat{T}$ as in (\ref{def: testing_statistic beta>1/4}). If $h = C_h n^{-\left(\frac{2}{4\beta+1}\wedge 1\right)}$ for a sufficiently large universal constant $C_h>0$, then
      \begin{equation}\label{eqn: MSE expression T_hat}
         E_{f,V}\left(\left|\hat{T} - T\right|^2\right) \lesssim n^{-8\alpha} + n^{-4\beta} + n^{-\frac{8\beta}{4\beta+1}} + \frac{||V-\bar{V}\mathbf{1}||^2_2}{n}.
      \end{equation}
\end{theorem}

\begin{remark}\label{remark:optimal_bandwidth_low}
    By selecting \(h = C_hn^{-\left(\frac{2}{4\beta+1} \wedge 1\right)}\) with \(C_h\) sufficiently large, it is feasible to find \(c\) that satisfies the condition (\ref{eqn: kernel integral condition}) in Proposition \ref{prop: MSE of T_hat beta>1/4}. It is important to note that when \(\beta>\frac{1}{4}\), we opt for a bandwidth \(h\asymp n^{-\frac{2}{4\beta+1}}\) to ensure \(h^{4\beta} \asymp \frac{1}{n^2h} \asymp n^{-\frac{8\beta}{4\beta+1}}\).
    However, when \(\beta\leq \frac{1}{4}\), simply setting \(h\asymp n^{-\frac{2}{4\beta+1}}\) would violate the condition (\ref{eqn: kernel integral condition}). In this case, if we choose \(h\asymp n^{-1}\), we can still satisfy (\ref{eqn: kernel integral condition}) achieve the optimal rate. Therefore, we observe a transition in the bandwidth selection at \(\beta = \frac{1}{4}\).
\end{remark}

\begin{remark}
    The \(h^2\) term in the bound of Proposition \ref{prop: MSE of T_hat beta>1/4} arises due to well-known boundary effects appearing in nonparametric regression. In our low-smoothness setting \(0 < \beta < \frac{1}{2}\), the boundary effects can be simply ignored since \(h < 1\) implies $h^{4\beta} \gtrsim h^2$ in the bound of Proposition \ref{prop: MSE of T_hat beta>1/4}. As noted in Section \ref{section:discussion}, addressing higher smoothness requires proper accounting for the boundary effects. We leave the case \(\beta \geq \frac{1}{2}\) open for future work.
\end{remark}

As promised, Theorem \ref{theorem: MSE of T_hat} gives us a $\hat{T}$ with an appropriate mean squared error (MSE) bound (\ref{eqn: MSE expression T_hat}). The following test will be used to test (\ref{problem:var0})-(\ref{problem:var1}),
\begin{equation}\label{test:optimal_test}
    \varphi^* = \mathbbm{1}_{\left\{\hat{T} > C_\eta' \zeta^2\right\}}    
\end{equation}
where \(\zeta = n^{-2\alpha} + n^{-\beta} + n^{-\frac{2\beta}{4\beta+1}}\) and \(C_\eta' > 0\) is a constant chosen to achieve testing risk \(\eta\). Utilizing Proposition \ref{prop: V-V bar less than T}, we have the following theorem.

\begin{theorem}\label{theorem: optimal test beta>1/4}
  Suppose \(\alpha > 0\) and \(\beta \in (0, \frac{1}{2})\). Fix \(\eta \in (0, 1)\). Then there exists \(C_\eta', C_\eta > 0\) depending only on \(\eta\) such that for all \(C > C_\eta\), we have 
  \begin{align*}
      \sup_{\substack{f \in \mathcal{H}_\alpha, \\ V \in \mathcal{V}_0}} P_{f,V}\left\{ \varphi^* = 1\right\} 
      + \sup_{\substack{f \in \mathcal{H}_\alpha, \\ V \in \mathcal{V}_{1,\beta}(C\zeta)}} P_{f,V}\left\{\varphi^* = 0\right\} \leq \eta
  \end{align*}
  where \(\varphi^*\) is given by (\ref{test:optimal_test}) and \(\zeta = n^{-2\alpha} + n^{-\beta} + n^{-\frac{2\beta}{4\beta+1}}\).
\end{theorem}

\noindent For testing problem (\ref{problem:var0})-(\ref{problem:var1}), Theorem \ref{theorem: optimal test beta>1/4} yields the upper bound $\varepsilon^*(\alpha, \beta)^2 \lesssim  n^{-4(\alpha \wedge 1)} + n^{-2\beta} + n^{-\frac{4\beta}{4\beta+1}}$ for \(\alpha > 0\) and \(\beta \in (0, \frac{1}{2})\).

\subsection{Lower bound}\label{section:lower_bounds}
In this section, lower bounds for the minimax separation rate are established. In fact, we establish lower bounds for all \(\beta > 0\), not just the regime \(\beta \in (0, \frac{1}{2})\). As in \cite{wang_effect_2008}, it is assumed in the lower bound argument that the noise is Gaussian, that is \(\xi_i \overset{iid}{\sim} N(0, 1)\) in (\ref{model}). Thus, the lower bound will apply to both the cases where the noise is standard Gaussian and where the noise is unknown; the lower bounds in this section apply to the entire \(L^2\) heteroskedasticity row of Table \ref{table:rates}. As seen in (\ref{rate:testing}), there are three terms to establish corresponding to different phenomena. Each is discussed in turn. 

The \(n^{-4\alpha}\) term in (\ref{rate:testing}) is due to the presence of \(f\) in (\ref{model}), which is a nuisance for detecting heteroskedasticity. Interestingly, the effect of \(f\) is the same in the variance estimation problem as seen in (\ref{rate:estimation}). For the purpose of establishing lower bounds, it can be assumed without loss of generality \(0 < \alpha < \frac{1}{4}\). Otherwise, \(n^{-4\alpha}\) is dominated by \(n^{-1}\), which in turn is dominated by \(n^{-\frac{4\beta}{4\beta+1}}\) in (\ref{rate:testing}). Therefore, \(n^{-4\alpha}\) is relevant only in the low-smoothness regime \(0 < \alpha < \frac{1}{4}\).  
\begin{proposition}\label{prop:nuisance_lbound}
    Suppose \(0 < \alpha < \frac{1}{4}\) and \(\beta > 0\). Assume \(n\) is larger than some sufficiently large constant. If \(\eta \in (0, 1)\), then there exists \(c_\eta > 0\) not depending on \(n\) such that for all \(0 < c < c_\eta\) we have 
    \begin{equation*}
        \inf_{\varphi}\left\{ \sup_{\substack{f\in \mathcal{H}_\alpha, \\ V \in \mathcal{V}_{0}}} P_{f, V}\left\{ \varphi = 1 \right\} + \sup_{\substack{f \in \mathcal{H}_\alpha, \\ V \in \mathcal{V}_{1, \beta}(cn^{-2\alpha})}} P_{f, V}\left\{\varphi = 0\right\} \right\} \geq 1-\eta. 
    \end{equation*}
\end{proposition}
\noindent The argument is, in spirit, the same as that in \cite{wang_effect_2008}. Namely, the problem is reduced to a Bayes testing problem for which the marginal distribution of the data under the null and alternative share a large number of moments. The moment matching property of the two marginal distributions can be shown to imply indistinguishability of the two hypotheses \cite{wu_polynomial_2020}, thus giving a lower bound on the minimax separation rate. The argument from \cite{wang_effect_2008} cannot be directly used as their construction gives homoskedastic observations in both hypotheses, but can be modified appropriately.

Moving on, the \(n^{-\frac{4\beta}{4\beta+1}}\) term reflects the intrinsic difficulty of detecting heteroskedasticity, even in the case where there is no nuisance from the mean, that is, \(f \equiv 0\). The rate \(n^{-\frac{4\beta}{4\beta+1}}\) has appeared in homoskedastic signal detection (i.e testing whether the mean function is identically zero) and in goodness-of-fit testing of smooth densities \cite{ingster_nonparametric_2003,gine_mathematical_2016}. Consider in (\ref{rate:testing}), the term \(n^{-2\beta}\) dominates \(n^{-\frac{4\beta}{4\beta+1}}\) for \(\beta \leq \frac{1}{4}\). Thus, for the purposes of establishing the term \(n^{-\frac{4\beta}{4\beta+1}}\) in lower bound, it can be assumed without loss of generality \(\beta > \frac{1}{4}\).

\begin{proposition}\label{prop:nonparam_lbound}
    Suppose \(\beta > \frac{1}{4}\). If \(\eta \in (0, 1)\), then there exists \(c_\eta > 0\) not depending on \(n\) such that for all \(0 < c < c_\eta\) we have
    \begin{equation*}
        \inf_{\varphi}\left\{ \sup_{\substack{f \in \mathcal{H}_\alpha, \\ V \in \mathcal{V}_{0}}} P_{f, V}\left\{\varphi = 1\right\} + \sup_{\substack{f \in \mathcal{H}_\alpha, \\ V \in \mathcal{V}_{1,\beta}\left(c n^{-\frac{2\beta}{4\beta+1}}\right)}} P_{f, V}\left\{\varphi = 0\right\} \right\} \geq 1-\eta. 
    \end{equation*}
\end{proposition}

\noindent The prior construction in the lower bound is standard \cite{ingster_nonparametric_2003,arias-castro_remember_2018}, though the calculations for bounding the resulting \(\chi^2\)-divergence are somewhat different from existing calculations due to the heteroskedasticity, but these differences are largely technical. 

Next, the \(n^{-2\beta}\) term is attributable to the discretization of \([0, 1]\) by the fixed design points in (\ref{model}). To elaborate, only responses at the design points are available, and so \(n^{-2\beta}\) is due to approximating the smooth \(V\) by the points \(\{V(x_i)\}_{i=1}^{n}\). If \(V\) is constant at the design points but nonconstant between them, it should not be possible to detect the heteroskedasticity when \(\beta\) is small. This discretization effect is well-known in other problems \cite{cai_optimal_2011,devore_constructive_1993}. Without loss of generality, it can be assumed \(\beta \leq \frac{1}{4}\) as it is only in this regime \(n^{-2\beta}\) dominates \(n^{-\frac{4\beta}{4\beta+1}}\) in (\ref{rate:testing}).

\begin{proposition}\label{prop:fixed_design_lbound}
    Suppose \(0 < \beta \leq \frac{1}{4}\). If \(\eta \in (0, 1)\), then there exists \(c_\eta > 0\) not depending on \(n\) such that for all \(0 < c < c_\eta\) we have 
    \begin{equation*}
        \inf_{\varphi}\left\{ \sup_{\substack{f \in \mathcal{H}_\alpha, \\ V \in \mathcal{V}_0}} P_{f,V}\left\{\varphi = 1\right\} + \sup_{\substack{f \in \mathcal{H}_\alpha, \\ V \in \mathcal{V}_{1, \beta}(cn^{-\beta})}} P_{f,V}\left\{\varphi = 0\right\} \right\} \geq 1-\eta. 
    \end{equation*}
\end{proposition}
The lower bound argument proceeds by the two-point testing technique (also known as Le Cam's two-point method). The two-point construction makes essential use of the fixed design nature of (\ref{model}). Taking \(f \equiv 0\) in both null and alternative hypotheses, two variance functions \(V_0\) and \(V_1\) are constructed. Specifically, \(V_0 \equiv \mathbf{1}\) is homoskedastic and \(V_1\) is a \(\beta\)-H\"{o}lder smooth function which is equal to one on the design points and ``spiky" between (see Figure \ref{fig:fixed_design_lbound} in the supplementary material) such that \(||V_1 - \bar{V}_1\mathbf{1}||^2 \gtrsim n^{-2\beta}\). Since observations from (\ref{model}) occur only at the design points, the marginal distribution of the data is the same under both null and alternative hypotheses. Thus, the heteroskedasticity from \(V_1\) cannot be detected.

\section{Arbitrary variance profile}\label{section:profile}
As discussed in Section \ref{section:intro}, the existing literature on heteroskedasticity testing might be criticized for only being able to test for particular specifications of the heteroskedasticity. Though other work (e.g. \cite{white_heteroskedasticity-consistent_1980}) may impose no such assumptions on the variance, they often involve restrictive conditions on the regression function. The important case of an arbitrary variance profile with a nonparametric regression function has not yet been addressed in the literature.

The statistician who imposes no smoothness assumptions on the variance function necessarily is not concerned with the behavior of \(V\) between design points. The goal is to detect heteroskedasticity among the design points alone. To notationally distinguish this setting, the model will be written as 
\begin{equation}\label{model:profile}
    Y_i = f(x_i) + V_i^{1/2} \xi_i
\end{equation}
for \(i = 1,...,n\) where \(x_i = \frac{i}{n}\) are fixed design points in the unit interval and \(\xi_1,...,\xi_n \overset{iid}{\sim} N(0, 1)\) denote noise variables. Denote the variance profile by \(V = (V_1,...,V_n) \in [0, \infty)^{n}\). For the variance profile, define the parameter spaces 
\begin{align*}
    \Sigma_0 &= \left\{ V \in \left[0, M\right]^{n} : V_i = \sigma^2 \text{ for all } i = 1,...,n \text{ for some } \sigma^2 \in [0, M]\right\}, \\
    \Sigma_1(\varepsilon) &= \left\{ V \in \left[0, M\right]^{n} : \sqrt{\frac{1}{n} \sum_{i=1}^{n} \left(V_i - \bar{V}_n\right)^2} \geq \varepsilon \right\},
\end{align*}
where \(\bar{V}_n = \frac{1}{n}\sum_{i=1}^{n} V_i\). Throughout, \(M\) is assumed to be a sufficiently large constant. Note that the set of variance profiles in \(\Sigma_1\) do not obey a H\"{o}lder-type smoothness condition; in fact, nothing except boundedness is assumed. Formally, the heteroskedasticity detection problem is testing the hypotheses 
\begin{align}
    H_0&: V \in \Sigma_0 \text{ and } f \in \mathcal{H}_\alpha, \label{problem:var0_profile} \\
    H_1&: V \in \Sigma_1(\varepsilon) \text{ and } f \in \mathcal{H}_\alpha. \label{problem:var1_profile}
\end{align}

\noindent It will be shown
\begin{equation}\label{rate:betazero}
    \varepsilon^*(\alpha)^2 \asymp n^{-4\alpha} + n^{-1/2}.
\end{equation}
Notably, the minimax separation rate does not follow simply as a consequence of taking \(\beta = 0\) in (\ref{rate:testing}) for the problem (\ref{problem:var0})-(\ref{problem:var1}). The key difference in (\ref{rate:testing}) is that \(V\) is treated as a full, honest function with the entire unit interval as its domain. The \(L^2\) norm is used to define the separation between the null and alternative hypotheses in (\ref{problem:var0})-(\ref{problem:var1}). The behavior of \(V\) between the design points is important. In contrast, the problem (\ref{problem:var0_profile})-(\ref{problem:var1_profile}) only concerns \(V\) exclusively at the design points. This difference is conceptually important; rates cannot simply be ``exchanged" between the two settings.

\subsection{Necessity of noise information}\label{section:profile_triviality}
The noise variables in (\ref{model:profile}) are assumed to be distributed according to the standard Gaussian distribution, whereas only the weak assumption of a finite fourth moment is imposed in (\ref{model}). As will be shown in this section, the heteroskedasticity testing problem turns out to be trivial (i.e. consistent testing is impossible) if the noise distribution is totally unknown because there are no smoothness conditions on \(V\).

To show the triviality, suppose the distribution \(P_\xi\) of the noise \(\xi_i\) satisfies \(E(\xi_i^2) = 1\) and \(E(\xi_i^4) \leq C_\xi\) for some large universal constant \(C_\xi < \infty\) but is otherwise unknown; denote by \(\Xi\) the set of such noise distributions. From a minimax perspective, ignorance is incorporated in the testing risk by taking supremum over \(P_\xi \in \Xi\). Concretely, the problem is trivial if and only if the minimax separation rate is at least constant order, that is \(\varepsilon^* \gtrsim 1\). 

To show the lower bound, a Bayes testing problem is constructed. Consider the problem of testing
\begin{align*}
    H_0 &: f \equiv 0, V_i = \frac{M+1}{2}, \text{ and } \xi_i \overset{iid}{\sim} \sqrt{\frac{2}{1+M}} \cdot \left(\frac{1}{2} N(0, 1) + \frac{1}{2}N(0, M)\right), \\
    H_1 &: f \equiv 0, V_i \overset{iid}{\sim} \frac{1}{2}\delta_1 + \frac{1}{2} \delta_{M}, \text{ and } \xi_i \overset{iid}{\sim} N(0, 1).
\end{align*}
Note \(E(\xi_i^2) = 1\) and \(E(\xi_i^4) \lesssim 1\) under both hypotheses. Observe \(V\) is constant under \(H_0\) and \(\frac{1}{n} \sum_{i=1}^{n} \left(V_i - \bar{V}_n\right)^2 \gtrsim 1\) with high probability under \(H_1\). Examining the marginal distribution of the data, since \(V_i^{1/2}\xi_i \sim \frac{1}{2}N(0, 1) + \frac{1}{2}N(0, M)\) under either \(H_0\) or \(H_1\), we have 
\begin{align*}
    H_0 &: Y_1,...,Y_n \overset{iid}{\sim} \frac{1}{2} N(0, 1) + \frac{1}{2}N(0, M), \\
    H_1 &: Y_1,...,Y_n \overset{iid}{\sim} \frac{1}{2} N(0, 1) + \frac{1}{2} N(0, M),
\end{align*}
and so the hypotheses are indistinguishable. Therefore, it follows \(\varepsilon^* \gtrsim 1\) as claimed. The following result summarizes this derivation and a formal proof is given in the supplementary material.
\begin{proposition}\label{prop:fourth_moment_info}
    Suppose \(\alpha > 0\). If \(\eta \in (0, 1)\), then there exists \(c_\eta > 0\) not depending on \(n\) such that for all \(0 < c < c_\eta\) and \(n\) sufficiently large depending only on \(\eta\), we have
    \begin{equation*}
        \inf_{\varphi}\left\{\sup_{\substack{f \in \mathcal{H}_\alpha, \\ V \in \Sigma_0, \\ P_\xi \in \Xi}} P_{f, V, P_\xi}\left\{\varphi = 1\right\} + \sup_{\substack{f \in \mathcal{H}_\alpha, \\ V \in \Sigma_1(c), \\ P_\xi \in \Xi}} P_{f, V, P_\xi}\left\{\varphi = 0\right\} \right\} \geq 1-\eta. 
    \end{equation*}
\end{proposition}
Note the prior on \(V\) in \(H_1\) does not yield a variance profile satisfying a H\"{o}lder-type smoothness condition as in (\ref{space:V1}). Intuitively, it is the lack of a smoothness constraint which results in the triviality. Without smoothness, the statistician cannot distinguish whether variability in the responses is due to heteroskedasticity or simply a fatter noise distribution. If smoothness constraints on \(V\) were imposed, there is hope to detect heteroskedasticity since an independent sequence of squared variates from \(P_\xi\) will not satisfy the H\"{o}lder condition whereas \(V\) does.

\subsection{Methodology}\label{section:profile_methodology}
Rather than estimating the quadratic functional \(\frac{1}{n} \sum_{i=1}^{n} \left(V_i - \bar{V}_n\right)^2\), a proxy will be considered instead. However, the proxy is more complicated than (\ref{def: T}) since it is no longer assumed \(V \in \mathcal{H}_\beta\); there is no smoothness of which to take advantage. To elaborate, consider the statistic (\ref{def: testing_statistic beta>1/4}) used to estimate the proxy (\ref{def: T}) is purely a function of the squared first-order differences \(R_i^2 = (Y_{i+1} - Y_i)^2\). In fact, all the tests described in Section \ref{section:related_work} use only the first-order differences as well. However, when \(V\) is not assumed to be smooth, it does not suffice to only use the \(\{R_i^2\}_{i=1}^{n-1}\). To illustrate, consider even \(n\) and the case where \(f \equiv 0\) and the variance profile is \(V = (1, 3, 1, 3, 1, 3, ..., 1, 3) \in \R^n\). Note since \(V\) is not smooth, the variances of neighboring points can to be quite different from one another as they are here. Clearly there is much heteroskedasticity as \(\frac{1}{n}\sum_{i=1}^{n} (V_i - \bar{V}_n)^2 \gtrsim 1\), and observe the squared first-order differences all have the same marginal distributions \(R_i^2 \sim 4 \cdot \chi^2_1\). However, observe they will still have the same distributions if the variance profile were actually homoskedastic \(V = (2, 2, 2, 2, ..., 2) \in \R^n\). Therefore, one cannot exclusively use \(\{R_i^2\}_{i=1}^{n-1}\) in order to test heteroskedasticity when there are no smoothness assumptions on \(V\). Thus, the methodology of \cite{dette_consistent_2002,dette_testing_1998,wang_effect_2008} and the statistic (\ref{def: testing_statistic beta>1/4}) fail to address the case of an arbitrary variance profile. Consequently, we consider a more complicated construction.

For \(1 \leq i \leq n-1\), define \(W_i = V_{i+1} + V_{i}\) and \(\delta_i = f(x_{i+1}) - f(x_i)\). Further, let \(\bar{W}_n\) and \(\overline{\delta^2_n}\) denote the respective averages. For \(1 \leq i \leq n-2\), define \(\tilde{W}_i = V_{i+2} + V_{i}\) and \(\tilde{\delta}_i = f(x_{i+2}) - f(x_i)\).  We will consider the following population level quantities
\begin{align}
    T &= \frac{1}{n}\sum_{i=1}^{n-1} \left(W_i + \delta_i^2 - \bar{W}_n - \overline{\delta^2_n}\right)^2, \label{def:T_betazero} \\
    \tilde{T}_1 &= \frac{1}{n} \sum_{i=1}^{n-3} \left(W_{i+1} + \delta_{i+1}^2 - V_i - V_{i+3} - (f(x_{i+3}) - f(x_i))^2 \right)^2 \label{def:T1_tilde_betazero},\\
    \tilde{T}_2 &= \frac{1}{n} \sum_{i=1}^{n-3} \left(\tilde{W}_{i+1} + \tilde{\delta}_{i+1}^2 - \tilde{W}_i - \tilde{\delta}_{i}^2\right)^2 \label{def:T2_tilde_betazero}.
\end{align}
\noindent The following result establishes that \(T, \tilde{T}_1,\) and \(\tilde{T}_2\) contain some signal to detect heteroskedasticity.
\begin{proposition}\label{prop:signal}
    If \(V \in \left[0, M\right]^{n}\) and \(f \in \mathcal{H}_\alpha\), then 
    \begin{equation*}
        \frac{1}{n} \sum_{i=1}^{n} \left(V_i - \bar{V}_n\right)^2 \lesssim n^{-1} +  n^{-4(\alpha\wedge 1)} + T + \tilde{T}_1 + \tilde{T}_2
    \end{equation*}
    where \(T, \tilde{T}_1,\) and \(\tilde{T}_2\) are given by (\ref{def:T_betazero}), (\ref{def:T1_tilde_betazero}), and (\ref{def:T2_tilde_betazero}) respectively. 
\end{proposition}
\begin{proof}
    The useful inequality \((a + b)^2 \lesssim a^2 + b^2\) will be frequently employed. By direct calculation and since \(V_i \lesssim 1\), we have
    \begin{align*}
        \frac{1}{n} \sum_{i=1}^{n} \left(V_i - \bar{V}_n\right)^2 &\lesssim \frac{1}{n} + \frac{1}{n}\sum_{i=1}^{n-1} \left(2V_i - 2\bar{V}_n\right)^2 \\
        &\asymp n^{-1} +  \frac{1}{n}\sum_{i=1}^{n-1} \left(2V_i + W_i - W_i + \delta_i^2 - \delta_i^2 + \overline{\delta^2_n} - \bar{W}_n + \bar{W}_n - 2\bar{V}_n - \overline{\delta^2_n}\right)^2 \\
        &\lesssim n^{-1} + n^{-4(\alpha \wedge 1)} + \left(\bar{W}_n - 2\bar{V}_n\right)^2 + \frac{1}{n} \sum_{i=1}^{n-1} \left(2V_i - W_i\right)^2 + \frac{1}{n}\sum_{i=1}^{n-1} \left(W_i + \delta_i^2 - \bar{W}_n - \overline{\delta^2_n}\right)^2 \\
        &\lesssim n^{-1} + n^{-4(\alpha \wedge 1)} + n^{-2} + \frac{1}{n} \sum_{i=1}^{n-1} \left(V_i - V_{i+1}\right)^2 + T \\
        &\lesssim n^{-1} + n^{-4(\alpha \wedge 1)} + \frac{1}{n} \sum_{i=1}^{n-3} \left(V_i - V_{i+1}\right)^2 + T
    \end{align*}
    where we have used \(\delta_i^2, \overline{\delta^2_n} \lesssim n^{-2(\alpha \wedge 1)}\). Let us now examine \(\frac{1}{n} \sum_{i=1}^{n-3} \left(V_i - V_{i+1}\right)^2\). Consider that for \(1 \leq i \leq n-3\),
    \begin{align*}
        V_i &= \frac{W_i + \tilde{W}_i - W_{i+1}}{2}, \\
        V_{i+1} &= \frac{W_{i+1} + \tilde{W}_{i+1} - W_{i+2}}{2}.
    \end{align*}
    Therefore, 
    \begin{align*}
        \frac{1}{n} \sum_{i=1}^{n-3} \left(V_i - V_{i+1}\right)^2 &\lesssim \frac{1}{n} \sum_{i=1}^{n-3} \left(\left(W_{i+1} - W_i\right) - (W_{i+2} - W_{i+1})\right)^2 + \frac{1}{n} \sum_{i=1}^{n-3} \left(\tilde{W}_{i+1} - \tilde{W}_i\right)^2 \\
        &\lesssim n^{-4(\alpha \wedge 1)} + \frac{1}{n} \sum_{i=1}^{n-3} \left(V_{i+1} + V_{i+2} - V_{i} - V_{i+1} - V_{i+2} - V_{i+3} + V_{i+1} + V_{i+2}\right)^2 \\
        &\;\; + \frac{1}{n} \sum_{i=1}^{n-3} \left(\tilde{W}_{i+1} + \tilde{\delta}_{i+1}^2 - \tilde{W}_i - \tilde{\delta}_{i}^2\right)^2 \\
        &\asymp n^{-4(\alpha \wedge 1)}  + \tilde{T}_2 + \frac{1}{n} \sum_{i=1}^{n-3} \left( V_{i+1} + V_{i+2} - V_i - V_{i+3} \right)^2 \\
        &\lesssim n^{-4(\alpha \wedge 1)} + \tilde{T}_2 \\
        &\;\; + \frac{1}{n} \sum_{i=1}^{n-3} \left(V_{i+1} + V_{i+2} + \delta_{i+1}^2 - V_i - V_{i+3} - (f(x_{i+3}) - f(x_i))^2 \right)^2 \\
        &\asymp n^{-4(\alpha \wedge 1)} + \tilde{T}_1 + \tilde{T}_2
    \end{align*}
    where we have used \(\tilde{\delta}_i^2 \lesssim n^{-2(\alpha \wedge 1)}\) and \(\left(f(x_{i+3}) - f(x_{i})\right)^2 \lesssim n^{-2(\alpha \wedge 1)}\). Hence, we have shown 
    \begin{equation*}
        \frac{1}{n} \sum_{i=1}^{n} \left(V_i - \bar{V}_n\right)^2 \lesssim n^{-1} + n^{-4(\alpha \wedge 1)} + T + \tilde{T}_1 + \tilde{T}_2
    \end{equation*}
    as desired. 
\end{proof}

Different estimators will be constructed to estimate \(T, \tilde{T}_1,\) and \(\tilde{T}_2\) respectively. A similar construction is employed with slight modifications for each target.

\subsection{Upper bound}\label{section:profile_upper_bound}

Recall the definition of the first-order differences \(R_i = Y_{i+1} - Y_i\) for \(1 \leq i \leq n-1\). For estimation of \(T\), the following statistic \(\hat{T}\) is used,
\begin{equation}\label{def:testing_statistic beta<1/4}
        \hat{T} = \frac{1}{2n^2} \sum_{|i-j| \geq 2} \frac{1}{3}\left(R_i^4 + R_j^4\right) - 2R_i^2R_j^2.
\end{equation}
Similar to the methodology in Section \ref{section: upper bound}, the sum is taken over \(|i-j| \geq 2\) to ensure \(R_i\) and \(R_j\) are independent which greatly simplifies calculations. 

\begin{proposition}\label{prop: MSE of T_hat beta<1/4}
Suppose \(\alpha > 0\). If \(f \in \mathcal{H}_\alpha\) and \(V \in \Sigma_1(0)\), then
  \begin{equation*}
     E_{f,V}\left(\left|\hat{T} - T\right|^2\right) \lesssim n^{-1} + n^{-8\alpha}
  \end{equation*}
  where \(\hat{T}\) and $T$ are given by (\ref{def:testing_statistic beta<1/4}) and (\ref{def:T_betazero}) respectively.
\end{proposition}

\noindent To estimate \(\tilde{T}_1\), define \(S_i = Y_{i+3} - Y_i\) for \(1 \leq i \leq n-3\). Consider the statistic 
\begin{equation}\label{def:T1hat}
    \hat{T}_1 = \frac{1}{n} \sum_{i=1}^{n-3} \frac{1}{3} \left(R_{i+1}^4 + S_{i}^4\right) - 2R_{i+1}^2 S_i^2.
\end{equation}
The following estimation guarantee is available.
\begin{proposition}\label{prop:mse_T1_tilde}
    Suppose \(\alpha > 0\). If \(f \in \mathcal{H}_\alpha\) and \(V \in \Sigma_1(0)\), then 
    \begin{equation*}
        E_{f, V}\left(\left|\hat{T}_1 - \tilde{T}_1\right|^2\right) \lesssim n^{-1} + n^{-8\alpha}
    \end{equation*}
    where \(\hat{T}_1\) and \(\tilde{T}_1\) are given by (\ref{def:T1hat}) and (\ref{def:T1_tilde_betazero}) respectively. 
\end{proposition}

\noindent A slightly different first-order difference is useful for estimating \(\tilde{T}_2\). For \(1 \leq i \leq n-2\), define \(\tilde{R}_i = Y_{i+2} - Y_i\). Consider 
\begin{equation}\label{def:T2hat}
    \hat{T}_2 = \frac{1}{n} \sum_{i=1}^{n-3} \frac{1}{3} \left(\tilde{R}_{i+1}^4 + \tilde{R}_{i}^4\right) - 2 \tilde{R}_{i+1}^2\tilde{R}_{i}^2. 
\end{equation}
A similar guarantee to Proposition \ref{prop:mse_T1_tilde} can be established.
\begin{proposition}\label{prop:mse_T2_tilde}
    Suppose \(\alpha > 0\). If \(f \in \mathcal{H}_\alpha\) and \(V \in \Sigma_1(0)\), then 
    \begin{equation*}
        E_{f, V}\left(\left|\hat{T}_2 - \tilde{T}_2\right|^2\right) \lesssim n^{-1} + n^{-8\alpha}
    \end{equation*}
    where \(\hat{T}_2\) and \(\tilde{T}_2\) are given by (\ref{def:T2hat}) and (\ref{def:T2_tilde_betazero}) respectively. 
\end{proposition}

The Gaussian nature of the noise is critical in all three estimators (\ref{def:testing_statistic beta<1/4}), (\ref{def:T1hat}), and (\ref{def:T2hat}). In fact, the normalizing factor \(\frac{1}{3}\) is a consequence of the fact that \(E(Z^4) = 3\) for \(Z \sim N(0, 1)\). Proper normalization is important to obtaining the correct bias behavior. Having in hand \(\hat{T}, \hat{T}_1,\) and \(\hat{T}_2\), which are given respectively by (\ref{def:T_betazero}), (\ref{def:T1_tilde_betazero}), and (\ref{def:T2_tilde_betazero}), the following statistic will be used for testing (\ref{problem:var0_profile})-(\ref{problem:var1_profile}) 
\begin{equation}\label{def:Shat}
    \hat{S} = \hat{T} + \hat{T}_1 + \hat{T}_2.
\end{equation}
From Propositions \ref{prop: MSE of T_hat beta<1/4}, \ref{prop:mse_T1_tilde}, and \ref{prop:mse_T2_tilde}, it follows 
\begin{equation*}
    \sup_{\substack{f \in \mathcal{H}_\alpha,\\ V \in \Sigma_1(0)}} E_{f, V}\left(\left| \hat{S} - (T + \tilde{T}_1 + \tilde{T}_2)\right|^2\right) \lesssim n^{-1} + n^{-8\alpha}. 
\end{equation*}
To illustrate how \(\hat{S}\) is useful for heteroskedasticity detection, consider under the null hypothesis it is clear \(T, \tilde{T}_1, \tilde{T}_2 \lesssim n^{-4(\alpha \wedge 1)}\) and so \(\hat{S} \lesssim n^{-1/2} + n^{-4\alpha}\) with high probability. Under the alternative hypothesis, it follows from Proposition \ref{prop:signal} that \(\frac{1}{n} \sum_{i=1}^{n} \left(V_i - \bar{V}_n\right)^2 \lesssim n^{-4\alpha} + n^{-1/2} + \hat{S}\) with high probability. Therefore, if \(\frac{1}{n} \sum_{i=1}^{n} \left(V_i - \bar{V}_n\right)^2 \geq \varepsilon^2 \gtrsim n^{-4 \alpha } + n^{-1/2}\), then it must be the case that \(\frac{1}{n} \sum_{i=1}^{n} \left(V_i - \bar{V}_n\right)^2 \lesssim \hat{S}\). Hence, \(\hat{S}\) can be used to detect the alternative. 

\begin{theorem}\label{thm:beta_zero_upper}
    Suppose \(\alpha > 0\). If \(\eta \in (0, 1)\), there exist \(C_\eta', C_\eta > 0\) depending only on \(\eta\) such that for all \(C > C_\eta\) we have 
    \begin{equation*}
        \sup_{\substack{f \in \mathcal{H}_\alpha, \\ V \in \Sigma_0}} P_{f, V}\left\{ \hat{S} > C_\eta' \left(n^{-4\alpha} + n^{-1/2}\right)\right\} + \sup_{\substack{f \in \mathcal{H}_\alpha, \\ V \in \Sigma_1(C(n^{-2\alpha} + n^{-1/4}))}} P_{f, V}\left\{ \hat{S} \leq C_\eta'\left( n^{-4\alpha} + n^{-1/2}\right)\right\} \leq \eta
    \end{equation*}
    where \(\hat{S}\) is given by (\ref{def:Shat}). 
\end{theorem}

The testing statistic \(\hat{S}\) is composed of three statistics \(\hat{T}, \hat{T}_1,\) and \(\hat{T}_2\) which are critically different in their details, but are all of the same flavor. Interestingly, though these are reminiscent of the one proposed in \cite{dette_testing_1998}, the statistic \(\hat{S}\) is able to test (\ref{problem:var0_profile})-(\ref{problem:var1_profile}) with no smoothness assumptions on \(V\) and does not require \(\alpha \geq \frac{1}{2}\). Specifically, \(\hat{S}\) is able to achieve (when \(\alpha\) is large enough) the same \(n^{-1/2}\) rate obtained in \cite{dette_testing_1998}. 

\subsection{Lower bound}\label{section:profile_lower_bound}
From a lower bound perspective, the \(n^{-4\alpha}\) term in (\ref{rate:betazero}) can be obtained essentially from the argument used to prove Proposition \ref{prop:nonparam_lbound}, which employed a moment-matching construction inspired by \cite{wang_effect_2008}. Therefore, only a proof of the \(n^{-1/2}\) term will be given.

\begin{theorem}\label{thm:beta_zero_lower}
    Suppose \(\alpha > 0\). If \(\eta \in (0, 1)\), then there exists \(c_\eta > 0\) not depending on \(n\) such that for all \(0 < c < c_\eta\) and \(n\) sufficiently large depending only on \(\eta\), we have 
    \begin{equation*}
        \inf_{\varphi}\left\{ \sup_{\substack{f \in \mathcal{H}_\alpha, \\ V \in \Sigma_0}} P_{f, V}\left\{\varphi = 1\right\} + \sup_{\substack{f \in \mathcal{H}_\alpha, \\ V \in \Sigma_1(cn^{-1/4})}} P_{f, V}\left\{\varphi = 0\right\} \right\} \geq 1-\eta. 
    \end{equation*}
\end{theorem}

The proof can be found in the supplementary material; a brief sketch is given here. The minimax testing risk is bounded from below by obtaining a simple-versus-simple testing problem by constructing a prior (nearly) supported on the alternative. Though the construction in the proof is slightly different, for all intents and purposes the following is essentially employed, 
\begin{align*}
    H_0 &: f \equiv 0 \text{ and } V \equiv 1, \\
    H_1 &: f \equiv 0 \text{ and } V_1,...,V_n \overset{iid}{\sim} \frac{1}{2} \delta_{1-cn^{-1/4}} + \frac{1}{2}\delta_{1+cn^{-1/4}}.
\end{align*}
In the alternative, \(\frac{1}{n}\sum_{i=1}^{n} (V_i - \bar{V}_n)^2 \gtrsim n^{-1/2}\) with high probability. For small enough \(c\), the \(\chi^2\)-divergence can be bounded by a small constant, establishing that the hypotheses cannot be distinguished despite the heteroskedasticity in the alternative. The indistinguishability can be intuited by noticing that under \(H_0\) the data are \(Y_1,...,Y_n \overset{iid}{\sim} N(0, 1)\), whereas under \(H_1\) the data are \(Y_1,...,Y_n \overset{iid}{\sim} \frac{1}{2} N(0, 1-cn^{-1/4}) + \frac{1}{2}N(0, 1 + cn^{-1/4})\). Specifically, the null and alternative distributions share the first three moments, indicating that the distributions are difficult to distinguish. The moment matching approach for establishing lower bounds was used earlier in Section \ref{section:lower_bounds}. 

\section{Adaptation}\label{section:adaptation}
In this section, we consider adaptation to \(\alpha\) in the setup of Section \ref{section:profile}. As seen in Theorem \ref{thm:beta_zero_upper}, the testing procedure we propose requires knowledge of \(\alpha\) to define the cutoff value for the rejection criterion. The H\"{o}lder exponent is not typically known in practice, and so adaptation is of interest. Definition \ref{def:adaptive_test} lays out what it means for a test to achieve a candidate rate simultaneously over all possible \(\alpha\). We say adaptation is possible if there exists exists a test \(\varphi\) which achieves the minimax rate (\ref{rate:betazero}) simultaneously over \(\alpha\) in the sense of Definition \ref{def:adaptive_test}.

\begin{definition}\label{def:adaptive_test}
    For a candidate rate \(\alpha \mapsto \varepsilon(\alpha)\), we say the test \(\varphi\) achieves the rate \(\varepsilon(\alpha)\) simultaneously over \(\alpha\) if for all \(\eta \in (0, 1)\) there exists \(C_\eta > 0\) depending only on \(\eta\) such that for all \(C \geq C_\eta\), 
\begin{equation*}
    \sup_{\alpha > 0} \sup_{\substack{f \in \mathcal{H}_{\alpha}, \\ V \in \Sigma_0}} P_{f,V}\left\{\varphi = 1\right\} + \sup_{\alpha > 0} \sup_{\substack{f \in \mathcal{H}_{\alpha}, \\ V \in \Sigma_1(C\varepsilon(\alpha))}} P_{f,V}\left\{\varphi = 0\right\} \leq \eta. 
\end{equation*}
\end{definition}

It turns out adaptation is impossible. In fact, it is not even possible to adapt to just two smoothness classes. To elaborate, fix \(\alpha_1 < \alpha_2 < \frac{1}{4}\) and consider that for any test \(\varphi\) and any candidate rate \(\alpha \mapsto \varepsilon(\alpha)\), we have 
\begin{align*}
    &\sup_{\alpha > 0} \sup_{\substack{f \in \mathcal{H}_{\alpha}, \\ V \in \Sigma_0}} P_{f,V}\left\{\varphi = 1\right\} + \sup_{\alpha > 0} \sup_{\substack{f \in \mathcal{H}_{\alpha}, \\ V \in \Sigma_1(C\varepsilon(\alpha))}} P_{f,V}\left\{\varphi = 0\right\} \\
    &\geq \sup_{\substack{f \in \mathcal{H}_{\alpha_1}, \\ V \in \Sigma_0}} P_{f,V}\left\{\varphi = 1\right\} + \sup_{\substack{f \in \mathcal{H}_{\alpha_2}, \\ V \in \Sigma_1(C\varepsilon(\alpha_2))}} P_{f,V}\left\{\varphi = 0\right\}.
\end{align*}
The testing risk on the right hand side corresponds to the risk of \(\varphi\) in the following problem, 
\begin{align}
    H_0 &: f \in \mathcal{H}_{\alpha_1} \text{ and } V \in \Sigma_0, \label{problem:alpha_adapt0}\\
    H_1 &: f \in \mathcal{H}_{\alpha_2} \text{ and } V \in \Sigma_{1}(\varepsilon(\alpha_2)). \label{problem:alpha_adapt1}
\end{align}
It turns out one can only distinguish the hypotheses if the candidate rate \(\alpha \mapsto \varepsilon(\alpha)\) satisfies \(\varepsilon(\alpha_2) \gtrsim n^{-2\alpha_1}\).  In other words, the statistician cannot exploit the additional smoothness in the alternative and instead is forced to aim at the slower rate associated to the rougher null hypothesis. The following theorem formally states this impossibility result.

\begin{theorem}\label{thm:alpha_adapt}
    Suppose \(0 < \alpha_1 < \alpha_2 < \frac{1}{4}\). Assume \(n\) is larger than some sufficiently large universal constant. If \(\eta \in (0, 1)\), then there exists \(c_\eta > 0\) not depending on \(n\) such that for all \(0 < c < c_\eta\) we have 
    \begin{equation*}
        \inf_{\varphi}\left\{ \sup_{\substack{f \in \mathcal{H}_{\alpha_1}, \\ V \in \Sigma_0}} P_{f, V}\left\{\varphi = 1\right\} + \sup_{\substack{f \in \mathcal{H}_{\alpha_2}, \\ V \in \Sigma_1\left(cn^{-2\alpha_1}\right)}} P_{f,V}\left\{\varphi = 0\right\} \right\} \geq 1-\eta.
    \end{equation*}
\end{theorem}

This result concerning adaptation is especially notable because Theorem \ref{thm:alpha_adapt} asserts the statistician pays the ultimate cost in the rate. 
No test can achieve a faster rate than that given by the rough exponent \(\alpha_1\) for the testing problem (\ref{problem:alpha_adapt0})-(\ref{problem:alpha_adapt1}). Theorem \ref{thm:alpha_adapt} is quite surprising given existing articles in the literature concerning adaptive testing and quadratic functional estimation which identify sharp costs which are only logarithmic \cite{efromovich_optimal_1996,cai_optimal_2006,spokoiny_adaptive_1996}. For example, Spokoiny \cite{spokoiny_adaptive_1996} establishes, among other results, that goodness-of-fit testing of \(\alpha\)-smooth Sobolev functions without knowledge of \(\alpha\) exhibits the sharp (squared) rate \(\left(\frac{n}{\sqrt{\log\log n}}\right)^{-\frac{4\alpha}{4\alpha+1}}\). The cost paid here for the ignorance of \(\alpha\) is very mild, whereas Theorem \ref{thm:alpha_adapt} shows an extremely large cost. Conceptually, the key difference in (\ref{problem:alpha_adapt0})-(\ref{problem:alpha_adapt1}) compared to, say, \cite{spokoiny_adaptive_1996} is that the null hypothesis depends on the H\"{o}lder exponent. Specifically, it may not agree with the alternative hypothesis. The effect of not knowing \(\alpha\) can be seen by examining the test in Theorem \ref{thm:beta_zero_upper}. When \(\alpha\) is not known, it is not clear how to choose the rejection threshold. Intuitively, the statistician is trapped into using the worst case \(\alpha = \alpha_1\). To the best of our knowledge, the brutal cost established in Theorem \ref{thm:alpha_adapt} is the first of its kind concerning adaptive testing. To illustrate the impossibility of adaptation, a numerical simulation is presented in the supplementary material.

\section{Heteroskedasticity across the design}\label{section:discrete_loss}
In some situations, smoothness of \(V\) may be reasonably assumed but the behavior of the variance function between design points may still be of no concern as in Section \ref{section:profile}. Specifically, the model (\ref{model}) is of interest, but the problem formulation of interest is different from (\ref{problem:var0})-(\ref{problem:var1}). Define the parameter spaces 
\begin{align*}
    \mathcal{D}_0 &= \left\{V : [0, 1] \to [0, \infty) : V(x_i) = \sigma^2 \text{ for all } i \text{ for some } 0 \leq \sigma^2 \leq M\right\}, \\
    \mathcal{D}_{1, \beta}(\varepsilon) &= \left\{ V \in \mathcal{H}_\beta : V \geq 0 \text{ and } \sqrt{\frac{1}{n+1} \sum_{i=0}^{n} \left(V(x_i) - \bar{V}_n\right)^2} \geq \varepsilon \right\}
\end{align*}
where \(\bar{V}_n = \frac{1}{n+1} \sum_{i=0}^{n} V(x_i)\) is the average variance across the design points. Throughout, \(M\) is a sufficiently large constant. Note the heteroskedasticity is measured only with respect to the design points; the behavior of \(V\) between the design points is immaterial. The testing problem of interest is formally given by 
\begin{align}
    H_0 &: V \in \mathcal{D}_0 \text{ and } f \in \mathcal{H}_\alpha, \label{problem:var_discrete_0}\\
    H_1 &: V \in \mathcal{D}_{1,\beta}(\varepsilon) \text{ and } f \in \mathcal{H}_\alpha. \label{problem:var_discrete_1}
\end{align}
In model (\ref{model}), the noise is only assumed to satisfy \(E(\xi_i) = 0, E(\xi_i^2) = 1\), and \(E(\xi_i^4) < \infty\). As hinted at in Section \ref{section:profile}, knowledge/ignorance of the noise distribution turns out to affect the minimax rates. The two settings are handled separately. 

\subsection{Known noise distribution}\label{section:discrete_known_noise}
In this section, we assume the noise distribution is Gaussian. The minimax separation rate for testing (\ref{problem:var_discrete_0})-(\ref{problem:var_discrete_1}) is 
\begin{equation}\label{rate:discrete_known_noise}
    \varepsilon^*(\alpha, \beta)^2 \asymp n^{-4\alpha} + n^{-\left(\frac{1}{2} \vee \frac{4\beta}{4\beta+1}\right)}
\end{equation}
for \(\alpha > 0\) and \(\beta \in (0, \frac{1}{2})\). Noting the phase transition at \(\beta = \frac{1}{4}\), the minimax rate simplifies to \(n^{-4\alpha} + n^{-1/2}\) for \(\beta < \frac{1}{4}\) and \(n^{-4\alpha} + n^{-\frac{4\beta}{4\beta+1}}\) for \(\beta \geq \frac{1}{4}\).

\subsubsection{Upper bound}\label{section: upper bound noise known}
The rate (\ref{rate:discrete_known_noise}) for the testing problem (\ref{problem:var_discrete_0})-(\ref{problem:var_discrete_1}) matches the rate (\ref{rate:testing}) for testing (\ref{problem:var0})-(\ref{problem:var1}) when $\frac{1}{4} \leq \beta < \frac{1}{2}$. The optimal test of Theorem \ref{theorem: optimal test beta>1/4} achieving (\ref{rate:testing}) does not rely on any information about the noise distribution. For \(\beta \geq \frac{1}{4}\), it turns out the same statistic (\ref{def: testing_statistic beta>1/4}) can be used to test (\ref{problem:var_discrete_0})-(\ref{problem:var_discrete_1}) even though the separation is defined across the design points rather than through the \(L^2\) norm as in (\ref{problem:var0})-(\ref{problem:var1}). However, for $\beta<\frac{1}{4}$ the rate (\ref{rate:discrete_known_noise}) is faster than (\ref{rate:testing}) and so a different approach is needed. 

The testing problem (\ref{problem:var0_profile})-(\ref{problem:var1_profile}), which deals with an arbitrary variance profile, shares the same separation metric as (\ref{problem:var_discrete_0})-(\ref{problem:var_discrete_1}) and the rate (\ref{rate:betazero}) matches (\ref{rate:discrete_known_noise}). Moreover, the test statistic $\hat{S}$, defined in (\ref{def:Shat}), uses the fact that the noise is Gaussian. It is shown the statistic $\hat{S}$ also furnishes a test achieving (\ref{rate:discrete_known_noise}) for the problem (\ref{problem:var_discrete_0})-(\ref{problem:var_discrete_1}) when $\beta<\frac{1}{4}$.

\begin{theorem}\label{theorem: optimal test discrete known noise}
  Suppose \(\alpha > 0\) and \(0 < \beta < \frac{1}{2}\). Fix \(\eta \in (0, 1)\). Then, there exists \(C_\eta', C_\eta > 0\) depending only on \(\eta\) such that for all \(C > C_\eta\), we have 
  
  \begin{align*}
      \sup_{\substack{f \in \mathcal{H}_\alpha, \\ V \in \mathcal{D}_0}} P_{f,V}\left\{ \hat{T} > C_\eta' \zeta^2 \right\} 
      + \sup_{\substack{f \in \mathcal{H}_\alpha, \\ V \in \mathcal{D}_{1,\beta}(C\zeta)}} P_{f,V}\left\{\hat{T} \leq C_\eta' \zeta^2\right\} \leq \eta
  \end{align*}
  where $\zeta = n^{-2\alpha} + n^{-\left(\frac{1}{4} \vee \frac{2\beta}{4\beta+1}\right)}$. Here, for \(\beta < \frac{1}{4}\), we set \(\hat{T} = \hat{S}\) where \(\hat{S}\) is defined by (\ref{def:Shat}), and for \(\beta \geq \frac{1}{4}\) we set \(\hat{T}\) by (\ref{def: testing_statistic beta>1/4}). 
\end{theorem}
\noindent In conclusion, Theorem \ref{theorem: optimal test discrete known noise} provides a test achieving (\ref{rate:discrete_known_noise}) for the testing problem  (\ref{problem:var_discrete_0})-(\ref{problem:var_discrete_1}).

\subsubsection{Lower bound}
As in Section \ref{section:profile_lower_bound}, the \(n^{-4\alpha}\) term in (\ref{rate:discrete_known_noise}) can be obtained by essentially the same proof as the proof of Proposition \ref{prop:nuisance_lbound}. Likewise, for \(\beta > \frac{1}{4}\), the term \(n^{-\frac{4\beta}{4\beta+1}}\) can be established by using essentially the same argument as that of Proposition \ref{prop:nonparam_lbound} in Section \ref{section:lower_bounds}. As the necessary modifications are slight, technical, and easily carried out, we will omit them. We focus only on establishing the term \(n^{-1/2}\) in the regime \(\beta \leq \frac{1}{4}\). 

It turns out the construction used in Section \ref{section:profile_lower_bound} can be used, even though the setting of an arbitrary variance profile corresponds to \(\beta = 0\). Specifically, for \(\beta \leq \frac{1}{4}\), a prior over \(\beta\)-H\"{o}lder variance functions can be constructed such that \(V(x_i)\) is distributed according to the prior in the alternative hypothesis in Section \ref{section:profile_lower_bound}. Since only data at the design points are observed, the lower bound argument proceeds exactly in the same manner. The following proposition gives a formal statement and proof details are given in Section \ref{section:discrete_known_noise}.
\begin{proposition}\label{prop:discrete_known_noise_sqrtn}
    Suppose \(\alpha > 0\) and \(0 < \beta \leq \frac{1}{4}\). If \(\eta \in (0, 1)\), then there exists \(c_\eta > 0\) not depending on \(n\) such that for all \(0 < c < c_\eta\) and \(n\) sufficiently large depending only on \(\eta\), we have 
    \begin{equation*}
        \inf_{\varphi}\left\{ \sup_{\substack{f \in \mathcal{H}_\alpha, \\ V \in \mathcal{D}_0}} P_{f, V}\left\{\varphi = 1\right\} + \sup_{\substack{f \in \mathcal{H}_\alpha, \\ V \in \mathcal{D}_{1,\beta}(cn^{-1/4})}} P_{f,V}\left\{\varphi = 0\right\}  \right\} \geq 1-\eta. 
    \end{equation*}
\end{proposition}
\noindent To give a sketch of the argument, consider the Bayes testing problem 
\begin{align*}
    H_0 &: f \equiv 0 \text{ and } V = \mathbf{1}, \\
    H_1 &: f \equiv 0 \text{ and } V \sim \pi
\end{align*}
where the prior \(\pi\) is defined as follows. Draw \(R_0,...,R_n \overset{iid}{\sim} \Rademacher\left(\frac{1}{2}\right)\) and set, for \(0 \leq x \leq 1\),
\begin{equation*}
    V(x) = 1 + \sum_{i=0}^{n} R_i cn^{-1/4} g\left(x - \frac{i}{n}\right)
\end{equation*}
where \(g(x) = (1 - 2n|x|)\mathbbm{1}_{\{|x| \leq \frac{1}{2n}\}}\). It is straightforward to verify \(V \geq 0\) for large \(n\) and \(V \in \mathcal{H}_\beta\) for \(\beta \leq \frac{1}{4}\). Furthermore, note \(V(x_i) = 1 + R_i cn^{-1/4} \sim \frac{1}{2}\delta_{1 + cn^{-1/4}} + \frac{1}{2}\delta_{1-cn^{-1/4}}\), which is precisely the distribution used Section \ref{section:profile_lower_bound} to prove Theorem \ref{thm:beta_zero_lower}. Thus, the same argument can be carried out to obtain Proposition \ref{prop:discrete_known_noise_sqrtn}.

\subsection{Unknown noise distribution}\label{section:discrete_unknown_noise}
In this section, suppose the noise distribution \(P_\xi\) is unknown to the statistician. Define the space of possible noise distributions 
\begin{equation}\label{space:Xi}
    \Xi = \left\{ P_\xi \text{ probability distribution }: \xi \sim P_\xi \implies E(\xi) = 0, E(\xi^2) = 1, \text{ and } E(\xi^4) \leq C_\xi\right\}  
\end{equation}
where \(C_\xi > 0\) is a large universal constant. To incorporate the ignorance of the noise distribution, the minimax testing risk is modified to take supremum over \(P_\xi \in \Xi\), 
\begin{equation*}
    \mathcal{R}(\varepsilon) = \inf_{\varphi}\left\{ \sup_{\substack{f \in \mathcal{H}_\alpha, \\ V \in \mathcal{D}_0, \\ P_\xi \in \Xi}} P_{f, V, P_\xi}\left\{\varphi = 1\right\} + \sup_{\substack{f \in \mathcal{H}_\alpha, \\ V \in \mathcal{D}_{1,\beta}(\varepsilon),\\P_\xi \in \Xi}} P_{f, V, P_\xi}\left\{\varphi = 0\right\}\right\}. 
\end{equation*}
The minimax separation rate is defined using this modified version of the minimax testing risk. In this section, it will be shown 
\begin{equation}\label{rate:discrete_unknown_noise}
    \varepsilon^*(\alpha, \beta)^2 \asymp n^{-4\alpha} + n^{-\frac{4\beta}{4\beta+1}} + n^{-2\beta}
\end{equation}
for \(\alpha > 0\) and \(\beta \in (0, \frac{1}{2})\). Interestingly, despite the separation between \(\mathcal{D}_0\) and \(\mathcal{D}_{1,\beta}(\varepsilon)\) being measured only with respect to the design points, the minimax separation rate (\ref{rate:discrete_unknown_noise}) matches the rate (\ref{rate:testing}) for the problem (\ref{problem:var0})-(\ref{problem:var1}) where the separation is measured using the \(L^2\) norm. Notably, the term \(n^{-2\beta}\) appears to admit different interpretations in (\ref{rate:discrete_unknown_noise}) and (\ref{rate:testing}). As discussed in Section \ref{section:main_contribution}, the \(n^{-2\beta}\) term in (\ref{rate:testing}) can be attributed to the fixed design nature of (\ref{model}). To reiterate, even if \(V\) is constant across the design points, it may be nonconstant between design points in (\ref{problem:var0})-(\ref{problem:var1}). Since no observations are made between design points, there is inescapable error. On the other hand, in (\ref{rate:discrete_unknown_noise}) this explanation is unsatisfactory because the behavior of \(V\) between design points is irrelevant. Rather, comparing (\ref{rate:discrete_unknown_noise}) to (\ref{rate:discrete_known_noise}) suggests the \(n^{-2\beta}\) term appears due to ignorance of the noise distribution.

\subsubsection{Upper bound}

In view of the fact that (\ref{rate:discrete_unknown_noise}) matches (\ref{rate:testing}), it is natural to anticipate the test statistic from Theorem \ref{theorem: optimal test beta>1/4} will achieve (\ref{rate:discrete_unknown_noise}) for testing (\ref{problem:var_discrete_0})-(\ref{problem:var_discrete_1}). The following theorem establishes this to be the case.

\begin{theorem}\label{theorem: optimal test discrete unknown noise}
  Suppose \(\alpha > 0\) and \(0 < \beta < \frac{1}{2}\). Fix \(\eta \in (0, 1)\). Then, there exists \(C_\eta', C_\eta > 0\) depending only on \(\eta\) such that for all \(C > C_\eta\), we have 
  
  \begin{align*}
      \sup_{\substack{f \in \mathcal{H}_\alpha, \\ V \in \mathcal{D}_0, \\ P_\xi \in \Xi}} P_{f, V, P_\xi}\left\{ \hat{T} > C_\eta' \zeta^2 \right\} 
      + \sup_{\substack{f \in \mathcal{H}_\alpha, \\ V \in \mathcal{D}_{1,\beta}(C\zeta), \\ P_\xi \in \Xi}} P_{f, V, P_\xi}\left\{\hat{T} \leq C_\eta' \zeta^2\right\} \leq \eta
  \end{align*}
  where $\zeta = n^{-2\alpha} + n^{-\beta} + n^{-\frac{2\beta}{4\beta+1}}$ and \(\hat{T}\) is defined by (\ref{def: testing_statistic beta>1/4}). 
\end{theorem}

\noindent In summary, Theorem \ref{theorem: optimal test discrete unknown noise} provides a test for (\ref{problem:var_discrete_0})-(\ref{problem:var_discrete_1}) achieving (\ref{rate:discrete_unknown_noise}) where the noise distribution is unknown.

\subsubsection{Lower bound}
The rate (\ref{rate:discrete_unknown_noise}) simplifies to \(n^{-4\alpha} + n^{-2\beta}\) for \(\beta < \frac{1}{4}\) and \(n^{-4\alpha} + n^{-\frac{4\beta}{4\beta+1}}\) for \(\beta \geq \frac{1}{4}\). Note the rate \(n^{-4\alpha} + n^{-\frac{4\beta}{4\beta+1}}\) matches the rate (\ref{rate:discrete_known_noise}) for \(\beta \geq \frac{1}{4}\) where the noise distribution is known. Therefore, the lower bound arguments from Section \ref{section:discrete_known_noise} carry over directly in the case \(\beta \geq \frac{1}{4}\). Moreover, note that the argument for the \(n^{-4\alpha}\) term carries over for all values of \(\beta > 0\). Thus, it only remains to prove the term \(n^{-2\beta}\) for \(\beta < \frac{1}{4}\).

As mentioned earlier, a comparison of (\ref{rate:discrete_unknown_noise}) to (\ref{rate:discrete_known_noise}) suggests the \(n^{-2\beta}\) is attributable to the statistician's ignorance of the noise distribution. Hence, the lower bound construction requires choosing different noise distributions in the the null and alternative hypotheses. We employ a construction similar to that in Section \ref{section:profile_triviality}.

A sketch of the construction is given here. Define the function \(g : \R \to \R\) with \(g(x) = (1 - 2n|x|)\mathbbm{1}_{\left\{|x| \leq \frac{1}{2n}\right\}}\). Define the prior distribution \(\pi\) where a draw \(V \sim \pi\) is obtained by setting, for \(0 \leq x \leq 1\),
\begin{equation*}
    V(x) = 1 + \sum_{i=0}^{n} R_i n^{-\beta} g\left(x - \frac{i}{n}\right)
\end{equation*}
where \(R_0,...,R_n \overset{iid}{\sim} \Rademacher\left(\frac{1}{2}\right)\). It is easily checked \(V \in \mathcal{H}_\beta\) and \(V \geq 0\) for \(n\) large. Note further \(V(x_i) = 1 + R_i n^{-\beta}\), and so \(\frac{1}{n+1} \sum_{i=0}^{n} \left(V_i - \bar{V}_n\right)^2 \gtrsim n^{-2\beta}\) with high probability where \(\bar{V}_n = \frac{1}{n+1}\sum_{i=0}^{n+1} V(x_i)\). With \(\pi\) in hand, consider 
\begin{align*}
    H_0 &: f \equiv 0, V = \mathbf{1}, P_\xi = \frac{1}{2}N\left(0, 1 + n^{-\beta}\right) + \frac{1}{2}N\left(0, 1 - n^{-\beta}\right), \\
    H_1 &: f \equiv 0, V \sim \pi, P_\xi = N(0, 1).
\end{align*}
Observe \(V_i^{1/2}\xi_i\) has the same distribution under the null and alternative hypotheses. Since \(f \equiv 0\), we have that the data satisfy \(Y_i = V_i^{1/2}\xi_i\) and so the hypotheses are indistinguishable, yielding the lower bound \(\varepsilon^*(\alpha, \beta)^2 \gtrsim n^{-2\beta}\). This sketch is not a complete proof since \(\pi\) is not actually supported on \(\mathcal{D}_{1,\beta}(n^{-\beta})\). However, it is nearly supported on this space and slight technical modifications can rectify this issue. The following proposition gives a rigorous statement of this derivation and is proved in the supplementary material.

\begin{proposition}\label{prop:discrete_unknown_noise_2beta}
    Suppose \(\alpha > 0\) and \(0 < \beta < \frac{1}{4}\). If \(\eta \in (0, 1)\), then there exists \(c_\eta > 0\) not depending on \(n\) such that for all \(0 < c < c_\eta\) and \(n\) sufficiently large depending only on \(\eta\), we have 
    \begin{equation*}
        \inf_{\varphi}\left\{ \sup_{\substack{f \in \mathcal{H}_\alpha, \\ V \in \mathcal{D}_0, \\ P_\xi \in \Xi}} P_{f, V, P_\xi} \left\{\varphi = 1\right\} + \sup_{\substack{f \in \mathcal{H}_\alpha, \\ V \in \mathcal{D}_{1, \beta}(cn^{-\beta}), \\ P_\xi \in \Xi}} P_{f, V, P_\xi} \left\{\varphi = 0\right\}\right\} \geq 1-\eta. 
    \end{equation*}
\end{proposition}

The intuition resembles the intuition discussed in Section \ref{section:profile_triviality}. Without knowledge of the noise distribution, the statistician faces difficulty in deciding whether perceived heteroskedasticity is actually due to ground truth heteroskedasticity or simply a fatter noise distribution. However, in contrast to Section \ref{section:profile_triviality}, \(V\) is guaranteed to be smooth. If the perceived variability is too great, the statistician can attribute it to the noise rather than \(V\). Thus, the triviality of Section \ref{section:profile_triviality} is avoidable. 

\section{Discussion}\label{section:discussion}
The fixed design nature of (\ref{model}) exerts some notable effects. First, the emergence of the term $n^{-2\beta}$ in (\ref{rate:testing}) directly stems from the use of fixed design. Since only responses at the fixed design points are available, the quadratic functional $||V-\bar{V}\mathbf{1}||_2^2$, which involves an integral over the unit interval, is approximated by a discrete sum across the design points, resulting in a cost of \(n^{-2\beta}\). We expect the \(n^{-2\beta}\) does not persist in a random design setting.

In a random design setting, it was recently shown \cite{shen_optimal_2020} the minimax estimation rate involves a term containing both \(\alpha\) and \(\beta\); there is a nontrivial interaction between the two H\"{o}lder exponents. In light of this result, we expect a nontrivial interaction should also appear in the minimax rate for heteroskedasticity testing under a random design analogue of (\ref{model}).

In the fixed design model (\ref{model}), our methodology applies only to the low-smoothness regime \(\beta \in (0, \frac{1}{2})\), which was left totally unaddressed by the previous work \cite{dette_testing_1998,dette_consistent_2002}. It is of clear practical interest to develop methodology to handle higher smoothness \(\beta \geq \frac{1}{2}\). We conjecture that the lower bounds we proved (which hold for all \(\alpha, \beta > 0\)) are sharp, and so it remains to construct a test delivering a matching upper bound. For the problem (\ref{problem:var0})-(\ref{problem:var1}), the main obstacle in extending the statistic developed in Section \ref{section: upper bound} is the presence of well-known boundary effects which can dominate for larger \(\beta\).

The literature contains works which study various hypothesis testing problems in heteroskedastic regression models from a minimax testing perspective. Though they do not address the heteroskedasticity testing problem we are interested in, it is informative to discuss them and highlight the differences. In the heteroskedastic Gaussian sequence model \(Y_j \overset{ind}{\sim} N(\theta_j, \sigma_j^2)\) for \(j \in \mathbb{N}\), Laurent et al. \cite{laurent_nonasymptotic_2012} consider the signal detection problem of testing \(H_0 : \theta = 0\) versus \(H_1: ||\theta|| \geq \varepsilon \text{ and } \theta \in \mathcal{F}\), where \(\mathcal{F} \subset \ell_2\) is a class of signals of interest (e.g. ellipsoid, \(\ell_p\) body, etc.). Laurent et al. assume the variances \(\{\sigma_j^2\}_{j=1}^{\infty}\) are known to the statistician and establish minimax separation rates by characterizing explicitly the rate's dependence on \(\{\sigma_j^2\}_{j=1}^{\infty}\). They directly consider the natural quadratic functional \(||\theta||^2\) when developing their methodology, and the tests they employ make careful use of the statistician's knowledge of the variances. The signal detection problem was also studied in \cite{ermakov_minimax_2004} under a formulation in a related heteroskedastic Gaussian white noise model. The recent work of Chhor et al. \cite{chhor_sparse_2024} considered finite-dimensional, sparse classes of signals in the heteroskedastic Gaussian sequence model of \cite{laurent_nonasymptotic_2012} and obtained sharp rates with explicit dependencies on the dimension, sparsity, and the variances; the problem of detecting separation in some other norms (namely \(||\theta||_t \geq \varepsilon\) for \(t \in [1, \infty]\)) was also settled in \cite{chhor_sparse_2024}. A common theme is that the minimax rates exhibit intricate dependencies on the variances. 

In all of these works, the variances are known to the statistician and so the null hypotheses studied are all simple nulls; there are no nuisance parameters. In contrast, the unknown mean function \(f\) is a nuisance parameter in the heteroskedasticity testing problem we consider, and so the null hypothesis that \(V\) is a constant function is a truly composite null. The minimax testing literature has noted that theory for composite nulls can be much more subtle than that for simple nulls (e.g. see Section 6.2.4 in \cite{gine_mathematical_2016} and references therein). Indeed, the impossibility of adaptation discussed in Section \ref{section:adaptation} can be viewed as a consequence of a composite null hypothesis.

Comminges and Dalalyan \cite{comminges_minimax_2013} investigate a testing problem involving a composite null defined via a quadratic functional. Specializing to one dimension here for discussion, Comminges and Dalalyan consider the homoskedastic regression model \(Y_i = f(x_i) + \xi_i\) where \(x_i\) are design points in \([0, 1]\), \(\xi_i\) are i.i.d. mean zero noise variables with known variance \(\tau^2\), and the unknown mean function lives in a subset of an ellipsoid, that is \(f \in \Sigma \subset \left\{\sum_{j=1}^{\infty} \theta_j \varphi_j : \sum_{j=1}^{\infty} c_j\theta_j^2 \leq 1 \right\}\). Here, \(\{\varphi_j\}_{j = 1}^{\infty}\) is an orthonormal basis of \(L^2([0, 1])\) and \(c_j > 0\) are known coefficients. For a quadratic functional of the form \(Q(f) = \sum_{j=1}^{\infty} q_j \theta_j^2\) where \(\{q_j\}_{j=1}^{\infty}\) are known, Comminges and Dalalyan study the problem of testing \(H_0 : Q(f) = 0 \text{ and } f \in \Sigma\) versus \(H_1 : |Q(f)| \geq \varepsilon^2 \text{ and } f \in \Sigma\). Notably, they allow \(q_j \leq 0\), in which case there may be some nonzero \(f\) which yield \(Q(f) = 0\); the null hypothesis can be truly composite. Comminges and Dalalyan \cite{comminges_minimax_2013} derive minimax separation rates and show they can depend delicately on both \(\{q_j\}_{j=1}^{\infty}\) and \(\{c_j\}_{j=1}^{\infty}\). Their test statistics are developed from considering properly normalized estimators of \(Q(f)\). Their results and methodology are closely related to quadratic functional estimation. 

The heteroskedasticity testing problem we consider can also be seen to be formulated in terms of a quadratic functional. However, it involves a quadratic functional of the variance function rather than the mean. This appears to be a fundamental difference, and it is not clear to us whether the problem can be easily transformed in an information-theoretically equivalent manner into a problem that falls within the framework considered by Comminges and Dalalyan \cite{comminges_minimax_2013}. So it is unclear that a minimax lower bound argument from \cite{comminges_minimax_2013} could be made to go through to establish a lower bound for heteroskedasticity testing. On the methodological side, the test statistics we propose are not developed from consideration of a natural estimand as in \cite{comminges_minimax_2013}, but rather from consideration of a proxy. It is not so clear to us how (if possible) this proxy can be cast into a target involving some basis coefficients to which the methodology of \cite{comminges_minimax_2013} can be applied. It would be an interesting challenge to work out whether their methodology (developed in a homoskedastic regression model with known noise variance) could be suitably modified to optimally test for heteroskedasticity. 

\bibliographystyle{skotekal.bst}
\bibliography{heteroskedastic.bib}

\begin{thebibliography}{10}

\bibitem{arias-castro_remember_2018}
Arias-Castro, E., Pelletier, B., and Saligrama, V. (2018).
\newblock Remember the curse of dimensionality: the case of goodness-of-fit
  testing in arbitrary dimension.
\newblock \emph{J. Nonparametr. Stat.} 30(2):448--471.

\bibitem{baraud_non-asymptotic_2002}
Baraud, Y. (2002).
\newblock Non-asymptotic minimax rates of testing in signal detection.
\newblock \emph{Bernoulli} 8(5):577--606.

\bibitem{bickel_using_1978}
Bickel, P.~J. (1978).
\newblock Using residuals robustly. {I}. {T}ests for heteroscedasticity,
  nonlinearity.
\newblock \emph{Ann. Statist.} 6(2):266--291.

\bibitem{breusch_simple_1979}
Breusch, T.~S. and Pagan, A.~R. (1979).
\newblock A simple test for heteroscedasticity and random coefficient
  variation.
\newblock \emph{Econometrica} 47(5):1287--1294.

\bibitem{brown_robust_1974}
Brown, M.~B. and Forsythe, A.~B. (1974).
\newblock Robust tests for the equality of variances.
\newblock \emph{J. Amer. Statist. Assoc.} 69(346):364--367.

\bibitem{cai_optimal_2006}
Cai, T.~T. and Low, M.~G. (2006).
\newblock Optimal adaptive estimation of a quadratic functional.
\newblock \emph{Ann. Statist.} 34(5):2298--2325.

\bibitem{cai_optimal_2011}
Cai, T.~T. and Yuan, M. (2011).
\newblock Optimal estimation of the mean function based on discretely sampled
  functional data: phase transition.
\newblock \emph{Ann. Statist.} 39(5):2330--2355.

\bibitem{chhor_sparse_2024}
Chhor, J., Mukherjee, R., and Sen, S. (2024).
\newblock Sparse signal detection in heteroscedastic {G}aussian sequence
  models: sharp minimax rates.
\newblock \emph{Bernoulli} 30(3):2127--2153.

\bibitem{chown_detecting_2018}
Chown, J. and Müller, U.~U. (2018).
\newblock Detecting heteroscedasticity in non-parametric regression using
  weighted empirical processes.
\newblock \emph{J. R. Stat. Soc. B. Stat. Methodol.} 80(5):951--974.

\bibitem{comminges_minimax_2013}
Comminges, L. and Dalalyan, A.~S. (2013).
\newblock Minimax testing of a composite null hypothesis defined via a
  quadratic functional in the model of regression.
\newblock \emph{Electron. J. Stat.} 7:146--190.

\bibitem{cook_diagnostics_1983}
Cook, R.~D. and Weisberg, S. (1983).
\newblock Diagnostics for heteroscedasticity in regression.
\newblock \emph{Biometrika} 70(1):1--10.

\bibitem{dette_consistent_2002}
Dette, H. (2002).
\newblock A consistent test for heteroscedasticity in nonparametric regression
  based on the kernel method.
\newblock \emph{J. Statist. Plann. Inference} 103:311--329.

\bibitem{dette_simple_2009}
Dette, H. and Hetzler, B. (2009).
\newblock A simple test for the parametric form of the variance function in
  nonparametric regression.
\newblock \emph{Ann. Inst. Statist. Math.} 61(4):861--886.

\bibitem{dette_testing_1998}
Dette, H. and Munk, A. (1998).
\newblock Testing heteroscedasticity in nonparametric regression.
\newblock \emph{J. R. Stat. Soc. Ser. B Stat. Methodol.} 60(4):693--708.

\bibitem{dette_new_2007}
Dette, H., Neumeyer, N., and Van~Keilegom, I. (2007).
\newblock A new test for the parametric form of the variance function in
  non-parametric regression.
\newblock \emph{J. R. Stat. Soc. Ser. B Stat. Methodol.} 69(5):903--917.

\bibitem{devore_constructive_1993}
DeVore, R.~A. and Lorentz, G.~G. (1993).
\newblock \emph{Constructive approximation}, volume 303 of \emph{Grundlehren
  der mathematischen Wissenschaften [Fundamental Principles of Mathematical
  Sciences]}.
\newblock Springer-Verlag, Berlin.

\bibitem{efromovich_optimal_1996}
Efromovich, S. and Low, M. (1996).
\newblock On optimal adaptive estimation of a quadratic functional.
\newblock \emph{Ann. Statist.} 24(3):1106--1125.

\bibitem{ermakov_minimax_2004}
Ermakov, M.~S. (2004).
\newblock Minimax detection of a signal in weighted {G}aussian white noise.
\newblock \emph{Zap. Nauchn. Sem. S.-Peterburg. Otdel. Mat. Inst. Steklov.
  (POMI)} 320:54--68, 226.

\bibitem{eubank_detecting_1993}
Eubank, R.~L. and Thomas, W. (1993).
\newblock Detecting heteroscedasticity in nonparametric regression.
\newblock \emph{J. Roy. Statist. Soc. Ser. B} 55(1):145--155.

\bibitem{gasser_residual_1986}
Gasser, T., Sroka, L., and Jennen-Steinmetz, C. (1986).
\newblock Residual variance and residual pattern in nonlinear regression.
\newblock \emph{Biometrika} 73(3):625--633.

\bibitem{gastwirth_impact_2009}
Gastwirth, J.~L., Gel, Y.~R., and Miao, W. (2009).
\newblock The impact of {L}evene's test of equality of variances on statistical
  theory and practice.
\newblock \emph{Statist. Sci.} 24(3):343--360.

\bibitem{gine_mathematical_2016}
Gin\'e, E. and Nickl, R. (2016).
\newblock \emph{Mathematical foundations of infinite-dimensional statistical
  models}, volume [40] of \emph{Cambridge Series in Statistical and
  Probabilistic Mathematics}.
\newblock Cambridge University Press, New York.

\bibitem{gine_simple_2008}
Giné, E. and Nickl, R. (2008).
\newblock A simple adaptive estimator of the integrated square of a density.
\newblock \emph{Bernoulli} 14(1):47--61.

\bibitem{goldfeld_tests_1965}
Goldfeld, S.~M. and Quandt, R.~E. (1965).
\newblock Some tests for homoscedasticity.
\newblock \emph{J. Amer. Statist. Assoc.} 60:539--547.

\bibitem{hall_asymptotically_1990}
Hall, P., Kay, J.~W., and Titterington, D.~M. (1990).
\newblock Asymptotically optimal difference-based estimation of variance in
  nonparametric regression.
\newblock \emph{Biometrika} 77(3):521--528.

\bibitem{harrison_test_1979}
Harrison, M.~J. and McCabe, B. P.~M. (1979).
\newblock A test for heteroscedasticity based on ordinary least squares
  residuals.
\newblock \emph{J. Amer. Statist. Assoc.} 74(366, part 1):494--499.

\bibitem{ingster_nonparametric_2003}
Ingster, Y.~I. and Suslina, I.~A. (2003).
\newblock \emph{Nonparametric goodness-of-fit testing under {Gaussian} models},
  volume 169 of \emph{Lecture {Notes} in {Statistics}}.
\newblock Springer-Verlag, New York.

\bibitem{koenker_robust_1982}
Koenker, R. and Bassett, G., Jr. (1982).
\newblock Robust tests for heteroscedasticity based on regression quantiles.
\newblock \emph{Econometrica} 50(1):43--61.

\bibitem{laurent_nonasymptotic_2012}
Laurent, B., Loubes, J.-M., and Marteau, C. (2012).
\newblock Non asymptotic minimax rates of testing in signal detection with
  heterogeneous variances.
\newblock \emph{Electron. J. Stat.} 6:91--122.

\bibitem{lepski_problem_1990}
Lepski, O.~V. (1990).
\newblock A problem of adaptive estimation in {Gaussian} white noise.
\newblock \emph{Teor. Veroyatnost. i Primenen.} 35(3):459--470.

\bibitem{levene_robust_1960}
Levene, H. (1960).
\newblock Robust tests for equality of variances.
\newblock In \emph{Contributions to probability and statistics}, volume~2 of
  \emph{Stanford Studies in Mathematics and Statistics}, pp. 278--292. Stanford
  Univ. Press, Stanford, CA.

\bibitem{shen_optimal_2020}
Shen, Y., Gao, C., Witten, D., and Han, F. (2020).
\newblock Optimal estimation of variance in nonparametric regression with
  random design.
\newblock \emph{Ann. Statist.} 48(6):3589--3618.

\bibitem{spokoiny_adaptive_1996}
Spokoiny, V.~G. (1996).
\newblock Adaptive hypothesis testing using wavelets.
\newblock \emph{Ann. Statist.} 24(6):2477--2498.

\bibitem{tsybakov_introduction_2009}
Tsybakov, A.~B. (2009).
\newblock \emph{Introduction to nonparametric estimation}.
\newblock Springer Series in Statistics. Springer, New York.
\newblock Revised and extended from the 2004 French original, Translated by
  Vladimir Zaiats.

\bibitem{wang_effect_2008}
Wang, L., Brown, L.~D., Cai, T.~T., and Levine, M. (2008).
\newblock Effect of mean on variance function estimation in nonparametric
  regression.
\newblock \emph{Ann. Statist.} 36(2):646--664.

\bibitem{white_heteroskedasticity-consistent_1980}
White, H. (1980).
\newblock A heteroskedasticity-consistent covariance matrix estimator and a
  direct test for heteroskedasticity.
\newblock \emph{Econometrica} 48(4):817--838.

\bibitem{wu_polynomial_2020}
Wu, Y. and Yang, P. (2020).
\newblock Polynomial methods in statistical inference: Theory and practice.
\newblock \emph{Found. Trends Commun. Inf. Theory} 17(4):402--586.

\bibitem{zheng_consistent_1996}
Zheng, J.~X. (1996).
\newblock A consistent test of functional form via nonparametric estimation
  techniques.
\newblock \emph{J. Econometrics} 75(2):263--289.

\bibitem{zheng_testing_2009}
Zheng, X. (2009).
\newblock Testing heteroscedasticity in nonlinear and nonparametric
  regressions.
\newblock \emph{Canad. J. Statist.} 37(2):282--300.

\end{thebibliography}

\newpage 

\begin{center}
    {\LARGE Supplement to ``Optimal heteroskedasticity testing in nonparametric regression''	
        
    }		
    \medskip
    
    {\large Subhodh Kotekal \quad and\quad Soumyabrata Kundu}
    \medskip
\end{center}
In this supplement, we provide a table of contents and the deferred proofs for results stated in the main text. Additionally, Appendix \ref{appendix:adaptation_simulation} presents a numerical simulation to illustrate the impossibility of adaptation discussed in Section \ref{section:adaptation}.

\tableofcontents

\begin{appendix}

\section{Upper bound}\label{appendix:upper_bound}
This section contains proofs of the upper bounds presented in Section \ref{section: upper bound}.

\subsection{\texorpdfstring{Mean squared error of $\hat{T}$}{Mean squared error}}

The proof of Proposition \ref{prop: MSE of T_hat beta>1/4} is presented here. For ease of notation, we will simply write \(P\) and \(E\) instead of \(P_{f,V}\) and \(E_{f,V}\). 

\begin{lemma}\label{lemma: kernel sums to 1}
For any fixed integer $i$ such that $\lceil nh\rceil \leq i\leq  n- \lceil nh\rceil$, we have $\sum_{j=0}^{n-1} K_n^h(i-j) = 1$.
\end{lemma}
\begin{proof}
Since $K$ is symmetric about $0$, by direct calculation we have,
\begin{align*}
\sum_{j=0}^{n-1} \mathbbm{1}_{\{|i-j|\geq 2\}}\int_{|i-j|/n}^{(|i-j|+1)/n} \frac{1}{h}K\left( \frac{u}{h} \right)du 
&= \sum_{j=0}^{i-2} \int_{(i-j)/n}^{(i-j+1)/n} \frac{1}{h}K\left( \frac{u}{h} \right)du  +  \sum_{j=i+2}^{n-1} \int_{(j-i)/n}^{(j-i+1)/n} \frac{1}{h}K\left( \frac{u}{h} \right)du \\
&= \sum_{j=0}^{i-2} \int_{(j-i-1)/n}^{(j-i)/n} \frac{1}{h}K\left( \frac{u}{h} \right)du  +   \sum_{j=i+2}^{n-1} \int_{(j-i)/n}^{(j-i+1)/n} \frac{1}{h}K\left( \frac{u}{h} \right)du \\
&=  \int_{-(i+1)/n}^{-2/n} \frac{1}{h}K\left( \frac{u}{h} \right)du  +  \int_{2/n}^{1-i/n} \frac{1}{h}K\left( \frac{u}{h} \right)du \\
&=  \int_{-(i+1)/n}^{1-i/n} \frac{1}{h}K\left( \frac{u}{h} \right)du  -  \int_{-2/n}^{2/n} \frac{1}{h}K\left( \frac{u}{h} \right)du \\
&= {\int_{-1/n}^{1} \frac{1}{h} K\left(\frac{u - i/n}{h}\right) \, du - \int_{-2/n}^{2/n} \frac{1}{h}K\left( \frac{u}{h} \right)du }\\
&=  \int_{-1/n}^{1} \frac{1}{h}K\left( \frac{i/n - u}{h} \right)du  -  \int_{-2/n}^{2/n} \frac{1}{h}K\left( \frac{u}{h} \right)du. 
\end{align*}
The final line follows from the symmetry of \(K\) about \(0\). Keep in mind that when $\lceil nh\rceil \leq i\leq n- \lceil nh\rceil$, we have $h< i/n\leq 1-h$. Considering that $K\left( \frac{i/n - u}{h} \right) = 0$ for $u$ outside the interval $[\frac{i}{n}-h, \frac{i}{n}+h]$, which is a subset of $[0,1]$, we can deduce the following,
\begin{align*}
\sum_{j=0}^{n-1} \mathbbm{1}_{|i-j|\geq 2} \int_{|i-j|/n}^{(|i-j|+1)/n} \frac{1}{h}&K\left( \frac{u}{h} \right)du 
=  \int_{i/n-h}^{i/n+h} \frac{1}{h}K\left( \frac{i/n - u}{h} \right)du  -  \int_{-2/n}^{2/n} \frac{1}{h}K\left( \frac{u}{h} \right)du \\
&=  \int_{-1}^1 K(u)du -  \int_{-2/n}^{2/n} \frac{1}{h}K\left( \frac{u}{h} \right)du  = 1 -  \int_{-2/n}^{2/n} \frac{1}{h}K\left( \frac{u}{h} \right)du. 
\end{align*}
Using the definition of $K_n^h$ in (\ref{def: kernel}), the proof is complete.
\end{proof}

\begin{proof}[Proof of Proposition \ref{prop: MSE of T_hat beta>1/4}]
To reduce the notational clutter, the sum ranges from $0$ to $n-1$ whenever the limits are not specified. \newline

\noindent \textbf{Bias Calculation:}
In bounding the mean squared error, we first examine the bias. Note
\begin{align*}
E_{f, V} (\widehat{T})
&= \frac{1}{n}\sum_{|i-j|\geq 2} K_n^h(i-j) (\delta_i^2 + W_i) ( \delta_j^2 + W_j) -  \frac{1}{n^2}\sum_{|i-j|\geq 2} (\delta_i^2 + W_i) ( \delta_j^2 + W_j)  \left(= {T}_1 - {T}_2\right).
\end{align*}
Consider that, from (\ref{def: kernel}), for fixed integer $i$ we get
\begin{align}\label{eqn: mod kernel sum is bounded}
\left| \sum_j K_n^h(i-j) \right| \leq \frac{1}{c} \sum_j \int_{|i-j|/n}^{(|i-j|+1)/n} \left| \frac{1}{h} K\left(\frac{u}{h}\right)\right|du \leq \frac{||K||_1}{c}.
\end{align}
Let us look at $T_1$ first:
\begin{align}
 &\left|T_1 -\frac{1}{n}\sum_i (\delta_i^2+W_i)^2 \right| \nonumber\\
=&  \frac{1}{n}\left|\sum_{|i-j|\geq 2} K_n^h(i-j) (\delta_i^2 + W_i) ( \delta_j^2 + W_j) -  \sum_{i}(\delta_i^2+W_i)^2\right| \nonumber\\
\leq&  \frac{1}{n}\left|\sum_{i,j} K_n^h(i-j) \delta_i^2 \delta_j^2 -\sum_i\delta_i^4\right|  + \frac{2}{n}\left|\sum_{i,j}K_n^h(i-j)\delta_i^2W_j - \sum_i\delta_i^2W_i\right|  + \frac{1}{n}\left| \sum_{i,j}K_n^h(i-j)W_iW_j - \sum_iW_i^2 \right|\nonumber\\
\lesssim& n^{-4(\alpha\wedge 1)} + \frac{2n^{-2(\alpha\wedge 1)}}{n}\sum_i\left| \sum_jK_n^h(i-j)W_j-W_i \right| + \frac{1}{n}\left| \sum_{i,j}K_n^h(i-j)W_iW_j - \sum_iW_i^2 \right|\label{eqn: bias split}
\end{align}
where the second equality holds because $K_n^h(i-j)$ equals $0$ when $|i-j|<2$. Moreover, $K_n^h(\cdot)$ is symmetric, meaning $K_n^h(i-j) = K_n^h(j-i)$ for all $i$ and $j$. The final inequality holds due to (\ref{eqn: mod kernel sum is bounded}) and the fact that $|\delta_i|\lesssim n^{-\alpha}$ when $\alpha<1$ and $|\delta_i|\lesssim n^{-1}$ when $\alpha>1$, since $f\in \mathcal{H}_\alpha \subset \mathcal{H}_1$. To bound the second term in (\ref{eqn: bias split}) consider that
\begin{align}
\sum_{i}\left| \sum_jK_n^h(i-j)W_j-W_i \right| \nonumber
&= \sum_{i}\left| \sum_jK_n^h(i-j)(W_j-W_i) - W_i\left(1 - \sum_iK_n^h(i-j)\right) \right|\nonumber\\
&\lesssim  \sum_{i,j}|K_n^h(i-j)||W_j-W_i| + \sum_i\left|1 - \sum_jK_n^h(i-j)\right|\nonumber\\
&=\sum_{i,j}|K_n^h(i-j)||W_j-W_i| + \sum_{\substack{i\in [0, \lceil nh\rceil] \\ \cup [n-\lceil nh\rceil, n]}}\left|1 - \sum_jK_n^h(i-j)\right|\nonumber\\
&\lesssim  \sum_{i,j} |K_n^h(i-j)|\left|\frac{i-j}{n}\right|^\beta + nh\label{eqn: bias split second term}
\end{align}
where the second last equality follows from Lemma \ref{lemma: kernel sums to 1} and the last inequality holds due to (\ref{eqn: mod kernel sum is bounded}). Finally, from the fact that $V\in \mathcal{H}_\beta$ and using the substitution $i-j=t$, we get
\begin{align}
\sum_{i,j} |K_n^h(i-j)|\left|\frac{i-j}{n}\right|^\beta
& \leq \sum_i\sum_{t=i-(n-1)}^{i} |K_n^h(t)|\left|\frac{t}{n}\right|^\beta\nonumber\\ 
& \leq n\sum_{t=-(n-1)}^{n-1} |K_n^h(t)||t/n|^\beta \nonumber\\ 
&=  n\sum_{t=-(n-1)}^{n-1} \left(\int_{|t|/n}^{(|t|+1)/n} \left| \frac{1}{h}K\left(\frac{u}{h}\right) \right| du\right) \left|\frac{t}{n}\right|^\beta \nonumber\\
&\leq n|h|^\beta \sum_{t=-(n-1)}^{n-1} \int_{|t|/n}^{(|t|+1)/n} \left| \frac{1}{h}K\left(\frac{u}{h}\right) \right| \left|\frac{u}{h}\right|^\beta du\nonumber \\
&= n|h|^\beta \sum_{t=-(n-1)}^{n-1} \int_{|t|/nh}^{(|t|+1)/nh}|u|^\beta |K(u)|du\nonumber \\
&\leq {n|h|^\beta \int_{-1}^{1} |u|^\beta |K(u)| \, du} \nonumber \\
&\leq { n|h|^\beta \int_{-1}^{1} |K(u)| \, du} \nonumber \\
&\leq  n|h|^\beta ||K||_1. \label{eqn: bias split second term kernel}
\end{align}
Here, we have used that \(K\) has support on \([-1, 1]\) and is bounded. Plugging (\ref{eqn: bias split second term kernel}) into (\ref{eqn: bias split second term}) we get the following bound,
\begin{equation}
     \sum_{i}\left| \sum_jK_n^h(i-j)W_j-W_i \right| \lesssim  nh^\beta + nh. \label{eqn: bias split second term bound}
\end{equation}
Let us move on to the remaining term in (\ref{eqn: bias split}). First, define the function $\tilde{W}_{n-1}:\Z\to \R$, with $\tilde{W}_{n-1}(i) = W_i$ for $0\leq i\leq n-1$ and $\tilde{W}_{n-1}(i) = 0$ otherwise. Note that, $\sum_{j\in \Z} K_h^n(i-j) = 1$ for any $i$. With this understanding, taking the substitution $i-j=t$ gives us
\begin{align}
 \left|\sum_{i,j} K_n^h(i-j)W_iW_j-\sum_{i} W_i^2\right| \nonumber
&= \left|\sum_{i,j\in \Z} K_n^h(i-j)\tilde{W}_{n-1}(i)\tilde{W}_{n-1}(j)-\sum_{i\in \Z}\tilde{W}_{n-1}(i)^2\right|\nonumber\\
&= \left|\sum_{i,t\in \Z} K_n^h(t)\tilde{W}_{n-1}(i)\tilde{W}_{n-1}(i-t)-\sum_{i\in \Z}\tilde{W}_{n-1}(i)^2\right|\nonumber\\
&= \left|\sum_{i,t\in \Z} K_n^h(t)\left(\tilde{W}_{n-1}(i)\tilde{W}_{n-1}(i-t)-\tilde{W}_{n-1}(i)^2\right)\right|\nonumber\\
&= \left|\sum_{t\in \Z} K_n^h(t)\left(\sum_{i\in\Z}\tilde{W}_{n-1}(i)\tilde{W}^-_{n-1}(t-i)-\tilde{W}_{n-1}(i)\tilde{W}_{n-1}^-(0-i)\right)\right|\nonumber\\
&= \left|\sum_{t\in \Z} K_n^h(t)\left(\tilde{W}_{n-1}\ast_D \tilde{W}^-_{n-1}(t)-\tilde{W}_{n-1}\ast_D\tilde{W}_{n-1}^-(0)\right)\right|\label{eqn: bias split third term}
\end{align}
where \(\tilde{W}_{n-1}^{-}(i) = \tilde{W}_{n-1}(-i)\) and $\ast_D$ is the discrete convolution. Note that, for any fixed $k\in \Z$ and $i\in [\max\{-k,0\}, n-1-\max\{k,0\}]$, we have $|\tilde{W}_{n-1}(i+k) - \tilde{W}_{n-1}(i)| \leq 2M|k/n|^{\beta}$ since $V\in \mathcal{H}_\beta(M)$. Then, using Proposition \ref{prop: convolution}, we have 
\begin{equation}
    \left|\tilde{W}_{n-1}\ast_D \tilde{W}_{n-1}^-(t) - \tilde{W}_{n-1}\ast_D \tilde{W}_{n-1}^-(0)\right|\leq nC(\beta, M) \left|\frac{t}{n-1}\right|^{2\beta} \label{eqn: convolution}
\end{equation}
for any \(t \in [-(n-1), n-1]\cap \Z\). Consider that \(K_n^h(t) = 0\) for \(|t| > nh\) since \(K\) is supported on \([-1, 1]\). Then, from (\ref{eqn: bias split third term}) and (\ref{eqn: convolution}) it follows that 
\begin{align}
\frac{1}{n} \left|\sum_{i,j} K_n^h(i-j)W_iW_j-\sum_{i} W_i^2\right| &= \frac{1}{n} \left|\sum_{t\in \Z} K_n^h(t)\left(\tilde{W}_{n-1}\ast_D \tilde{W}^-_{n-1}(t)-\tilde{W}_{n-1}\ast_D\tilde{W}_{n-1}^-(0)\right)\right|\nonumber \\
&\leq \frac{1}{n } \sum_{t=-\lceil nh \rceil}^{\lceil nh \rceil} |K_n^h(t)| \cdot \left|\tilde{W}_{n-1}\ast_D \tilde{W}^-_{n-1}(t)-\tilde{W}_{n-1}\ast_D\tilde{W}_{n-1}^-(0)\right|\nonumber \\
&\lesssim \frac{1}{n} \sum_{t = -\lceil nh\rceil}^{\lceil nh\rceil} |K_n^h(t)||t/n|^{2\beta} \nonumber\\
&\leq \frac{1}{n} \sum_{t \in \Z} |K_n^h(t)||t/n|^{2\beta} \nonumber\\
&\lesssim \frac{|h|^{2\beta}}{c} \sum_{t \in \Z} \int_{|t|/n}^{(|t|+1)/n}|u/h|^{2\beta} |h^{-1}K(u/h)|du  \nonumber\\
&\lesssim |h|^{2\beta} \int_{-\infty}^{\infty} |u|^{2\beta} |K(u)|du  \nonumber\\
&= |h|^{2\beta} \int_{-1}^{1} |u|^{2\beta} |K(u)|du \nonumber\\
&\leq |h|^{2\beta} ||K||_1.\label{eqn: bias split third term bound}
\end{align}
Here, we have used \(K\) is supported on \([-1, 1]\). Plugging (\ref{eqn: bias split second term bound}) and (\ref{eqn: bias split third term bound}) into (\ref{eqn: bias split}) and using the expression $ab \lesssim a^2 + b^2 $ for reals $a$ and $b$ gives us,
\begin{equation}
    \left|T_1 -\frac{1}{n}\sum_i (\delta_i^2+W_i)^2 \right| \lesssim n^{-4(\alpha\wedge 1)} + n^{-2(\alpha \wedge 1)}(h^{\beta} + h) + h^{2\beta} \lesssim n^{-4(\alpha\wedge 1)} + h^{2\beta} + h^2. \label{eqn: bias T_1 bound}
\end{equation}
On the other hand, from the boundedness of $f$ and $V$, we have,
\begin{align}
\left|T_2 - \left(\frac{1}{n}\sum_{i} \delta_i^2 + W_i \right)^2 \right|  
&=  \left|\frac{1}{n^2}\sum_{|i-j|\geq 2} (\delta_i^2 + W_i ) (\delta_j^2 + W_j ) - \left(\frac{1}{n}\sum_{i} \delta_i^2 + W_i \right)^2 \right| \nonumber\\
&= \frac{1}{n^2}\sum_{i} 2(\delta_i^2 + W_i) (\delta_{i+1}^2 + W_{i+1}) +  (\delta_i^2 + W_i)^2 \nonumber\\
&\asymp n^{-1}.\label{eqn: bias T_2 bound}
\end{align}
In summary, using (\ref{eqn: bias T_1 bound}) and (\ref{eqn: bias T_2 bound}), we can derive the following bound on the bias:
\begin{align}
\left| E(\widehat{T}) - T  \right| \nonumber
&= \left| T_1-T_2 -   \frac{1}{n}\sum_i \left(\delta_i^2+W_i - \bar{W}_n - \overline{\delta^2_n}\right)\right|\nonumber\\
&= \left| T_1-T_2 -   \frac{1}{n}\sum_i  \delta_i^2 + W_i ^2 +  \left(\frac{1}{n}\sum_{i} \delta_i^2 + W_i \right)^2\right|\nonumber\\
&\leq \left\vert T_1 - \frac{1}{n}\sum_i  (\delta_i^2 + W_i) ^2\right\vert +  \left|T_2 - \left(\frac{1}{n}\sum_{i} \delta_i^2 + W_i \right)^2 \right|\nonumber  \\
&\lesssim n^{-4(\alpha\wedge 1)} + h^{2\beta} + h^2 + n^{-1}. \label{eqn: bias bound}
\end{align}
With the bias handled, let us now investigate the variance. \newline 

\noindent \textbf{Variance Calculation: }

\noindent First, consider that we can write 
\begin{align*}
\hat{T} = \frac{1}{n^2} \sum_{|i - j| \geq 2} R_i^2R_j^2 \left(nK_n^{h}(i-j) - 1\right)
\end{align*}
and so 
\begin{align*}
\Var\left(\hat{T}\right) &= \frac{1}{n^4} \sum_{|i-j| \geq 2} \sum_{|k - l| \geq 2} \Cov\left(R_i^2R_j^2, R_k^2R_l^2\right) \left(nK_n^{h}(i-j) - 1\right)\left(n K_n^{h}(k-l) - 1\right). 
\end{align*}
For \(|k-l| \geq 2\) such that \(\{k, l\} \cap \{i-1, i, i+1, j-1, j, j+1\} = \emptyset\), the random variables \(R_i^2, R_j^2,R_k^2,\) and \(R_l^2\) are all independent, resulting in the covariance to be zero. We will now address the remaining terms. To simplify notation, we will define a set for each pair \(i,j\):
\begin{equation*}
G_{ij} = \left\{i-1,i,i+1,j-1,j,j+1\right\} \cap \left\{0,...,n-1\right\}. 
\end{equation*}
Then we have 
\begin{align}
\Var\left(\hat{T}\right) &= \frac{1}{n^4} \sum_{|i-j| \geq 2} \sum_{k, l \in G_{ij}} \Cov\left(R_i^2R_j^2, R_k^2R_l^2\right) \left(nK_n^{h}(i-j) - 1\right)\left(n K_n^{h}(k-l) - 1\right) \label{eqn:G_var}\\
&+ \frac{2}{n^4} \sum_{|i-j| \geq 2} \sum_{\substack{|k-l| \geq 2 \\ k \in G_{ij}\\ l \not \in G_{ij}}} \Cov\left(R_i^2R_j^2, R_k^2R_l^2\right) \left(nK_n^{h}(i-j) - 1\right)\left(n K_n^{h}(k-l) - 1\right). \label{eqn:B_var}
\end{align}

\noindent Before we delve into the calculations let us prove the following two identities about the kernel,
\begin{align}
& |nK_n^h(i-j)-1| \lesssim 1+h^{-1}, \label{eqn: bound on kernel}\\
& \sum_{j} |nK_n^h(i-j)-1| \lesssim n. \label{eqn: bound on sum of kernel}
\end{align}
The bound (\ref{eqn: bound on sum of kernel}) is a consequence of (\ref{eqn: mod kernel sum is bounded}), and (\ref{eqn: bound on kernel}) follows from \(||K_n^h||_\infty \lesssim \frac{1}{nh}\) since 
\begin{align*}
|K_n^h(t)| \leq \frac{\left|\int_{|t|/n}^{(|t|+1)/n} \frac{1}{h}K\left(\frac{u}{h}\right) \, du \right|}{\left| 1 - \int_{-2/nh}^{2/nh} K(u) du \right|} \leq \frac{||K||_\infty}{cnh} \lesssim \frac{1}{nh}.  
\end{align*}

\noindent \textbf{Bounding (\ref{eqn:G_var}):} Let us examine (\ref{eqn:G_var}) first. Since \(V\) and \(f\) are bounded, we have \(\Cov(R_i^2R_j^2, R_k^2R_l^2) \lesssim 1\) and so 
\begin{align}
&\frac{1}{n^4} \sum_{|i-j| \geq 2} \sum_{k, l \in G_{ij}} \Cov\left(R_i^2R_j^2, R_k^2R_l^2\right) \left(nK_n^{h}(i-j) - 1\right)\left(n K_n^{h}(k-l) - 1\right) \nonumber\\
&\lesssim \frac{1}{n^4} \sum_{|i-j| \geq 2} |nK_n^h(i-j) - 1| \sum_{k, l \in G_{ij}} |nK_n^h(k-l)-1| \nonumber\\
&{\lesssim} \frac{1}{n^4} \sum_{|i-j| \geq 2} |nK_n^h(i-j) - 1| \left( |G_{ij}|^2 (1+h^{-1})\right) \nonumber\\
&\lesssim \frac{1}{n^4} n^2 (1+h^{-1}) \\
&\lesssim \frac{1}{n^2} + \frac{1}{n^2h}.\label{eqn:G_var bound}
\end{align}

\noindent \textbf{Bounding (\ref{eqn:B_var}):} It is convenient to split (\ref{eqn:B_var}) further into,
\begin{align}
&\frac{2}{n^4} \sum_{|i-j| \geq 2} \sum_{\substack{|k-l| \geq 2 \\ k \in G_{ij}\\ l \not \in G_{ij}}} \Cov\left(R_i^2R_j^2, R_k^2R_l^2\right) \left(nK_n^{h}(i-j) - 1\right)\left(n K_n^{h}(k-l) - 1\right) \nonumber \\
&= \frac{2}{n^4} \sum_{|i-j| \geq 3} \sum_{\substack{|k-l| \geq 2 \\ k \in G_{ij}\\ l \not \in G_{ij}}} \Cov\left(R_i^2R_j^2, R_k^2R_l^2\right) \left(nK_n^{h}(i-j) - 1\right)\left(n K_n^{h}(k-l) - 1\right)\label{eqn:B_var_sep} \\
&+ \frac{2}{n^4} \sum_{|i-j| = 2} \sum_{\substack{|k-l| \geq 2 \\ k \in G_{ij}\\ l \not \in G_{ij}}} \Cov\left(R_i^2R_j^2, R_k^2R_l^2\right) \left(nK_n^{h}(i-j) - 1\right)\left(n K_n^{h}(k-l) - 1\right).\label{eqn:B_var_close}
\end{align}
This split is convenient because for \(|i-j| \geq 3\), we have that \(\{i-1,i,i+1\} \cap \{j-1,j,j+1\} = \emptyset\), and so \(k \in G_{ij}\) implies that either \(|k-i| < 2\) or \(|k-j| < 2\) but not both. The term (\ref{eqn:B_var_sep}) has the most involved analysis, so let us first deal with (\ref{eqn:B_var_close}). Using the boundedness of \(f\) and \(V\), as well as the inequalities (\ref{eqn: bound on sum of kernel}) and (\ref{eqn: bound on kernel}), we have 
\begin{align}
&\frac{2}{n^4} \sum_{|i-j| = 2} \sum_{\substack{|k-l| \geq 2 \\ k \in G_{ij}\\ l \not \in G_{ij}}} \Cov\left(R_i^2R_j^2, R_k^2R_l^2\right) \left(nK_n^{h}(i-j) - 1\right)\left(n K_n^{h}(k-l) - 1\right) \nonumber \\
&{\lesssim} \frac{1}{n^4} \sum_{|i-j| = 2} (|nK_n^h(i-j) - 1|) \left( \sum_{\substack{k \in G_{ij}, \\ l \not \in G_{ij}}} |nK_n^h(k-l) -1| \right) \nonumber \\
&\lesssim \frac{1}{n^4} \sum_{|i-j| = 2} (1+h^{{-1}}) |G_{ij}|n \nonumber \\
&\lesssim \frac{1}{n^4}  (n^2+n^2h^{{-1}})  \nonumber\\
&\lesssim \frac{1}{n^2} + \frac{1}{n^2h}.  \label{eqn: B_var sep bound}
\end{align}
We will now bound (\ref{eqn:B_var_sep}). Fix \(|i-j| \geq 3\) and \(|k-l| \geq 2\) with \(k\in G_{ij}, l \not \in G_{ij}\). Note that since \(|i-j| \geq 3\), we have either \(k \in \{i-1,i,i+1\}\) or \(k \in \{j-1, j, j+1\}\) but not both. Without loss of generality, suppose \(k \in \{i-1,i,i+1\}\). Note that the three random objects \((R_i^2, R_k^2), R_j^2, R_l^2\) are mutually independent. Then consider
\begin{align*}
\Cov\left(R_i^2R_j^2,R_k^2R_l^2\right) &= E\left(R_i^2R_j^2R_k^2R_l^2\right) - E(R_i^2R_j^2)E(R_k^2R_l^2) \\
&= E(R_i^2R_k^2)E(R_j^2)E(R_l^2) - E(R_i^2)E(R_j^2)E(R_k^2)E(R_l^2) \\
&=E(R_j^2)E(R_l^2)\Cov(R_i^2,R_k^2) .
\end{align*}
For ease of notation, define \(U_r = \delta_r^2 + W_r\). Note \(E(R_j^2) = U_j\) and \(E(R_l^2) = U_l\). Then 
\begin{align}
&\frac{2}{n^4} \sum_{|i-j| \geq 3} \sum_{\substack{|k-l| \geq 2 \\ k \in G_{ij}\\ l \not \in G_{ij}}} \Cov\left(R_i^2R_j^2, R_k^2R_l^2\right) \left(nK_n^{h}(i-j) - 1\right)\left(n K_n^{h}(k-l) - 1\right) \nonumber\\
&= \frac{4}{n^4} \sum_{|i-j| \geq 3} \sum_{\substack{|k-l| \geq 2 \\ |k-i| \leq 1 \\ l \not \in G_{ij}}} U_jU_l \Cov(R_i^2, R_k^2) \left(nK_n^{h}(i-j) - 1\right)\left(n K_n^{h}(k-l) - 1\right) \nonumber\\
&= \frac{4}{n^4} \sum_{|i-k| \leq 1} \Cov(R_i^2, R_k^2) \sum_{|j-i| \geq 3} (nK_n^h(i-j)-1) U_j \sum_{\substack{|l-k| \geq 2 \\ l \not \in G_{ij}}} U_l(nK_n^h(k-l) - 1) \nonumber\\
&= \frac{4}{n^4} \sum_{|i-k| \leq 1} \Cov(R_i^2, R_k^2) \left( \left(\sum_{|j-i| \geq 3} (nK_n^h(i-j)-1) U_j\right) \left(\sum_{l} U_l(nK_n^h(k-l) - 1)\right) \right. \nonumber\\
&\left. - \sum_{|j-i| \geq 3} (nK_n^h(i-j)-1) U_j \left(\sum_{\substack{|l-k| < 2, \\ or\  l \in G_{ij}}}U_l(nK_n^h(k-l) - 1)\right)\right) \nonumber\\
&= \frac{4}{n^4} \sum_{|i-k| \leq 1} \Cov(R_i^2, R_k^2) \left( \left(\sum_{j} (nK_n^h(i-j)-1) U_j\right) \left(\sum_{l} U_l(nK_n^h(k-l) - 1)\right) \right. \nonumber\\
&\left. - \left(\sum_{|j-i| < 3} (nK_n^h(i-j)-1) U_j\right) \left(\sum_{l} U_l(nK_n^h(k-l) - 1)\right)\right. \nonumber\\
&\left. - \sum_{|j-i| \geq 3} (nK_n^h(i-j)-1) U_j\left( \sum_{\substack{|l-k| < 2, \\ or \ l \in G_{ij}}}U_l(nK_n^h(k-l) - 1)\right)\right) \label{eqn:B_var_sep_split}.
\end{align} 
There are three main terms to handle. The first term is the important one, with the latter two being remainders which we will bound now. Consider that by boundedness of \(f\) and \(V\) we have 
\begin{align}
&\frac{4}{n^4} \sum_{|i-k| \leq 1} \Cov(R_i^2,R_k^2) \left(\sum_{|j-i| < 3} (nK_n^h(i-j)-1) U_j\right) \left(\sum_{l} U_l(nK_n^h(k-l) - 1)\right) \nonumber\\
&\lesssim \frac{4}{n^4} \sum_{|i-k| \leq 1} |\Cov(R_i^2,R_k^2)| \left(\sum_{|j-i| < 3} |nK_n^h(i-j)-1| \right) \left(\sum_{l} |nK_n^h(k-l) - 1|\right) \nonumber\\
&\lesssim  \frac{1}{n^4} \sum_{|i-k| \leq 1} |\Cov(R_i^2,R_k^2)| (1+h^{-1})n \nonumber\\
&\lesssim \frac{1}{n^2} + \frac{1}{n^2h}. \label{eqn: B_var bound 1}
\end{align}
A similar argument shows 
\begin{align}
  \frac{4}{n^4} \sum_{|i-k| \leq 1} \Cov(R_i^2, R_k^2)\left(\sum_{|j-i| \geq 3} (nK_n^h(i-j)-1) U_j\left( \sum_{\substack{|l-k| < 2, \\ or\ l \in G_{ij}}}U_l(nK_n^h(k-l) - 1)\right)\right) \lesssim \frac{1}{n^2 h} + \frac{1}{n^2}.\label{eqn: B_var bound 2}
\end{align}
Therefore, from (\ref{eqn:B_var_sep_split}) we have
\begin{align}
&\frac{2}{n^4} \sum_{|i-j| \geq 3} \sum_{\substack{|k-l| \geq 2 \\ k \in G_{ij}\\ l \not \in G_{ij}}} \Cov\left(R_i^2R_j^2, R_k^2R_l^2\right) \left(nK_n^{h}(i-j) - 1\right)\left(n K_n^{h}(k-l) - 1\right) \nonumber \\
&\leq C\left(\frac{1}{n^2h} + \frac{1}{n^2}\right) + \frac{4}{n^4} \sum_{|i-k| \leq 1} \Cov(R_i^2, R_k^2) \left( \left(\sum_{j} (nK_n^h(i-j)-1) U_j\right) \left(\sum_{l} U_l(nK_n^h(k-l) - 1)\right) \right) \nonumber\\
&= C\left(\frac{1}{n^2h} + \frac{1}{n^2}\right) + S_{-1} + S_0 + S_1
\label{eqn:B_var_sep_setup}
\end{align}
where \(C > 0\) is a universal constant and for \(m \in \{-1, 0, 1\}\) 
\begin{equation*}
S_m = \frac{4}{n^4}\sum_{i} \Cov\left(R_i^2, R_{i+m}^2\right) \left( \left(\sum_{j} (nK_n^h(i-j)-1) U_j\right) \left(\sum_{l} U_l(nK_n^h(i+m-l) - 1)\right) \right)
\end{equation*}
with the convention that a summand is taken to be zero if \(i+m \not\in\{0,...,n-1\}\). By Cauchy-Schwarz and the fact \(f, V\) are bounded, we have 
\begin{align}
S_m &\leq \sqrt{\frac{4}{n^4}\sum_i |\Cov(R_i^2, R_{i+m}^2)| \left(\sum_{j} U_j(nK_n^h(i-j) - 1)\right)^2} \nonumber \\
&\cdot \sqrt{\frac{4}{n^4}\sum_i |\Cov(R_i^2, R_{i+m}^2)| \left(\sum_{l} U_l(nK_n^h(i+m-l) - 1)\right)^2} \nonumber \\
&\lesssim \sqrt{\frac{1}{n^4} \sum_i \left(\sum_j U_j(nK_n^h(i-j) - 1)\right)^2} \sqrt{\frac{1}{n^4}\sum_i \left(\sum_{l} U_l(nK_n^h(i+m-l) - 1)\right)^2} \nonumber\\
&\lesssim \frac{1}{n^4} \sum_i \left(\sum_j U_j(nK_n^h(i-j) - 1)\right)^2 \nonumber \\
&\lesssim \frac{1}{n^4} \sum_i \left(\sum_j \delta_j^2 (nK_n^h(i-j)-1)\right)^2 + \frac{1}{n^4}\sum_i\left(\sum_j W_j (nK_n^h(i-j)-1)\right)^2 .\label{eqn: B_var S_m split}
\end{align}
Since \(f \in \mathcal{H}_\alpha\), we have \(|\delta_j| \leq n^{-(\alpha\wedge 1)}\) for any $\alpha>0$. In addition, using (\ref{eqn: bound on sum of kernel}), we can conclude that the first term in equation (\ref{eqn: B_var S_m split}) is bounded by
\begin{align}
 \frac{1}{n^4} \sum_i \left(\sum_j \delta_j^2 (nK_n^h(i-j)-1)\right)^2  
\lesssim \frac{n^{-4(\alpha \wedge 1)}}{n^4} \sum_i \left(\sum_j |nK_n^h(i-j) - 1|\right)^2  
\lesssim \frac{n^{-4(\alpha \wedge 1)}}{n}.  \label{eqn: B_var bound 3}
\end{align}
The second term in (\ref{eqn: B_var S_m split}) can be bounded with
\begin{align}
\frac{1}{n^4}\sum_i\left(\sum_j W_j (nK_n^h(i-j)-1)\right)^2 &= \frac{1}{n^2} \sum_i \left(\left(\sum_j W_j K_n^h(i-j)\right) - W_i + W_i - \bar{W}_n\right)^2 \nonumber \\
&\lesssim \frac{1}{n^2} \sum_i \left(\sum_{j}K_n^h(i-j)W_j - W_i\right)^2 +  \frac{||W - \bar{W}_n\mathbf{1}_n||^2}{n^2},  \label{eqn: B_var bound 4}
\end{align}
and following the same argument as (\ref{eqn: bias split second term})-(\ref{eqn: bias split second term kernel}) we have,
\begin{align}
 \frac{1}{n^2} \sum_i \left(\sum_{j}K_n^h(i-j)W_j - W_i\right)^2
 &\lesssim \frac{1}{n^2} \sum_i \left(\sum_{j}|K_n^h(i-j)||(i-j)/n|^{\beta})\right)^2 + \frac{h}{n}\nonumber\\
 &\lesssim \frac{h^{2\beta}}{n} + \frac{h}{n} .\label{eqn: B_var bound 5}
\end{align}
Plugging (\ref{eqn: B_var bound 5}) into (\ref{eqn: B_var bound 4}) and then (\ref{eqn: B_var bound 4}) and (\ref{eqn: B_var bound 3}) into (\ref{eqn: B_var S_m split}) gives us that, for $m\in \{-1, 0, 1\}$,
\begin{equation}
    S_m \lesssim \frac{n^{-4(\alpha \wedge 1)} + h^{2\beta} + h}{n} +\frac{||W - \bar{W}_n\mathbf{1}_n||_2^2}{n^2}\lesssim n^{-8(\alpha \wedge 1)} + h^{4\beta} + h^2 + \frac{1}{n^2} + \frac{||W - \bar{W}_n\mathbf{1}_n||_2^2}{n^2},\label{eqn: B_var S_m bound}
\end{equation}
where the last inequality follows from the identity, $ab\lesssim a^2 + b^2$ for reals $a$ and $b$. Plugging these (\ref{eqn: B_var S_m bound}) into (\ref{eqn:B_var_sep_setup}) and then into (\ref{eqn:B_var_close}), and (\ref{eqn: B_var sep bound}) into (\ref{eqn:B_var_sep}) gives us the following bound for (\ref{eqn:B_var}) 
\begin{align}
&\frac{2}{n^4} \sum_{|i-j| \geq 3} \sum_{\substack{|k-l| \geq 2 \\ k \in G_{ij}\\ l \not \in G_{ij}}} \Cov\left(R_i^2R_j^2, R_k^2R_l^2\right) \left(nK_n^{h}(i-j) - 1\right)\left(n K_n^{h}(k-l) - 1\right) \nonumber \\
&\lesssim \frac{1}{n^2h} + \frac{1}{n^2} + n^{-8(\alpha \wedge 1)} + h^{4\beta} +h^2  +\frac{||W - \bar{W}_n\mathbf{1}_n||_2^2}{n^2}.\label{eqn:B_var_sep_bound}
\end{align}
We have the desired bound for this term. Therefore, plugging the bounds for (\ref{eqn:G_var}) and (\ref{eqn:B_var})
yields the following bound for variance
\begin{align}
 \Var\left(\hat{T}\right)
&\lesssim \frac{1}{n^2h} + \frac{1}{n^2} + n^{-8(\alpha \wedge 1)} + h^{4\beta} + h^2  + \frac{||W - \bar{W}_n\mathbf{1}_n||_2^2}{n^2}\label{eqn: variance bound}. 
\end{align}
\newline
\noindent \textbf{Synthesis: }\newline 
Finally, putting together the bias bound from (\ref{eqn: bias bound}) and variance bound from (\ref{eqn: variance bound}) yields the desired bound,
\begin{align*}
 E\left(\left|\hat{T} - T\right|^2\right) 
 &\lesssim  n^{-8(\alpha \wedge 1)} + h^{4\beta} + h^4 + h^2 + \frac{1}{n^2h} + \frac{1}{n^2}   +\frac{||W - \bar{W}_n\mathbf{1}_n||_2^2}{n^2}\\
 & \lesssim   n^{-8(\alpha \wedge 1)} + h^{4\beta} + h^2   +\frac{1}{n^2h} + \frac{||W - \bar{W}_n\mathbf{1}_n||_2^2}{n^2},
\end{align*}
where the last inequality holds because $h<1$.  
\end{proof}

\begin{proof}[Proof of Theorem \ref{theorem: MSE of T_hat}]
We will utilize Propositions \ref{prop: MSE of T_hat beta>1/4} here. To apply Proposition \ref{prop: MSE of T_hat beta>1/4}, we need to satisfy condition (\ref{eqn: kernel integral condition}). Given that \(K\) is bounded, if we select \(h = C_hn^{-\left(\frac{2}{4\beta+1} \wedge 1\right)}\), then
\begin{equation*}
\left|\int_{-2/n}^{2/n} \frac{1}{h}K\left(\frac{u}{h}\right) du\right|  \leq \frac{4||K||_\infty}{nh} = \frac{4||K||_\infty}{C_h}n^{-\left(\frac{4\beta - 1}{4\beta+1} \vee 0\right)}
\end{equation*}
which can be made small for sufficiently large \(C_h\) when \(\beta \leq \frac{1}{4}\) and converges to \(0\) when \(\beta >\frac{1}{4}\). Hence, we can find a universal positive constant \(c\) to satisfy the condition (\ref{eqn: kernel integral condition}) of Proposition \ref{prop: MSE of T_hat beta>1/4}. Consequently, by using Propositions \ref{prop: MSE of T_hat beta>1/4}, we can derive for any \(0 < \beta < \frac{1}{2}\),
\begin{equation}
 E_{f,V}\left(\left|\hat{T} - T\right|^2\right) \lesssim n^{-8(\alpha \wedge 1)} + n^{-\frac{8\beta}{4\beta+1}} + n^{-4\beta} + \frac{||W - \bar{W}_n\mathbf{1}_n||_2^2}{n^2}. \label{MSE: expression in proof}
\end{equation}
Using that $V\in \mathcal{H}_\beta$, we have
\begin{align}
\frac{||W - \bar{W}_n\mathbf{1}_n||_2^2}{n^2}
&= \frac{\sum_i  \left(W_i -  \bar{W}_n\right)^2}{n^2}\nonumber\\
&= \frac{\sum_i \int_{i/n}^{(i+1)/n}\left(W_i - 2V(x)+ 2V(x) - 2\bar{V} +2\bar{V}- \bar{W}_n\right)^2dx}{n}\nonumber\\
&\lesssim \frac{\sum_i \int_{i/n}^{(i+1)/n}\left(W_i - 2V(x)\right)^2dx}{n} + \frac{||V-\bar{V}||_2^2}{n} + \frac{\sum_i \int_{i/n}^{(i+1)/n}\left(2\bar{V}- \bar{W}_n\right)^2dx}{n}\nonumber\\
&\lesssim \frac{n^{-2(\beta \wedge 1)} }{n} +  \frac{||V-\bar{V}\mathbf{1}||_2^2}{n}\nonumber\\
&\lesssim n^{-4(\beta\wedge 1)} +\frac{1}{n^2} +  \frac{||V-\bar{V}\mathbf{1}||_2^2}{n}.\label{eqn: W-W_bar less than V-V_bar}
\end{align}
Plugging (\ref{eqn: W-W_bar less than V-V_bar}) into (\ref{MSE: expression in proof}) gives us the desired bound
\begin{equation*}
E_{f,V}\left(\left|\hat{T} - T\right|^2\right) \lesssim n^{-8\alpha} + n^{-\frac{8\beta}{4\beta+1}} + n^{-4\beta} + \frac{||V-\bar{V}\mathbf{1}||_2^2}{n}
\end{equation*}
where $\alpha\wedge1$ and $\beta\wedge1$ are replaced by $\alpha$ and $\beta$ respectively and $n^{-2}$ is disappears as it is dominated $n^{-2}\lesssim n^{-\frac{8\beta}{4\beta+1}}$ for all $\beta>0$.
\end{proof}

\subsection{Optimal Test for Heteroskedasticity}

Here, we present the proof of Theorem \ref{theorem: optimal test beta>1/4}.

\begin{proof}[Proof of Theorem \ref{theorem: optimal test beta>1/4} ]
Fix \(\eta \in (0, 1)\) and let \(C > C_\eta\) where \(C_\eta\) is to be set. In the course of the proof, we will note where \(C_\eta, C_\eta'\) needs to be taken sufficiently large depending on \(\eta\). Let us first examine the Type I error. Fix \(f \in \mathcal{H}_\alpha\) and \(V \in \mathcal{V}_0\). For ease of notation, we will simply write \(P\) instead of \(P_{f,V}\) as the context is clear. By Theorem \ref{theorem: MSE of T_hat}, when $h\asymp n^{-\left(\frac{2}{4\beta+1} \wedge 1\right)}$, for some universal constant $L>0$, we have
\begin{equation*}
E(\hat{T}^2) \leq  L\left(n^{-8\alpha} + n^{-4\beta} + n^{-\frac{8\beta}{4\beta+1}}\right)\leq L\zeta^4
\end{equation*}
where we have used \(V \in \mathcal{V}_0\) implies \(V=\bar{V}\) and  \(W_i = \bar{W}_n\) for all \(0 \leq i \leq n-1\) and so \(T = \frac{1}{n} \sum_{i=0}^{n-1} \left(\delta_i^2 - \overline{\delta_{n}^2}\right)^2 \lesssim n^{-4(\alpha \wedge 1)}\). Therefore, by Chebyshev's inequality, we have 
\begin{equation*}
P\left\{ \hat{T} > C_\eta' \zeta^2\right\} \leq \frac{L\zeta^4}{C_\eta'^2\zeta^4}  \leq \frac{\eta}{2}
\end{equation*}
by taking \(C_\eta'\) sufficiently large. Taking supremum over \(f \in \mathcal{H}_\alpha\) and \(V \in \mathcal{V}_0\) shows that the Type I error is suitably bounded. Let us now examine the Type II error. Fix \(f \in \mathcal{H}_\alpha\) and \(V \in \mathcal{V}_{1, \beta}(C\zeta)\). We have by Theorem \ref{theorem: MSE of T_hat}, for some universal constant $L'>0$,
\begin{align*}
E\left(\left|\hat{T} - T\right|^2\right) \leq  L'\left(\zeta^4 + \frac{||V-\bar{V}\mathbf{1}||_2^2}{n}\right).
\end{align*}
Now, since $V \in \mathcal{V}_{1, \beta}(C\zeta)$, Proposition \ref{prop: V-V bar less than T} tells us that, for some universal constant $L''>0$, we have $C^2\zeta^2\leq ||V - \bar{V}\mathbf{1}||_2^2 \leq L''(T + \zeta^2)$. Choosing \(C_\eta\) sufficiently large will ensure that \( 2C_\eta'\zeta^2 \leq \frac{C^2}{2L''}\zeta^2 \leq T \). Using this inequality and Chebyshev's inequality gives us,
\begin{align*}
P\left\{\hat{T} \leq C_\eta' \zeta^2\right\} 
&= P\left\{T - C_\eta' \zeta^2 \leq T - \hat{T}\right\} \\
&\leq \frac{E(|T - \hat{T}|^2)}{\left(T - C_\eta' \zeta^2\right)^2} \\
&\leq \frac{L' \zeta^4 + L'  n^{-1}||V-\bar{V}\mathbf{1}||_2^2}{\left(T - C_\eta' \zeta^2\right)^2} \\
 &\leq \frac{L' \zeta^4 + L'  n^{-1}||V-\bar{V}\mathbf{1}||_2^2}{(T/2)^2} \\
&\leq \frac{L' \zeta^4 + L'L''  \zeta^2(T + \zeta^2)}{(T/2)^2} \\
&\leq \frac{4L'(1 + L'')\zeta^4}{T^2}  + \frac{4L'L''  \zeta^2}{T} \\
&\leq \frac{16L'L''^2(1 + L'')}{C^4}  + \frac{8L'L''^2 }{C^2} \\
&\leq \frac{\eta}{2}
\end{align*}
when $C_\eta$ is large enough. Taking supremum over \(f\) and \(V\) yields the desired bound on Type II error. The proof is complete. 

\end{proof}


\section{Lower bounds} \label{appendix:lower_bounds}
This section contains proofs of the lower bounds presented in Section \ref{section:lower_bounds}.

\subsection{\texorpdfstring{Proving the \(n^{-4\alpha}\) bound}{Proving the nuisance rate}}
Proposition \ref{prop:nuisance_lbound} is proved here. Recall from the discussion in Section \ref{section:lower_bounds} that the proof is, in spirit, the same as the moment matching argument of \cite{wang_effect_2008}. However, the argument requires modification as the construction in \cite{wang_effect_2008} results in homoskedastic models under both null and alternative hypotheses. Furthermore, the argument is streamlined by using the moment matching \(\chi^2\)-divergence bound proved in \cite{wu_polynomial_2020}.

\begin{proof}[Proof of Proposition \ref{prop:nuisance_lbound}]
Fix \(\eta \in (0, 1)\) and let \(0 < c < c_\eta\) where \(c_\eta\) is to be selected later. Define the function \(h(t) = e^{-1/t}\mathbbm{1}_{\{t > 0\}}\) and note \(h \in C^\infty(\R)\). Further define the function \(\psi(t) = \frac{h(t)}{h(t) + h(1-t)}\) and note \(\psi \in C^\infty(\R)\), \(\psi(t) = 1\) for \(t \geq 1\) and \(\psi(t) = 0\) for \(t \leq 0\). Define the variance function \(V_1 : [0, 1] \to \R\), 
\begin{equation*}
V_1(x) = 1 + 2cn^{-2\alpha} \psi\left(n^{\frac{2\alpha}{\lceil \beta \rceil}} \left(x - \left(\frac{1}{2} - \frac{1}{2}n^{-\frac{2\alpha}{\lceil \beta \rceil}}\right)\right)\right).
\end{equation*}

\begin{figure}[ht]
\centering
\begin{tikzpicture}[scale=0.9, declare function={
    func(\x)= (\x < 1/2-(1/(6))) * (1)   +
            and(\x >= 1/2-(1/(6)), \x < 1/2+(1/(6))) * (1+exp(-1/(3*(x - (1/2-1/6))))/(exp(-1/(3*(x - (1/2-1/6)))) + exp(-1/(1-3*(x - (1/2-1/6))))))    +
            (\x >= 1/2+(1/(6))) * (2)
;
}]
    \begin{axis}[
        axis x line=middle, axis y line=middle,
        ymin=0.7, ymax=2.3, ytick={1, 1.5, 2}, yticklabels={\(1\), \(1+cn^{-2\alpha}\), \(1+2cn^{-2\alpha}\)}, ylabel=$V_1(x)$,
        xmin=0, xmax=1, xtick={0, 1/2-1/6, 1/2+1/6, 1}, xticklabels={\(0\),\(\frac{1}{2} - \frac{1}{2}n^{-\frac{2\alpha}{\lceil\beta\rceil}}\), \(\frac{1}{2} + \frac{1}{2}n^{-\frac{2\alpha}{\lceil\beta\rceil}}\),\(1\)}, extra x ticks = {0}, xlabel=$x$,
        domain=0:1,samples=101, 
        every major tick/.append style={very thick, black},
    ]
    \draw[dashed](0, 1.5) -- (1,1.5);

    \addplot [blue, thick]{func(x)};
    \end{axis}
\end{tikzpicture}
\caption{A cartoon plot of \(V_1(x)\).}
\end{figure}

\noindent Consider that for all \(k \in \{0,...,\lceil \beta \rceil\}\), we have
\begin{align*}
|V_1^{(k)}(x)| = 2cn^{-2\alpha} n^{\frac{2\alpha k}{\lceil \beta \rceil}} \psi^{(k)}\left(n^{\frac{2\alpha}{\lceil \beta \rceil}}\left( x - \left(\frac{1}{2} - \frac{1}{2} n^{-\frac{2\alpha}{\lceil \beta \rceil}}\right)\right)\right) \leq 2c ||\psi^{(k)}||_\infty
\end{align*}
for all \(x \in [0, 1]\). Therefore, taking \(c_\eta\) sufficiently small (depending only on \(M\) and \(\beta\)) guarantees \(||V_1^{(k)}||_\infty \leq M\) for all \(k \in \{0,...,\lceil \beta \rceil\}\), which implies \(V_1 \in \mathcal{H}_\beta\). Note that we have used mean-value theorem to obtain \(|V^{\lfloor \beta \rfloor}(x) - V^{\lfloor \beta \rfloor}(y)| \leq ||V_1^{\lceil \beta \rceil}||_\infty |x-y| \leq M |x - y|^{\beta - \lceil \beta \rceil }\) since \(|x - y| \leq 1\). Further note \(V_1 \geq 0\) and \(\int_{0}^{1} V_1(x) \, dx = 1 + cn^{-2\alpha}\) by symmetry. Consider that 
\begin{equation*}
||V_1 - \bar{V}_1\mathbf{1}||_2^2 = \int_{0}^{1} (V_1(x) - 1 - cn^{-2\alpha})^2 \, dx \geq c^2n^{-4\alpha} \cdot 2 \left(\frac{1}{2} - \frac{1}{2} n^{-\frac{2\alpha}{\lceil\beta\rceil}}\right) = c^2n^{-4\alpha}(1 - n^{-\frac{2\alpha}{\lceil\beta\rceil}}) \geq \frac{1}{2}c^2n^{-4\alpha}
\end{equation*}
where we can insist on \(n\) being sufficiently large such that \(1 - n^{-\frac{2\alpha}{\lceil\beta\rceil}} \geq \frac{1}{2}\). Therefore, \(V_1 \in \mathcal{V}_{1,\beta}(cn^{-2\alpha}/\sqrt{2})\).

We will now define a prior \(\pi\) on \(\mathcal{H}_\alpha\) using the moment matching idea of \cite{wang_effect_2008}. Let \(q\) be a fixed integer such that \(2\alpha(q+1) > 1\). Let \(g : [0, 1] \to \R\) denote the function \(g(x) = (1 - 2n|x|)\mathbbm{1}_{\left\{|x| \leq \frac{1}{2n}\right\}}\). Let \(G\) denote the symmetric distribution supported on \([-B, B]\), possessing the same first \(q\) moments as the standard normal distribution asserted to exist by Lemma \ref{lemma:moment_matching}.  A draw \(f \sim \pi\) is obtained by drawing \(R_0,...,R_n \overset{iid}{\sim} G\) and setting 
\begin{equation*}
f(x) = \sum_{i=0}^{n} R_i \left(\sqrt{V_1\left(\frac{i}{n}\right) - 1}\right) g\left(x - \frac{i}{n}\right). 
\end{equation*}
Since \(\sqrt{V_1\left(\frac{i}{n}\right) - 1} \lesssim cn^{-\alpha}\), as argued in \cite{wang_effect_2008} we have \(\pi\) is supported on \(\mathcal{H}_\alpha\) (since \(c_\eta\) can be chosen sufficiently small independent of \(n\)). Therefore, we have 
\begin{align}
&\inf_{\varphi}\left\{ \sup_{\substack{f\in \mathcal{H}_\alpha, \\ V \in \mathcal{V}_{0}}} P_{f, V}\left\{ \varphi = 1 \right\} + \sup_{\substack{f \in \mathcal{H}_\alpha, \\ V \in \mathcal{V}_{1, \beta}(cn^{-2\alpha})}} P_{f, V}\left\{\varphi = 0\right\} \right\} \nonumber \\
&\geq \inf_{\varphi}\left\{ P_{\pi, \mathbf{1}}\left\{ \varphi = 1 \right\} + P_{0, V_1}\left\{\varphi = 0\right\} \right\} \nonumber \\
&= 1 - d_{TV}(P_{\pi, \mathbf{1}}, P_{0, V_1}) \nonumber \\
&\geq 1-\frac{1}{2}\sqrt{\chi^2(P_{\pi, \mathbf{1}}, P_{0, V_1})} \label{eqn:chisquare_lowerbound}
\end{align}
where \(P_{\pi, \mathbf{1}} = \int P_{f, \mathbf{1}} \pi(df)\) denotes the mixture distribution induced by \(\pi\) and \(\mathbf{1}\) denotes the constant function on \([0, 1]\) equal to one. 

Consider under \(P_{\pi, \mathbf{1}}\), we have \(Y_i \,| \, R_i \overset{ind}{\sim} N\left(R_i \sqrt{V_1\left(\frac{i}{n}\right) - 1}, 1\right)\) since \(g(0) = 1\) and \(g\) is supported on \(\left[-\frac{1}{2n}, \frac{1}{2n}\right]\). Therefore, the testing problem associated to (\ref{eqn:chisquare_lowerbound}), 
\begin{align*}
H_0 &: (Y_0,...,Y_n) \sim P_{\pi, \mathbf{1}}, \\
H_1 &: (Y_0,...,Y_n) \sim P_{0, V_1}
\end{align*}
can be written as
\begin{align*}
H_0 &: Y_i \overset{ind}{\sim} \int N\left(r \sqrt{V_1(i/n) - 1}, 1\right) \, G(dr) = \nu_{0,i} * N(0, 1), \\
H_1 &: Y_i \overset{ind}{\sim} N(0, V_1(i/n)) = \nu_{1, i} * N(0, 1)
\end{align*}
where \(\nu_{0, i}\) is the distribution of \( R \sqrt{V_1(i/n) - 1}\) for \(R \sim G\) and \(\nu_{1, i} = N(0, V_1(i/n) - 1)\). By the choice of \(G\), it follows that \(\nu_{0,i}\) and \(\nu_{1,i}\) share the first \(q\) moments. Further note \(\nu_{1,i}\) is \(\sqrt{2c}n^{-\alpha}\)-subgaussian (see Definition \ref{def:subgaussian}). Likewise, consider by Lemma \ref{lemma:bounded_subgaussian} that \(\nu_{0,i}\) is \(B\sqrt{2c}n^{-\alpha}\)-subgaussian. Without loss of generality, we can assume \(B \geq 1\) and so \(\nu_{1,i}\) is also \(B\sqrt{2c}n^{-\alpha}\)-subgaussian. By Proposition \ref{prop:chisquare_moment_matching} and Lemma \ref{lemma:chisquare_tensorization}, we have 
\begin{align*}   
\chi^2\left(\bigotimes_{i=0}^{n} \left(\nu_{1,i} * N(0,1)\right), \bigotimes_{i=0}^{n} \left(\nu_{0, i} * N(0, 1)\right)\right) &= \left(\prod_{i=0}^{n} (\chi^2(\nu_{1,i} * N(0,1) , \nu_{0, i} * N(0, 1)) + 1)\right) - 1\\
&\leq \left(1 + \frac{16}{\sqrt{q}} \frac{(2c)^{q+1}B^{2q+2}n^{-2\alpha(q+1)}}{1- 2B^2cn^{-2\alpha}} \right)^{n+1} - 1\\
& \leq \left(1 + C c^{q+1} n^{-2\alpha(q+1)}\right)^{n+1} - 1
\end{align*}
where \(C > 0\) is some constant not depending on \(n\). Since \(2\alpha(q+1) > 1\), it follows that 
\begin{align*}
\chi^2\left(\bigotimes_{i=0}^{n} \left(\nu_{1,i} * N(0,1)\right), \bigotimes_{i=0}^{n} \left(\nu_{0, i} * N(0, 1)\right)\right) 
&\leq \left(1 + \frac{Cc^{q+1}n^{1-2\alpha(q+1)}}{n+1}\right)^{n+1} - 1\\
&\leq \left(1 + \frac{Cc^{q+1}}{n+1}\right)^{n+1} - 1\\
&\leq e^{Cc_\eta^{q+1}} - 1
\end{align*}
where the value of \(C > 0\) may have changed but remains a universal constant. Taking \(c_\eta \leq \left(C^{-1}\log\left(1 + 4\eta^2\right)\right)^{\frac{1}{q+1}}\) sufficiently small as noted above, we have from (\ref{eqn:chisquare_lowerbound})
\begin{equation*}
\inf_{\varphi}\left\{ \sup_{\substack{f\in \mathcal{H}_\alpha, \\ V \in \mathcal{V}_{0}}} P_{f, V}\left\{ \varphi = 1 \right\} + \sup_{\substack{f \in \mathcal{H}_\alpha, \\ V \in \mathcal{V}_{1, \beta}(cn^{-2\alpha})}} P_{f, V}\left\{\varphi = 0\right\} \right\} \geq 1-\eta. 
\end{equation*}
The proof is complete. 
\end{proof}

\subsection{\texorpdfstring{Proving the \(n^{-\frac{4\beta}{4\beta+1}}\) bound}{Proving the nonparametric rate}}
The proof of Proposition \ref{prop:nonparam_lbound} is presented here. As discussed in Section \ref{section:lower_bounds}, the term \(n^{-\frac{4\beta}{4\beta+1}}\) has appeared in other nonparametric testing problems \cite{ingster_nonparametric_2003,gine_mathematical_2016}. While the prior construction is standard (Proposition \ref{prop:prior_construction}), bounding the resulting \(\chi^2\)-divergence is more complicated than earlier work as the alternative involves heteroskedasticity. Recall, as discussed in Section \ref{section:lower_bounds}, it can be assumed without loss of generality \(\beta > \frac{1}{4}\).

\begin{proof}[Proof of Proposition \ref{prop:nonparam_lbound}]
Fix \(\eta \in (0, 1)\) and let \(0 < c < c_\eta\) where \(c_\eta \leq 1 \) will be chosen later. Set \(m = \left\lfloor n^{\frac{2}{4\beta+1}}\right\rfloor\) and fix \(\psi \in C^{\infty}(\R)\) supported on \([0, 1]\) with \(||\psi||_2 = 1\) and \(\int \psi = 0\). Note \(m \leq n\) since \(\beta > \frac{1}{4}\). Let \(\rho = cc' n^{-\frac{2\beta + 1}{4\beta+1}}\) for a sufficiently small positive constant \(c' > 0\). For any \(\kappa \in \left\{-1, 1\right\}^m\), let \(V_\kappa\) denote the function given in Proposition \ref{prop:prior_construction}. Note that by taking \(c'\) sufficiently small (not depending on \(n\)), we have for all \(q \in \{0,1,...,\lfloor \beta \rfloor\}\),
\begin{equation*}
\rho m^{\frac{1}{2}+q} ||\psi^{(q)}||_\infty \leq c' n^{\frac{-2\beta+2q}{4\beta+1}} ||\psi^{(q)}||_\infty \leq \frac{1}{2} < 1,
\end{equation*}
and 
\begin{equation*}
\rho m^{\frac{1}{2}+\beta} \left(4||\psi^{(\lfloor \beta \rfloor)}||_\infty \vee 2 ||\psi^{(\lfloor \beta \rfloor + 1)}||_\infty \right) \leq c'  \left(4||\psi^{(\lfloor \beta \rfloor)}||_\infty \vee 2 ||\psi^{(\lfloor \beta \rfloor + 1)}||_\infty \right)  \leq 1. 
\end{equation*}
Hence, Proposition \ref{prop:prior_construction} asserts \(V_\kappa \in \mathcal{V}_{1, \beta}\left(m^{\frac{1}{2}}\rho\right)\). Let \(P_\pi\) denote the prior distribution on \(\mathcal{V}_{1, \beta}\left(m^{\frac{1}{2}}\rho\right)\) where a draw \(V \sim P_\pi\) is obtained by drawing \(\kappa \sim \Uniform\left(\{-1, 1\}^m\right)\) and setting \(V = V_\kappa\). With this prior in hand, we have 
\begin{align*}
&\inf_{\varphi}\left\{ \sup_{\substack{f \in \mathcal{H}_\alpha, \\ V \in \mathcal{V}_{0}}} P_{f, V}\left\{\varphi = 1\right\} + \sup_{\substack{f \in \mathcal{H}_\alpha, \\ V \in \mathcal{V}_{1,\beta}\left(\rho m^{1/2}\right)}} P_{f, V}\left\{\varphi = 0\right\} \right\} \\
&\geq \inf_{\varphi}\left\{ P_{0, \mathbf{1}}\left\{\varphi = 1\right\} + P_{0, \pi} \left\{\varphi = 0\right\} \right\} \\
&= 1 - d_{TV}(P_{0, \mathbf{1}}, P_{0, \pi}) \\
&\geq 1- \frac{1}{2}\sqrt{\chi^2(P_{0, \pi}, P_{0,\mathbf{1}})}
\end{align*}
where \(\mathbf{1}\) denotes the constant function equal to \(1\) on the unit interval and \(P_{0, \pi} = \int P_{0, V} \, \pi(dV)\) denotes the mixture over \(P_{0,V}\) induced by drawing \(V \sim \pi\). The penultimate line follows from the Neyman-Pearson lemma. We now work towards bounding the \(\chi^2\)-divergence. 

First, note that \(\kappa_1,...,\kappa_m \overset{iid}{\sim} \Rademacher(1/2)\). Consider that \(\{\psi_j\}_{j=1}^{m}\) in the definition of \(V_\kappa\) have disjoint support, namely \(\psi_j \in C^\infty(\R)\) and \(\psi_j\) is supported on \(\left[\frac{j-1}{m}, \frac{j}{m}\right]\). Therefore, \(V_\kappa(0) = 1 + \rho \kappa_1\psi_1(0)\). Since \(\psi_1\) is supported on \([0, 1/m]\) and is a \(C^\infty(\R)\) function, we have \(\psi_1(0) = 0\) and so \(V(0) = 1 + \rho \kappa_1 \psi_1(0) = 1\) almost surely. Therefore, \(Y_0 \,|\, \kappa \sim N(0, 1)\) for all \(\kappa\). In particular, under the prior \(\pi\), we have that \(Y_0\) is independent of \(Y_1,...,Y_n\). Moreover, \(Y_0 \sim N(0, 1)\) under both the null and the alternative hypotheses, and thus contains no distinguishing information. Hence, it can be thrown away and we can focus attention solely on \(Y_1,...,Y_n\). Without loss of generality, assume \(m\) divides \(n\) and set \(K = \frac{m}{n}\). Since the \(\{\psi_j\}_{j=1}^{m}\) in the definition of \(V_\kappa\) have disjoint support, it follows that we have the mutual independence of random vectors 
\begin{equation*}
(Y_1,...,Y_K) \indep (Y_{K+1},...,Y_{2K}) \indep ... \indep (Y_{n-K+1},...,Y_{n}). 
\end{equation*}
Further note that the marginal distributions of these \(m\) random vectors are all the same. Let us call the marginal distribution of the random vector \((Y_1,...,Y_K)\) as \(Q_0\) and \(Q_\pi\) induced by \(P_{0, \mathbf{1}}\) and \(P_{0, \pi}\) respectively. By Lemma \ref{lemma:chisquare_tensorization}, 
\begin{equation}\label{eqn:chisquare_tensorization}
\chi^2(P_{0,\pi}, P_{0,\mathbf{1}}) = \left((1 + \chi^2(Q_\pi, Q_0))^m \right) - 1. 
\end{equation}
Consider that we can write the problem of testing \(Q_\pi\) against \(Q_0\) as 
\begin{align*}
H_0 &: Y_1,...,Y_K \overset{iid}{\sim} N(0, 1), \\
H_1 &: R \sim \Rademacher\left(\frac{1}{2}\right) \text{ and } Y_k | R \overset{ind}{\sim} N(0, 1 + R\rho\psi_1(k/n)) \text{ for } 1 \leq k \leq K. 
\end{align*}
Consider that \(Y_1^2,...,Y_K^2\) are sufficient statistics, and so letting \(W_k = Y_k^2\) we can equivalently consider the problem 
\begin{align*}
H_0 &: W_1,...,W_K \overset{iid}{\sim} \chi^2_1, \\
H_1 &: R \sim \Rademacher(1/2) \text{ and } W_k | R \overset{ind}{\sim} (1 + R \rho \psi_1(k/n)) \chi^2_1 \text{ for } 1 \leq k \leq K. 
\end{align*}
Let \(\tilde{Q}_0\) and \(\tilde{Q}_\pi\) denote the marginal distribution of \((W_1,...,W_K)\) under the null and alternative respectively, and note \(\chi^2(Q_\pi, Q_0) = \chi^2(\tilde{Q}_\pi, \tilde{Q}_0)\). For \(r \in \{\pm 1\}\), denote 
\begin{equation*}
\tilde{Q}_r = \bigotimes_{k=1}^{K} (1+\rho r \psi_1(k/n)) \chi^2_1
\end{equation*}
and let \(\tilde{q}_r(w_1,...,w_K) = \prod_{k=1}^{K} \tilde{q}_r(w_k)\) denote its probability density function. Note \(\tilde{Q}_\pi = \frac{1}{2} \tilde{Q}_{-1} + \frac{1}{2}\tilde{Q}_1\). Further note \(\tilde{Q}_0 = (\chi^2_1)^{\otimes K}\) so let us denote \(\tilde{q}_0 = \prod_{k=1}^{K} \tilde{q}_0(w_k)\) as its probability density function. Let \(\tilde{q}_{\pi}(w_1,...,w_k)\) denote the probability density function of \(\tilde{Q}_\pi\). By direct calculation, 
\begin{align}
\chi^2(\tilde{Q}_\pi, \tilde{Q}_0) + 1 &= \int_{\R^K} \frac{\tilde{q}_\pi(w)^2}{\tilde{q}_0(w)} \, dw \nonumber \\
&= \frac{1}{4} \sum_{r, r' \in \{\pm 1\}}\int_{\R^K} \frac{\tilde{q}_r(w)\tilde{q}_{r'}(w)}{\tilde{q}_0(w)} \, dw \nonumber \\
&= \frac{1}{4} \sum_{r,r' \in \{\pm1\}} \prod_{k=1}^{K} \int_{\R} \frac{\tilde{q}_r(w_k)\tilde{q}_{r'}(w_k)}{\tilde{q}_0(w_k)} \, dw_k \label{eqn:tilde_chisquare}. 
\end{align}
For ease of notation, let \(v_k := \rho \psi_1(k/n)\) for \(1 \leq k \leq K\). Note from the density of the \(\chi^2_1\) distribution, 
\begin{align*}
\tilde{q}_0(w_k) &= \frac{1}{\sqrt{2} \Gamma(1/2)} \frac{1}{\sqrt{w_k}} e^{-\frac{w_k}{2}} \mathbbm{1}_{\{w_k > 0\}}, \\
\tilde{q}_r(w_k) &= \frac{1}{1+ r v_k} \tilde{q}_0 \left(\frac{w_k}{1+ r v_k}\right) = \frac{1}{\sqrt{2} \Gamma(1/2)} \frac{1}{\sqrt{1+ r v_k}}  \frac{1}{\sqrt{w_k}} e^{-\frac{w_k}{2(1+ r v_k)}} \mathbbm{1}_{\{w_k > 0\}}. 
\end{align*}
From direct calculation, 
\begin{align*}
\frac{\tilde{q}_r(w_k)\tilde{q}_{r'}(w_k)}{\tilde{q}_0(w_k)} &= \frac{1}{\sqrt{2}\Gamma(1/2)} \frac{1}{\sqrt{w_k}} \cdot \frac{1}{\sqrt{(1+r v_k)(1+r' v_k)}} \exp\left(-\frac{w_k}{2(1+r v_k)} - \frac{w_k}{2(1+r' v_k)} + \frac{w_k}{2}\right) \mathbbm{1}_{\{w_k \geq 0\}} \\
&= \left(\frac{1}{\sqrt{2}\Gamma(1/2)} \frac{1}{\sqrt{w_k}} \cdot \frac{1}{\sqrt{(1+r v_k)(1+r' v_k)}} \exp\left(-\frac{w_k}{2(1+r v_k)(1+r'v_k)}\right) \mathbbm{1}_{\{w_k \geq 0\}} \right)\\
&\;\; \cdot \exp\left(\frac{w_k rr' v_k^2}{2(1+r v_k)(1+r' v_k)}\right) \\
&= \frac{1}{(1+r v_k)(1+r' v_k)} \tilde{q}_0\left(\frac{w_k}{(1+r v_k)(1+r' v_k)}\right) \cdot \exp\left(\frac{w_k rr' v_k^2}{2(1+r v_k)(1+r' v_k)}\right),
\end{align*}
and so
\begin{equation*}
\int_{\R} \frac{\tilde{q}_r(w_k) \tilde{q}_{r'}(w_k)}{\tilde{q}_0(w_k)} \, dw_k = E\left(\exp\left(\frac{rr'v_k^2}{2} \cdot U\right)\right)
\end{equation*}
where \(U \sim \chi^2_1\). Noting \(v_k = \rho \psi_1(k/n) \leq \rho m^{1/2} ||\psi||_\infty < 1\), it follows \(v_k^2 < 1\) and so the moment generating function admits the formula \(E\left(\exp\left(\frac{rr'v_k}{2} \cdot U\right)\right) = \frac{1}{\sqrt{1 - rr'v_k^2}}\). Note by taking \(c'\) sufficiently small, we have \(v_k^2 < \frac{1}{2}\), and so by the inequality \(\frac{1}{\sqrt{1-x}} \leq e^x\) which holds for \(x < \frac{1}{2}\), it follows 
\begin{equation*}
\int_{\R} \frac{\tilde{q}_r(w_k) \tilde{q}_{r'}(w_k)}{\tilde{q}_0(w_k)} \, dw_k \leq \exp\left(rr'v_k^2\right). 
\end{equation*}
Therefore, from (\ref{eqn:tilde_chisquare}) we have 
\begin{align}
\chi^2(\tilde{Q}_\pi, \tilde{Q}_0) + 1 &\leq \frac{1}{4} \sum_{r, r' \in \{\pm 1\}} \prod_{k=1}^{K} \exp\left(rr'v_k^2\right) \nonumber \\
&= \frac{1}{2} \exp\left(\sum_{k=1}^{K} v_k^2\right) + \frac{1}{2} \exp\left(-\sum_{k=1}^{K} v_k^2\right) \nonumber \\
&= \cosh\left(\sum_{k=1}^{K} v_k^2\right) \nonumber \\
&\leq \exp\left(\left(\sum_{k=1}^{K} v_k^2\right)^2\right) \label{eqn:tilde_chisquare_simple}
\end{align}
where we have used \(\cosh(x) \leq e^{x^2}\). To continue developing the bound, observe 
\begin{align*}
\sum_{k=1}^{K} v_k^2 &= n\rho^2 \cdot \frac{1}{n} \sum_{k=1}^{K} \psi_1(k/n)^2 \\
&= n\rho^2 \sum_{k=1}^{K} \int_{\frac{k-1}{n}}^{\frac{k}{n}} \psi_1(k/n)^2 \,dx \\
&\leq 2n\rho^2 + 2n\rho^2 \sum_{k=1}^{K} \int_{\frac{k-1}{n}}^{\frac{k}{n}} |\psi_1(k/n) - \psi_1(x)|^2 \, dx
\end{align*}
where we have used the inequality \((a + b)^2 \leq 2a^2 + 2b^2\) for \(a, b \in \R\) as well as \(||\psi_1||^2 = 1\). By definition of \(\psi_1\), consider \(|\psi_1(x) - \psi_1(y)| \leq ||\psi_1'||_\infty |x - y| \leq m^{3/2} ||\psi'||_\infty |x - y|\). Therefore, 
\begin{align*}
\sum_{k=1}^{K}v_k^2 &\leq 2n\rho^2 + 2n\rho^2 m^3 ||\psi'||_\infty^2 \sum_{k=1}^{K} \int_{\frac{k-1}{n}}^{\frac{k}{n}} \left|x - \frac{k}{n}\right|^2 \, dx \\
&= 2n\rho^2 + 2n\rho^2 m^3 ||\psi'||_\infty^2 \sum_{k=1}^{K} \int_{-\frac{1}{n}}^{0} u^2 \, du \\
&= 2n\rho^2 + 2n\rho^2 \cdot \frac{Km^3 ||\psi'||_\infty^2}{3n^3} \\
&\leq 2n\rho^2 + 2n\rho^2 \cdot \frac{nm^2||\psi'||_\infty^2}{3n^3} \\
&\leq 2n\rho^2 + 2n\rho^2 \cdot \frac{||\psi'||_\infty^2}{3} \\
&\leq C_1 n\rho^2
\end{align*}
where \(C_1 > 0\) is a universal constant. Here, we have used \(m \leq n\) since \(\beta > \frac{1}{4}\). Plugging into (\ref{eqn:tilde_chisquare_simple}) yields \(\chi^2\left(\tilde{Q}_\pi, \tilde{Q}_0\right) + 1 \leq \exp\left(C_1^2 n^2\rho^4\right)\). Therefore, (\ref{eqn:chisquare_tensorization}) gives 
\begin{align*}
\chi^2(P_{0, \pi}, P_{0, \mathbf{1}}) &\leq \exp\left(C_1^2 mn^2\rho^4\right) - 1 \\
&\leq \exp\left(C_1^2 c^4(c')^4 n^{\frac{2}{4\beta+1}} \cdot n^2 \cdot n^{-\frac{8\beta+4}{4\beta+1}}\right) - 1\\
&= \exp\left(C_1^2 c^4(c')^4\right) - 1.
\end{align*}
Set \(c_\eta = 1 \wedge \left(\frac{\log\left(1 + 4\eta^2\right)}{(c')^4 C_1^2}\right)^{\frac{1}{4}}\). Since \(c < c_\eta\), we have \(\chi^2(P_{0,\pi}, P_{0, \mathbf{1}}) \leq 4\eta^2\) and so 
\begin{equation*}
\inf_{\varphi}\left\{ \sup_{\substack{f \in \mathcal{H}_\alpha, \\ V \in \mathcal{V}_{0}}} P_{f, V}\left\{\varphi = 1\right\} + \sup_{\substack{f \in \mathcal{H}_\alpha, \\ V \in \mathcal{V}_{1,\beta}\left(\rho m^{1/2}\right)}} P_{f, V}\left\{\varphi = 0\right\} \right\} \geq 1 - \frac{1}{2}\sqrt{4\eta^2} \geq 1-\eta. 
\end{equation*}
Observe that \(\rho m^{1/2} \geq c \frac{c'}{2} n^{-\frac{2\beta}{4\beta+1}}\) and so \(\mathcal{V}_{1,\beta}\left(c\frac{c'}{2}n^{-\frac{2\beta}{4\beta+1}}\right) \supset \mathcal{V}_{1,\beta}(\rho m^{1/2})\). Hence, the desired result is proved. 
\end{proof}

The construction of the prior used in the proof of Proposition \ref{prop:nonparam_lbound} is given in the following proposition. As mentioned, the construction is standard in the literature, and here we specifically follow Theorem 2.1 of \cite{arias-castro_remember_2018}. 
\begin{proposition}\label{prop:prior_construction}
Suppose \(\psi \in C^\infty(\R)\) supported on \([0,1]\) such that \(||\psi||_2 = 1\) and \(\int \psi = 0\). Suppose \(m\) is a nonnegative integer. For \(1 \leq j \leq m\), define \(\psi_j(x) := m^{1/2} \psi(mx - j + 1)\). For \(\rho > 0\) and for \(\kappa \in \left\{-1, 1\right\}^m\), define 
\begin{equation*}
V_\kappa(x) := 1 + \rho \sum_{j=1}^{m} \kappa_j \psi_j(x). 
\end{equation*}
If \(\rho m^{1/2+q}||\psi^{(q)}||_\infty \leq 1\) for all \(q \in \{0,1,...,\lfloor \beta \rfloor\}\) and \(\rho m^{1/2+\beta}(4||\psi^{(\lfloor\beta\rfloor)}||_\infty \vee 2 ||\psi^{(\lfloor \beta \rfloor + 1)} ||_\infty) \leq 1\), then \(V_\kappa \in \mathcal{V}_{1, \beta}(m^{1/2}\rho)\).
\end{proposition}
\begin{proof}
Note \(\psi_j\) is supported on \([(j-1)/m, j/m]\) with \(||\psi_j||_2 = 1\) and \(\int \psi_j = 0\). Therefore, \(\int V_\kappa(x) \, dx = 1\). Further consider that since the \(\{\psi_j\}_{j=1}^{m}\) have disjoint supports, we have 
\begin{equation*}
||V_\kappa - \bar{V}_\kappa \mathbf{1}||_2^2 = ||V_\kappa - \mathbf{1}||_2^2 = \rho^2 \sum_{j=1}^{m} ||\psi_j||^2 = m\rho^2. 
\end{equation*}

Now let us show \(V_\kappa\) is nonnegative. Since the \(\{\psi_j\}_{j=1}^{m}\) have disjoint supports, it follows immediately that \(V_\kappa \geq 0\) since \(\rho m^{1/2} ||\psi||_\infty \leq 1\). Now let us show \(V_\kappa \in \mathcal{H}_\beta\). First, note \(V_\kappa \in C^\infty(\R)\) and so it has derivatives of all orders. Next, since \(\rho m^{1/2+q}||\psi^{(q)}||_\infty \leq 1\) for all \(q \in \{0,1,...,\lfloor \beta \rfloor\}\), it also follows from the disjoint supports of \(\{\psi_j\}_{j=1}^{m}\) that \(||V_\kappa^{(q)}||_\infty \leq 1\) for all \(q \in \{0,1,...,\lfloor \beta \rfloor\}\). 

Second, let us fix \(x, y \in [0, 1]\) and let \(1 \leq k, \ell \leq m\) such that \(\frac{k-1}{m} \leq x \leq \frac{k}{m}\) and \(\frac{\ell-1}{m} \leq y \leq \frac{\ell}{m}\). Then 
\begin{align*}
|V_\kappa^{(\lfloor \beta \rfloor)}(x)- V_\kappa^{(\lfloor \beta \rfloor)}(y)| &= \rho m^{1/2} \left|\sum_{j=1}^{m} \kappa_j \left(\psi_j^{(\lfloor \beta \rfloor)}(x) - \psi_j^{(\lfloor \beta \rfloor)}(y)\right)  \right| \\
&\leq \rho m^{1/2} \left(\left| \psi_k^{(\lfloor \beta \rfloor)}(x) - \psi_k^{(\lfloor \beta \rfloor)}(y) \right| + \left| \psi_\ell^{(\lfloor \beta \rfloor)}(x) - \psi_\ell^{(\lfloor \beta \rfloor)}(y) \right|\right) \\
&= \rho m^{1/2 + \lfloor \beta \rfloor} \left(\left|\psi^{(\lfloor \beta \rfloor)}(mx-k+1) - \psi^{(\lfloor \beta \rfloor)}(my-k+1)\right| \right. \\
&\;\;\; \left. + \left|\psi^{(\lfloor \beta \rfloor)}(mx-\ell+1) - \psi^{(\lfloor \beta \rfloor)}(my-\ell+1)\right|\right) \\
&\leq \rho m^{1/2 + \lfloor \beta \rfloor} \left( 4||\psi^{(\lfloor \beta \rfloor)}||_\infty \wedge \left( 2||\psi^{(\lfloor \beta \rfloor + 1)}||_\infty m |x - y|\right) \right) \\
&\leq \rho m^{1/2 + \lfloor \beta \rfloor} \left( 4||\psi^{(\lfloor \beta \rfloor)}||_\infty \vee 2||\psi^{(\lfloor \beta \rfloor + 1)}||_\infty \right)(1 \wedge m|x-y|) \\
&\leq \rho m^{1/2 + \lfloor \beta \rfloor} \left( 4||\psi^{(\lfloor \beta \rfloor)}||_\infty \vee 2||\psi^{(\lfloor \beta \rfloor + 1)}||_\infty \right)(1 \wedge m|x-y|)^{\beta - \lfloor \beta \rfloor} \\
&\leq \rho m^{1/2 + \beta} \left( 4||\psi^{(\lfloor \beta \rfloor)}||_\infty \vee 2||\psi^{(\lfloor \beta \rfloor + 1)}||_\infty \right) |x - y|^{\beta - \lfloor \beta \rfloor} \\
&\leq |x - y|^{\beta - \lfloor \beta \rfloor}. 
\end{align*}
The proof is complete. 
\end{proof}

\subsection{\texorpdfstring{Proving the \(n^{-2\beta}\) term}{Proving the fixed design term}}
The proof of Proposition \ref{prop:fixed_design_lbound} is presented here. As noted in Section \ref{section:lower_bounds}, it suffices to consider \(\beta \leq \frac{1}{4}\). 
\begin{proof}[Proof of Proposition \ref{prop:fixed_design_lbound}]
Fix \(\eta \in (0, 1)\) and let \(0 < c < c_\eta\) where \(c_\eta\) is to be selected later. Define the function \(V_1 : [0, 1] \to \R\) to be the piecewise linear function with the following knots and values 
\begin{equation*}
V_1(x) = 
\begin{cases}
1 &\textit{if } x \in \left\{\frac{i}{2n}, \frac{i}{n}\right\} \text{ for some } 0 \leq i \leq n, \\
1 + \sqrt{3} cn^{-\beta} &\textit{if } x = \frac{i}{n} + \frac{1}{4n} \text{ for some } 0 \leq i \leq n-1, \\
1 - \sqrt{3} cn^{-\beta} &\textit{if } x = \frac{i}{n} + \frac{3}{4n} \text{ for some } 0 \leq i \leq n-1. 
\end{cases}
\end{equation*}
\begin{figure}[ht!]
\centering
\begin{tikzpicture}[scale=0.9, declare function={
func(\x)= and(\x >= 0, \x < 1/4) * (4*(1/2)*(x-0)+1)     +
and(\x >= 1/4, \x < 1/2) *(4*(-1/2)*(x-1/4)+(1+1/2))
+ and(\x >= 1/2, \x < 3/4)*(((-1)*4*(1/2)*(x-1/2)+1)) +
and(\x >= 3/4, \x < 1)*((-1)*4*(-1/2)*(x-3/4)+(1-1/2))
+
and(\x >= 1, \x < 1+1/4) * (4*(1/2)*(x-1)+1)     +
and(\x >= 1+1/4, \x < 1+1/2) *(4*(-1/2)*(x-1/4-1)+(1+1/2))
+ and(\x >= 1+1/2, \x < 1+3/4)*(((-1)*4*(1/2)*(x-1/2-1)+1)) +
and(\x >= 1+3/4, \x < 2)*((-1)*4*(-1/2)*(x-3/4-1)+(1-1/2))
+
and(\x >= 2, \x < 2+1/4) * (4*(1/2)*(x-2)+1)     +
and(\x >= 2+1/4, \x < 2+1/2) *(4*(-1/2)*(x-1/4-2)+(1+1/2))
+ and(\x >= 2+1/2, \x < 2+3/4)*(((-1)*4*(1/2)*(x-1/2-2)+1)) +
and(\x >= 2+3/4, \x < 3)*((-1)*4*(-1/2)*(x-3/4-2)+(1-1/2))
+
and(\x >= 3, \x < 3+1/4) * (4*(1/2)*(x-3)+1)     +
and(\x >= 3+1/4, \x < 3+1/2) *(4*(-1/2)*(x-1/4-3)+(1+1/2))
+ and(\x >= 3+1/2, \x < 3+3/4)*(((-1)*4*(1/2)*(x-1/2-3)+1)) +
and(\x >= 3+3/4, \x <= 4)*((-1)*4*(-1/2)*(x-3/4-3)+(1-1/2))
;
}]
\begin{axis}[
    axis x line=middle, axis y line=middle,
    ymin=0.1, ymax=2.1, ytick={1/2, 1, 3/2}, yticklabels={$1-\sqrt{3}cn^{-\beta}$, $1$, $1+\sqrt{3}cn^{-\beta}$}, ylabel=$V_1(x)$,
    xmin=0, xmax=4.1, xtick={0,1,2,3,4}, xticklabels={$0$,$\frac{1}{n}$, $\frac{2}{n}$, $\frac{3}{n}$, $\frac{4}{n}$},  xlabel=$x$,
    domain=0:4,samples=101, 
    extra x ticks = {0},
    every major tick/.append style={very thick, black},
]
\draw[dashed](0,1) -- (4,1);

\addplot [blue, thick]{func(x)};
\end{axis}
\end{tikzpicture}
\caption{A cartoon plot of \(V_1(x)\) zoomed in to the domain \(\left[0, \frac{4}{n}\right]\). The extreme points occur at \(\left\{\frac{1}{4n}, \frac{3}{4n}, \frac{5}{4n}, \frac{7}{4n},...\right\}\). The function takes value one at the half-integers \(\left\{0, \frac{1}{2n},\frac{1}{n},\frac{3}{2n},...\right\}\).}\label{fig:fixed_design_lbound}
\end{figure}
\noindent In particular, using the spike function \(g(x) = \sqrt{3}cn^{-\beta}\left(1 - 4n|x|\right)\mathbbm{1}_{\left\{|x| \leq \frac{1}{4n}\right\}}\), we can write for \(0\leq x \leq 1\),
\begin{equation}\label{def:V1_n2beta}
V_1(x) = 1 + \sum_{i=0}^{n-1} g\left(x - \left(\frac{i}{n} + \frac{1}{4n}\right)\right) - g\left(x - \left(\frac{i}{n} + \frac{3}{4n}\right)\right). 
\end{equation}
\noindent Note that by taking \(c_\eta\) sufficiently small (depending on \(\beta\)), we have \(V_1 \in \mathcal{H}_\beta\). Further consider that \(\int_{0}^{1} V_1(x) \, dx = 1\) and so from (\ref{def:V1_n2beta}) we have
\begin{align*}
\int_{0}^{1} (V_1(x) - 1)^2 \,dx &= 4n \cdot \int_{0}^{\frac{1}{4n}} (V_1(x) - 1)^2 \, dx \\
&= 4n \int_{0}^{\frac{1}{4n}} \left(\sqrt{3}cn^{-\beta} \left(1-4n\left|x - \frac{1}{4n}\right|\right)\right)^2 \, dx \\
&= \left(3c^2n^{-2\beta}\right) \cdot 4n \int_{0}^{\frac{1}{4n}} \left(1 - 4n\left(\frac{1}{4n} - x\right)\right)^2 \, dx\\
&= \left(3c^2n^{-2\beta}\right) \cdot 4n \int_{0}^{\frac{1}{4n}} \left(4nx\right)^2 \, dx\\
&= \left(3c^2n^{-2\beta}\right) \cdot (4n)^3 \cdot \int_{0}^{\frac{1}{4n}} x^2 \, dx \\
&= c^2 n^{-2\beta}. 
\end{align*}
Therefore, \(V_1 \in \mathcal{V}_{1, \beta}\left(cn^{-\beta}\right)\). With \(V_1\) in hand, consider that 
\begin{equation}\label{eqn:fixed_design_lbound}
\inf_{\varphi}\left\{ \sup_{\substack{f \in \mathcal{H}_\alpha, \\ V \in \mathcal{V}_0}} P_{f,V}\left\{\varphi = 1\right\} + \sup_{\substack{f \in \mathcal{H}_\alpha, \\ V \in \mathcal{V}_{1, \beta}\left(cn^{-\beta}\right)}} P_{f,V}\left\{\varphi = 0\right\} \right\} \geq \inf_{\varphi}\left\{ P_{0,\mathbf{1}}\left\{\varphi = 1\right\} + P_{0, V_1}\left\{\varphi = 0\right\} \right\}.
\end{equation}
The right hand side corresponds to the two point testing problem 
\begin{align*}
H_0 &: f \equiv 0, V \equiv 1, \\
H_1 &: f \equiv 0, V = V_1. 
\end{align*}
Since \(V_1(i/n) = 1\) for all \(0 \leq i \leq n\), the testing problem is precisely 
\begin{align*}
H_0 &: Y_0,...,Y_n \overset{iid}{\sim} N(0, 1), \\
H_1 &: Y_0,...,Y_n \overset{iid}{\sim} N(0, 1). 
\end{align*}
The null and alternative hypotheses are exactly the same and so the proof is complete. 
\end{proof}

\section{Arbitrary variance profile}\label{appendix:arbitrary_variance}
This section contains the proofs for the results in Section \ref{section:profile}. 

\subsection{Upper bound}
The results stated in Section \ref{section:profile_methodology} and \ref{section:profile_upper_bound} are proved here. 
\begin{proof}[Proof of Proposition \ref{prop: MSE of T_hat beta<1/4}] 
To reduce the notational clutter, the sums range from $0$ to $n-1$ whenever the limits are not specified in what follows. For ease of notation, we will simply write \(P\) and \(E\) instead of \(P_{f,V}\) and \(E_{f,V}\). Recall the notation \(R_i = Y_{i+1} - Y_i\), \(W_i = V(x_i)+V(x_{i+1})\), \(\delta_i = f(x_{i+1}) - f(x_i)\) for \(1 \leq i\leq n-1\), and \(\bar{W}_n\) and \(\overline{\delta^2_n}\) are the respective averages. Note \(E(R_i^2) = W_i + \delta_i^2\) and \(E(R_i^4) = 3W_i^2 + 6W_i\delta_i^2 + \delta_i^4\). Further note that \(|\delta_i| \lesssim n^{-(\alpha \wedge 1)}\) since \(f \in \mathcal{H}_\alpha\). By direct calculation,
\begin{align*}
 E(\hat{T}) &= \frac{1}{2n^2} \sum_{|i-j| \geq 2} \frac{1}{3} \left( 3W_i^2 + 6W_i\delta_i^2 + \delta_i^4 + 3W_j^2 + 6W_j\delta_j^2 + \delta_j^4 \right) - 2(W_i + \delta_i^2)(W_j + \delta_j^2) \\
 &= \frac{1}{2n^2} \sum_{|i-j| \geq 2} \frac{1}{3} \left( 3W_i^2 + 6W_i\delta_i^2 + 3\delta_i^4 + 3W_j^2 + 6W_j\delta_j^2 + 3\delta_j^4 \right) - 2(W_i + \delta_i^2)(W_j + \delta_j^2) - \frac{2}{3}(\delta_i^4+\delta_j^4)\\
 &= \frac{1}{2n^2} \sum_{|i-j| \geq 2} (W_i + \delta_i^2 - W_j - \delta_j^2)^2 - \frac{2}{3} (\delta_i^4 + \delta_j^4) \\
 &= \left(\frac{1}{2n^2} \sum_{|i-j| \geq 2} (W_i + \delta_i^2 - W_j - \delta_j^2)^2\right) + O(n^{-4(\alpha \wedge 1)}) \\
 &= \left(\frac{1}{2n^2} \sum_{|i-j| \geq 1} (W_i + \delta_i^2 - W_j - \delta_j^2)^2\right) + O(n^{-1}) + O(n^{-4(\alpha \wedge 1)}) \\
 &= \left(\frac{1}{n^2} \sum_{i < j} (W_i + \delta_i^2 - W_j - \delta_j)^2\right) + O(n^{-1} + n^{-4(\alpha \wedge 1)}) \\
 &= \left(\frac{1}{n} \sum_{i=1}^{n-1} \left(W_i + \delta_i^2 - \bar{W}_n - \overline{\delta^2_n}\right)^2\right) + O(n^{-1} + n^{-4(\alpha \wedge 1)}) \\
 &= T + O(n^{-1} + n^{-4(\alpha \wedge 1)}).
\end{align*}
Here, we have used the identity \(\sum_{i < j} (a_i - a_j)^2 = (n-1)\sum_{i=1}^{n-1} (a_i - \bar{a}_n)^2\) for any real numbers \(a_1,...,a_{n-1} \in \R\), where \(\bar{a}_n = \frac{1}{n-1} \sum_{i=1}^{n-1} a_i\). We have also used the boundedness of \(f\) and \(V\). With the bias handled, let us now investigate the variance. Consider that by independence and the boundedness of \(V\) and \(f\), we have 
\begin{align*}
\Var(\hat{T}) &= \frac{1}{4n^4} \sum_{|i-j| \geq 2} \sum_{|k-l| \geq 2} \Cov\left(\frac{1}{3}(R_i^4 + R_j^4) - 2R_i^2R_j^2, \frac{1}{3}(R_k^4 + R_l^4) - 2R_k^2R_l^2\right) \\
&= \frac{1}{4n^4} \sum_{|i-j| \geq 2} \sum_{\substack{|k-l| \geq 2, \\ \{k, l\} \cap \{i-1,i,i+1, j-1,j,j+1\} \neq \emptyset}} \Cov\left(\frac{1}{3}(R_i^4 + R_j^4) - 2R_i^2R_j^2, \frac{1}{3}(R_k^4 + R_l^4) - 2R_k^2R_l^2\right) \\
&\lesssim \frac{1}{n^4} \sum_{|i-j| \geq 2} n \\
&\asymp \frac{n^3}{n^4} \\
&= n^{-1}. 
\end{align*}
\noindent Putting together the bias and variance bounds yields \(E\left(\left|\hat{T} - T\right|^2\right) \lesssim n^{-1} + n^{-8(\alpha \wedge 1)}\) as claimed. 
\end{proof}

\begin{proof}[Proof of Proposition \ref{prop:mse_T1_tilde}]
For notational ease, write \(P\) and \(E\) instead of \(P_{f, V}\) and \(E_{f, V}\). Note \(R_{i+1}\) and \(S_i\) are independent since \(R_{i+1} = Y_{i+2} - Y_{i+1}\) and \(S_i = Y_{i+3} - Y_i\). Consider that \(E(R_{i+1}^2) = W_{i+1} + \delta_{i+1}^2\) and \(E(S_i^2) = V_{i+3} + V_i + (f(x_{i+3}) - f(x_i))^2\). Further, note \(E(R_{i+1}^4) = 3W_{i+1}^2 + 6W_{i+1}\delta_{i+1}^2 + \delta_{i+1}^4\) and \(E(S_i^4) = 3(V_{i+3}+V_{i})^2 + 6(V_{i+3} + V_i)(f(x_{i+3}) - f(x_i))^2 + (f(x_{i+3}) - f(x_i))^4\). Thus, calculations similar to those of the proof of Proposition \ref{prop: MSE of T_hat beta<1/4} yield 
\begin{align*}
E\left(\hat{T}_1\right) &= \frac{1}{n} \sum_{i=1}^{n-3} \left(W_{i+1} + \delta_{i+1}^2 - (V_i + V_{i+3}) - (f(x_{i+3}) - f(x_i))^2\right)^2 - \frac{2}{3}\left(\delta_{i+1}^4 + (f(x_{i+3}) - f(x_i))^4\right) \\
&= \tilde{T}_1 + O\left(n^{-4(\alpha \wedge 1)}\right).
\end{align*}
The variance is bounded by an entirely analogous argument to that found in the proof of Proposition \ref{prop: MSE of T_hat beta<1/4}. Consider that by the boundedness of \(f\) and \(V\), we have
\begin{align*}
\Var\left(\hat{T}_1\right) &= \frac{1}{n^2} \sum_{i=1}^{n-3} \sum_{j=1}^{n-3} \Cov\left(\frac{1}{3} \left(R_{i+1}^4 + S_{i}^4\right) - 2R_{i+1}^2 S_i^2, \frac{1}{3} \left(R_{j+1}^4 + S_{j}^4\right) - 2R_{j+1}^2 S_j^2\right) \\
&= \frac{1}{n^2} \sum_{|i-j| < 4} \Cov\left(\frac{1}{3} \left(R_{i+1}^4 + S_{i}^4\right) - 2R_{i+1}^2 S_i^2, \frac{1}{3} \left(R_{j+1}^4 + S_{j}^4\right) - 2R_{j+1}^2 S_j^2\right) \\
&\lesssim n^{-1}. 
\end{align*}
where we have used \((R_{i+1}, S_i) \indep (R_{j+1}, S_j)\) for \(|i-j| \geq 4\). Putting together the bias and variance bounds yields the desired result.
\end{proof}

\begin{proof}[Proof of Proposition \ref{prop:mse_T2_tilde}]
The proof is entirely analogous to the proof of Proposition \ref{prop:mse_T1_tilde}.
\end{proof}

\begin{proof}[Proof of Theorem \ref{thm:beta_zero_upper}]
Fix \(\eta \in (0, 1)\) and let \(C > C_\eta\) where \(C_\eta\) is to be set. We will note where \(C_\eta'\) and \(C_\eta\) need to be taken sufficiently large depending only on \(\eta\). First, the Type I error will be bounded. Fix \(f \in \mathcal{H}_\alpha\) and \(V \in \Sigma_0\). For ease of notation, \(P\) and \(E\) will be used in place of \(P_{f, V}\) and \(E_{f, V}\). Since \(V \in \Sigma_0\) implies \(T, \tilde{T}_1, \tilde{T}_2 \lesssim n^{-4(\alpha \wedge 1)}\), it follows by Propositions \ref{prop: MSE of T_hat beta<1/4}, \ref{prop:mse_T1_tilde}, and \ref{prop:mse_T2_tilde} that 
\begin{equation*}
E\left(\hat{S}^2\right) \lesssim E\left(\left|\hat{S} - (T + \tilde{T}_1 + \tilde{T}_2)\right|^2\right) + n^{-8(\alpha \wedge 1)} \lesssim n^{-1} + n^{-8(\alpha \wedge 1)}. 
\end{equation*}
Therefore, Chebyshev's inequality yields 
\begin{align*}
P\left\{\hat{S} > C_\eta'\left(n^{-4\alpha} + n^{-1/2}\right) \right\} &\leq \frac{E\left(\hat{S}^2\right)}{C_\eta'^2 \left(n^{-4\alpha} + n^{-1/2}\right)^2} \leq \frac{C_1\left(n^{-1} + n^{-8(\alpha \wedge 1)}\right)}{C_\eta'^2(n^{-8\alpha} + n^{-1})} \leq \frac{C_1}{C_\eta'^2} \leq \frac{\eta}{2}
\end{align*}
where \(C_1 > 0\) is a universal constant whose value may change from instance to instance. Note we have used \(n^{-1} + n^{-8(\alpha \wedge 1)} \asymp n^{-1} + n^{-8\alpha}\). The final inequality follows from taking \(C_\eta' > 0\) sufficiently large. Since \(f \in \mathcal{H}_\alpha\) and \(V \in \Sigma_0\) were arbitrary, this bound holds uniformly over \(f \in \mathcal{H}_\alpha\) and \(V \in \Sigma_0\). The Type I error is thus bounded.

It remains to bound the Type II error. Fix \(f \in \mathcal{H}_\alpha\) and \(V \in \Sigma_1\left(C(n^{-2\alpha} + n^{-1/4})\right)\). By Proposition \ref{prop:signal}, we have \(\frac{1}{n} \sum_{i=1}^{n} \left(V_i - \bar{V}\right)^2 \lesssim n^{-1} + n^{-4(\alpha \wedge 1)} + T + \tilde{T}_1 + \tilde{T}_2\). Since \(C_\eta\) is sufficiently large and \(\frac{1}{n} \sum_{i=1}^{n} \left(V_i - \bar{V}\right)^2 \geq C^2\left(n^{-2\alpha} + n^{-1/4}\right)^2 \geq C^2\left(n^{-4\alpha} + n^{-1/2}\right)\), we have \(T + \tilde{T}_1 + \tilde{T}_2 \geq \frac{C^2}{L'} \left(n^{-4\alpha} + n^{-1/2}\right)\) for some universal constant \(L'\). Consider that for sufficiently large \(C_\eta\), an application of Propositions \ref{prop: MSE of T_hat beta<1/4}, \ref{prop:mse_T1_tilde}, and \ref{prop:mse_T2_tilde} yields 
\begin{align*}
P\left\{\hat{S} \leq C_\eta'(n^{-4\alpha} + n^{-1/2}) \right\} &= P\left\{T + \tilde{T}_1 + \tilde{T}_2 - C_\eta'(n^{-4\alpha} + n^{-1/2}) \leq T + \tilde{T}_1 + \tilde{T}_2 - \hat{S}\right\} \\
&\leq \frac{E\left( \left|\hat{S} - T - \tilde{T}_1 - \tilde{T}_2\right|^2\right)}{\left(T + \tilde{T}_1 + \tilde{T}_2 - C_\eta'(n^{-4\alpha} + n^{-1/2})\right)^2} \\
&\leq \frac{C_2 \left(n^{-1} + n^{-8(\alpha \wedge 1)}\right)}{\left(\frac{C^2}{L'} - C_\eta'\right)^2 \left(n^{-4\alpha} + n^{-1/2}\right)^2} \\
&\leq \frac{C_2}{\left(\frac{C^2}{L'} - C_\eta'\right)^2} \\
&\leq \frac{\eta}{2}
\end{align*}
where \(C_2 > 0\) is a universal constant whose value may change from instance to instance. Note we have used \(n^{-1} + n^{-8(\alpha \wedge 1)} \asymp n^{-1} + n^{-8\alpha}\). The final inequality follows from taking \(C_\eta\) sufficiently large. The Type II error is suitably bounded and so the desired result is proved.
\end{proof}

\subsection{Lower bound}\label{section:beta_zero_lower_proof}

\begin{proof}[Proof of Proposition \ref{prop:fourth_moment_info}]
Fix \(\eta \in (0, 1)\) and let \(0 < c < c_\eta\) where we take \(c_\eta = 1\). To define a prior \(\pi\) on \(\Sigma_1(c)\), first consider the distribution \(\tilde{\pi} = \left(\frac{1}{2}\delta_1 + \frac{1}{2}\delta_M\right)^{\otimes n}\) on \([0, M]^n\). To obtain a prior supported on \(\Sigma_1(c)\), we will condition on a high probability event under \(\tilde{\pi}\). For \(V \sim \tilde{\pi}\), define the event 
\begin{equation*}
\mathcal{E} = \left\{ \frac{1}{n} \sum_{i=1}^{n} \left(V_i - \bar{V}_n\right)^2 \geq 1\right\}.
\end{equation*}
For \(V \sim \tilde{\pi}\), consider \(E\left(\frac{1}{n} \sum_{i=1}^{n} \left(V_i - \bar{V}_n\right)^2\right) = \frac{n-1}{n} \cdot \frac{(M-1)^2}{4}\) and \(\Var\left(\frac{1}{n} \sum_{i=1}^{n} \left(V_i - \bar{V}_n\right)^2\right) \lesssim \frac{1}{n}\). Since \(M\) is a sufficiently large constant, by Chebyshev's inequality we have 
\begin{equation*}
P\left\{\frac{1}{n} \sum_{i=1}^{n}\left(V_i - \bar{V}_n\right)^2 \leq 2 - t\right\} \leq P\left\{\frac{1}{n} \sum_{i=1}^{n}\left(V_i - \bar{V}_n\right)^2 \leq \frac{n-1}{n}\cdot \frac{(M-1)^2}{4} - t\right\} \lesssim \frac{1}{nt^2}
\end{equation*}
for \(t > 0\). Therefore, taking \(t = 1\) yields \(\tilde{\pi}(\mathcal{E}) \geq 1 - \eta\) where we have used \(n\) is sufficiently large by assumption. The prior \(\pi\) can now be defined as the probability measure on \(\R^n\) such that for any event \(A\), 
\begin{equation*}
\pi(A) = \frac{\tilde{\pi}(A \cap \mathcal{E})}{\tilde{\pi}(\mathcal{E})}.
\end{equation*}
By definition of \(\mathcal{E}\), it follows \(\pi\) is supported on \(\Sigma_1(1)\). Next, define the noise distribution \(Q_\xi\) to be the distribution of the random variable \(\sqrt{\frac{2}{1+M}} U\) where \(U \sim \frac{1}{2}N(0, 1) + \frac{1}{2}N(0, M)\). Note \(E\left(\left(\sqrt{\frac{2}{1+M}}U\right)^2\right) = \frac{2}{1+M} \left(\frac{1}{2} + \frac{M}{2}\right) = 1\) and \(E\left(U^4\right) \lesssim 1\). Therefore, \(P_\xi \in \Xi\).

With \(\pi\) and \(P_\xi\) in hand, we are in position to bound the minimax testing risk from below by the testing risk in a Bayes testing problem. Specifically, consider the problem 
\begin{align*}
H_0 &: f \equiv 0, V = \frac{1+M}{2} \mathbf{1}_n, \text{ and } \xi_i \overset{iid}{\sim} Q_\xi, \\
H_1 &: f \equiv 0, V \sim \tilde{\pi}, \text{ and } \xi_i \overset{iid}{\sim} N(0, 1),
\end{align*}
where \(\mathbf{1}_n \in \R^n\) denotes the vector with all entries equal to one. Since \(c < c_\eta = 1\) implies \(\Sigma_1(c) \supset \Sigma_1(1)\), we have 
\begin{align*}
&\inf_{\varphi}\left\{ \sup_{\substack{f \in \mathcal{H}_\alpha, \\ V \in \Sigma_0, \\ P_\xi \in \Xi}} P_{f, V, P_\xi}\left\{\varphi = 1 \right\} + \sup_{\substack{f \in \mathcal{H}_\alpha, \\ V \in \Sigma_1(c), \\ P_\xi \in \Xi}} P_{f, V, P_\xi}\left\{\varphi = 0\right\}\right\} \\
&\geq \inf_{\varphi}\left\{ \sup_{\substack{f \in \mathcal{H}_\alpha, \\ V \in \Sigma_0, \\ P_\xi \in \Xi}} P_{f, V, P_\xi}\left\{\varphi = 1 \right\} + \sup_{\substack{f \in \mathcal{H}_\alpha, \\ V \in \Sigma_1(1), \\ P_\xi \in \Xi}} P_{f, V, P_\xi}\left\{\varphi = 0\right\} \right\} \\
&\geq 1 - d_{TV}\left(P_{0,\frac{1+M}{2}\mathbf{1}_n, Q_\xi}, P_{0, \pi, N(0, 1)}\right) \\
&\geq 1-d_{TV}\left(P_{0,\frac{1+M}{2}\mathbf{1}_n,Q_\xi}, P_{0, \tilde{\pi}, N(0, 1)}\right) - d_{TV}\left(P_{0, \tilde{\pi}, N(0, 1)},P_{0, \pi, N(0, 1)}\right) 
\end{align*}
where \(P_{0, \pi, N(0, 1)} = \int P_{0, V, N(0, 1)} \,\pi(dV)\) is the mixture induced by \(\pi\). The final line is a consequence of triangle inequality. By definition of \(\pi\), it follows \(d_{TV}\left(P_{0, \tilde{\pi}, N(0, 1)},P_{0, \pi, N(0, 1)}\right) \leq \tilde{\pi}(\mathcal{E}^c) \leq \eta\). It remains to bound \(d_{TV}\left(P_{0,\frac{1+M}{2}\mathbf{1}_n,Q_\xi}, P_{0, \tilde{\pi}, N(0, 1)}\right)\). This total variation distance is the optimal testing error in the problem 
\begin{align*}
H_0 &: f \equiv 0, V = \frac{1+M}{2} \mathbf{1}_n, \text{ and } \xi_i \overset{iid}{\sim} Q_\xi, \\
H_1 &: f \equiv 0, V \sim \tilde{\pi} \text{ and } \xi_i \overset{iid}{\sim} N(0, 1).
\end{align*}
Recall \(\tilde{\pi} = \left(\frac{1}{2}\delta_1 + \frac{1}{2}\delta_M\right)^{\otimes n}\) and \(Q_\xi\) is the distribution of \(\sqrt{\frac{2}{1+M}} U\) where \(U \sim \frac{1}{2}N(0, 1) + \frac{1}{2}N(0, M)\). Therefore, under both \(H_0\) and \(H_1\) we have \(V_i^{1/2}\xi_i \overset{iid} {\sim} \frac{1}{2}N(0, 1) + \frac{1}{2}N(0, M)\). Since \(f \equiv 0\) under both \(H_0\) and \(H_1\), the marginal distribution of the data agree 
\begin{align*}
H_0 &: Y_1,...,Y_n \overset{iid}{\sim} \frac{1}{2}N(0, 1) + \frac{1}{2}N(0, M), \\
H_1 &: Y_1,...,Y_n \overset{iid}{\sim} \frac{1}{2} N(0, 1) + \frac{1}{2}N(0, M).
\end{align*}
Therefore, \(d_{TV}\left(P_{0,\frac{1+M}{2}\mathbf{1}_n,Q_\xi}, P_{0, \tilde{\pi}, N(0, 1)}\right) = 0\) and so the proof is complete.
\end{proof}

\begin{proof}[Proof of Theorem \ref{thm:beta_zero_lower}]
Fix \(\eta \in (0, 1)\) and let \(0 < c < c_\eta\) where \(c_\eta\) is to be set later. Let \(0 < \tau < 1\) where \(\tau\) will be chosen later. A prior \(\pi\) will be defined on \(\Sigma_1(c n^{-1/4})\). First, define \(\tilde{\pi}\) where a draw \(V \sim \tilde{\pi}\) is obtained by drawing \(R_1,...,R_n \overset{iid}{\sim} \Rademacher\left(\frac{1}{2}\right)\) and setting \(V_i = 1 + \tau + \rho R_i\) where \(\rho = \sqrt{2}cn^{-1/4}\). Since \(n\) is sufficiently large, we have \(V_i \geq 0\). Observe that \(\bar{V}_n = 1 + \tau + \rho\bar{R}_n\) where \(\bar{R}_n = \frac{1}{n}\sum_{i=1}^{n} R_i\) and so \(\frac{1}{n} \sum_{i=1}^{n} \left(V_i - \bar{V}_n\right)^2 = \rho^2 \cdot \frac{1}{n} \sum_{i=1}^{n} (R_i - \bar{R}_n)^2\). It is clear that \(\tilde{\pi}\) is not supported on \(\Sigma_1(cn^{-1/4})\) since \(\frac{1}{n} \sum_{i=1}^{n} (R_i - \bar{R}_n)^2\) may be small, so we will need to condition on an high probability event. For \(V \sim \tilde{\pi}\), define the event 
\begin{align*}
\mathcal{E} &= \left\{ \frac{1}{n} \sum_{i=1}^{n} \left(V_i - \bar{V}_n\right)^2 > \frac{\rho^2}{2}\right\}. 
\end{align*}
Since \(E\left(\frac{1}{n} \sum_{i=1}^{n} (R_i - \bar{R}_n)^2\right) = \frac{n-1}{n}\) and \(\Var\left(\frac{1}{n} \sum_{i=1}^{n} \left(R_i - \bar{R}_n\right)^2\right) \lesssim \frac{1}{n}\), it follows that for any $t>0$, 
\begin{equation*}
P\left\{\frac{1}{n} \sum_{i=1}^{n} (V_i - \bar{V}_n)^2 \leq \rho^2 - t\right\} \lesssim \frac{\rho^4}{n\left(t - \frac{1}{n}\rho^2\right)^2}  
\end{equation*}
by Chebyshev's inequality. Therefore, \(\tilde{\pi}(\mathcal{E}) \geq 1 - \frac{\eta}{2}\) where we have used \(n\) is sufficiently large by assumption.

We are now in position to define \(\pi\). Define \(\pi\) to be the probability measure on \(\R^n\) such that for any event \(A\),
\begin{equation*}
\pi(A) = \frac{\tilde{\pi}(A \cap \mathcal{E})}{\tilde{\pi}(\mathcal{E})}.
\end{equation*}
Since \(\rho = \sqrt{2}cn^{-1/4}\), it follows \(\pi\) is supported on \(\Sigma_1(cn^{-1/4})\). With \(\pi\) defined, consider the testing problem 
\begin{align*}
H_0 &: f \equiv 0, V = (1+\tau)\mathbf{1}_n, \\
H_1 &: f \equiv 0, V \sim \pi. 
\end{align*}
Here, \(\mathbf{1}_n \in \R^n\) denotes the vector with all entries equal to one. It is clear that the minimax testing risk is bounded from below by the optimal testing risk for the above Bayes testing problem. In particular, we have 
\begin{align*}
\inf_{\varphi}\left\{ \sup_{\substack{f \in \mathcal{H}_\alpha, \\ V \in \Sigma_0}} P_{f, V}\left\{\varphi = 1\right\} + \sup_{\substack{f \in \mathcal{H}_\alpha, \\ V \in \Sigma_1(cn^{-1/4})}} P_{f, V}\left\{\varphi = 0\right\} \right\} \geq 1 - d_{TV}(P_{0, (1+\tau)\mathbf{1}_n}, P_{0, \pi})
\end{align*}
where \(P_{0, \pi} = \int P_{0, V} \pi(dV)\) is the mixture induced by \(\pi\). To obtain the desired result, we will aim to show \(d_{TV}(P_{0, (1+\tau)\mathbf{1}_n}, P_{0, \pi}) \leq \eta\). Consider by triangle inequality 
\begin{align*}
d_{TV}(P_{0, (1+\tau)\mathbf{1}_n}, P_{0, \pi}) &\leq d_{TV}(P_{0, (1+\tau)\mathbf{1}_n}, P_{0, \tilde{\pi}}) + d_{TV}(P_{0, \tilde{\pi}}, P_{0, \pi}) \\
&\leq d_{TV}(P_{0, (1+\tau)\mathbf{1}_n}, P_{0, \tilde{\pi}}) + \tilde{\pi}(\mathcal{E}^c) \\
&\leq d_{TV}(P_{0, (1+\tau)\mathbf{1}_n}, P_{0, \tilde{\pi}}) + \frac{\eta}{2}. 
\end{align*}
Thus, it remains to bound \(d_{TV}(P_{0, (1+\tau)\mathbf{1}_n}, P_{0, \tilde{\pi}})\), which corresponds to the optimal testing risk for the problem 
\begin{align*}
H_0 &: Y_1,...,Y_n \overset{iid}{\sim} N(0, 1+\tau), \\
H_1 &: Y_1,...,Y_n \overset{iid}{\sim} \frac{1}{2} N(0, 1+\tau-\rho) + \frac{1}{2}N(0, 1+\tau+\rho). 
\end{align*}
Note we can write \(N(0, 1+\tau) = \nu_0 * N(0, 1)\) and \(\frac{1}{2} N(0, 1+\tau-\rho) + \frac{1}{2}N(0, 1+\tau+\rho) = \nu_1 * N(0, 1)\) where \(\nu_0 = N(0, \tau)\) and \(\nu_1 = \frac{1}{2}N(0, \tau-\rho) + \frac{1}{2} N(0, \tau+\rho)\). Here, \(*\) denotes convolution. Let us pick \(\tau = 2\rho\). Then, observe the first three moments of \(\nu_0\) and \(\nu_1\) match. Further note \(\nu_0\) and \(\nu_1\) are \(C_1\sqrt{\rho}\)-subgaussian for some universal constant \(C_1 > 0\). Since \(C_1 \sqrt{\rho} < 1\) as \(n\) is sufficiently large by assumption, it follows by Proposition \ref{prop:chisquare_moment_matching} 
\begin{equation*}
\chi^2\left(\nu_0 * N(0, 1), \nu_1*N(0, 1)\right) \lesssim \rho^{4} = \frac{4c^4}{n}. 
\end{equation*}
Selecting \(c_\eta\) sufficiently small depending only on \(\eta\), it follows by Lemma \ref{lemma:chisquare_tensorization} that
\begin{align*}
d_{TV}(P_{0, (1+\tau)\mathbf{1}_n}, P_{0, \tilde{\pi}}) \leq \frac{1}{2}\sqrt{\chi^2\left(P_{0, (1+\tau)\mathbf{1}_n}, P_{0, \tilde{\pi}}\right)} \leq \frac{\eta}{2}. 
\end{align*}
The proof is complete.
\end{proof}

\section{Adaptation}\label{appendix:adaptation}
\begin{proof}[Proof of Theorem \ref{thm:alpha_adapt}]
The proof is a direct application of Proposition \ref{prop:nuisance_lbound}. Fix \(\eta \in (0, 1)\) and let \(0 < c < c_\eta\) where \(c_\eta\) is set as in the proof of Proposition \ref{prop:nuisance_lbound}. Let \(V_1 \in \R^n \) be the restriction to the design points of the variance function defined in the proof of Proposition \ref{prop:nuisance_lbound} with the choice \(\alpha = \alpha_1\) and \(\beta > 0\) arbitrarily small. Let \(\pi\) denote the prior on \(\mathcal{H}_{\alpha_1}\) defined in the proof of Proposition \ref{prop:nuisance_lbound}. Then observe 
\begin{align*}
&\inf_{\varphi}\left\{ \sup_{\substack{f \in \mathcal{H}_{\alpha_1}, \\ V \in \Sigma_0}} P_{f, V}\left\{\varphi = 1\right\} + \sup_{\substack{f \in \mathcal{H}_{\alpha_2}, \\ V \in \Sigma_1\left(cn^{-2\alpha_1}\right)}} P_{f,V}\left\{\varphi = 0\right\} \right\} \\
&\geq \inf_{\varphi}\left\{ P_{\pi, \mathbf{1}_n}\{\varphi = 1\} + P_{0, V_1}\left\{\varphi = 0\right\}\right\} \\
&= 1 - d_{TV}(P_{\pi, \mathbf{1}_n}, P_{0, V_1}) \\
&\geq 1 - \frac{1}{2}\sqrt{\chi^2(P_{\pi, \mathbf{1}_n}, P_{0, V_1})}
\end{align*}
where \(P_{\pi, \mathbf{1}_n} = \int P_{f, \mathbf{1}_n} \,\pi(df)\) is the mixture induced by \(\pi\). Observe that \(1 - \frac{1}{2}\sqrt{\chi^2(P_{\pi, \mathbf{1}_n}, P_{0, V_1})}\) is analogous to (\ref{eqn:chisquare_lowerbound}), and so the proof of Proposition \ref{prop:nuisance_lbound} can be essentially repeated to obtain \(1 - \frac{1}{2}\sqrt{\chi^2(P_{\pi, \mathbf{1}_n}, P_{0, V_1})} \geq 1-\eta\) as desired. 
\end{proof}

\subsection{Numerical simulation}\label{appendix:adaptation_simulation}
To illustrate the impossibility of adaptation asserted by Theorem \ref{thm:alpha_adapt}, we conducted a numerical simulation. Though a simulation cannot definitely demonstrate the impossibility of a statistical task, a numerical demonstration can aid in gaining intuition. In particular, we will illustrate that the rate-optimal testing procedure proposed in Section 3 fails to adapt in the testing problem (\ref{problem:alpha_adapt0})-(\ref{problem:alpha_adapt1}). Fix \(\alpha_1, \alpha_2 \in (0, \frac{1}{4})\). Specifically, we will demonstrate that when \(\alpha_1 < \alpha_2\), the test has low power when the separation in the alternative hypothesis is of order \(n^{-2\alpha_1}\), as asserted in Theorem \ref{thm:alpha_adapt}. On the other hand, when \(\alpha_1 \geq \alpha_2\), we expect that the test can have nontrivial power when the separation is of the desired order \(\varepsilon^*(\alpha_2)\) given by (\ref{rate:betazero}) corresponding to the smoothness \(\alpha_2\) of the alternative hypothesis. 

To elaborate on why this behavior is to be expected when \(\alpha_1 \geq \alpha_2\), recall the test statistic \(\hat{S}\) given by (\ref{def:Shat}) and its corresponding test \(\varphi_{\alpha_1} = \mathbbm{1}_{\{\hat{S} > C(n^{-4\alpha_1} + n^{-1/2})\}}\). By choosing \(C\) appropriately, the Type I error of \(\varphi_{\alpha_1}\) is controlled via arguments of Section 3 since the smoothness in \(H_0\) is \(\alpha_1\). Examining the Type II error, it is clear \(n^{-4\alpha_1} \leq n^{-4\alpha_2}\) since \(\alpha_1 \geq \alpha_2\). Therefore, it follows \(\varphi_{\alpha_1} \geq \varphi_{\alpha_2}\) where \(\varphi_{\alpha_2} = \mathbbm{1}_{\{\hat{S} > C(n^{-4\alpha_2} + n^{-1/2})\}}\). Since the arguments of Section 3 control the Type II error of \(\varphi_{\alpha_2}\) when the separation is of order \(\varepsilon^*(\alpha_2)\) (because the smoothness in \(H_1\) is \(\alpha_2\)), it follows directly from \(\varphi_{\alpha_1} \geq \varphi_{\alpha_2}\) that the Type II error of \(\varphi_{\alpha_1}\) is also controlled. Therefore, \(\varphi_{\alpha_1}\) can successfully detect when \(\alpha_1 \geq \alpha_2\).

To demonstrate the two behaviors of the test in the two regimes, we conduct the following simulation. The testing problem (31)-(32) is a composite-vs-composite problem, so for simulation purposes we actually consider a simple-vs-simple Bayes testing subproblem which is close to that used in the lower bound argument of Theorem \ref{thm:alpha_adapt}. For our simulation, we consider the sample size \(n = 2000\), and we examine the testing problem
\begin{align*}
H_0 &: f \sim \pi_{\alpha_1} \text{ and } V_i = 1,\\
H_1 &: f \equiv 0 \text{ and } V_i = V_{\alpha_1 \wedge \alpha_2}(i/n),
\end{align*}
where \(V_\alpha(x) = \mathbbm{1}_{\{x < 1/2\}} + (1+10n^{-2\alpha})\mathbbm{1}_{\{x > 1/2\}}\). Here, a draw \(f \sim \pi_\alpha\) is given by \(f(x) = \sum_{i=1}^{n} R_i (\sqrt{V_\alpha(x) - 1}) g(x-i/n)\) with \(R_1,...,R_n \overset{iid}{\sim} N(0, 1)\) and \(g(x) = (1-2n|x|)\mathbbm{1}_{\{|x| \leq \frac{1}{2n}\}}\). Note the separation is of order \(n^{-2(\alpha_1 \wedge \alpha_2)}\) in the alternative hypothesis, since we wish to illustrate low power at separation \(n^{-2\alpha_1}\) when \(\alpha_1 < \alpha_2\), and nontrivial power at separation \(n^{-2\alpha_2}\) when \(\alpha_1 > \alpha_2\). Further note this testing problem does not exactly match the problem in the lower bound argument of Theorem \ref{thm:alpha_adapt}, but the differences are minor and our simulation is set up to illustrate the essence of the impossibility phenomenon. Additionally, we only consider separations of the form \(n^{-2\alpha}\) rather than \(n^{-2\alpha} + n^{-1/4}\) in order to stay faithful to the lower bound construction of Theorem \ref{thm:alpha_adapt}. Note that the term \(n^{-1/4}\) appearing in the minimax rate \(\varepsilon^*(\alpha_2)\) is proved via a different lower bound construction and thus is not relevant in this simulation.

The simulation was conducted as follows. We considered a collection \(\mathcal{A}\) of \(20\) different values of \(\alpha\) spread evenly in the interval \([0, \frac{1}{4}]\). For each \(\alpha_1 \in \mathcal{A}\), we computed the cutoff \(c_{\alpha_1}\) via \(100\) simulations of the null hypothesis associated to smoothness \(\alpha_1\) so that the test \(\varphi_{\alpha_1} = \mathbbm{1}_{\{\hat{S} > c_{\alpha_1}\}}\) (approximately) controls the Type I error at level \(0.05\). Turning to the power, for each \(\alpha \in \mathcal{A}\), we simulated \(100\) datasets from the alternative hypothesis associated to smoothness \(\alpha\). Then for each \(\alpha_1, \alpha_2 \in \mathcal{A}\), we applied the test \(\varphi_{\alpha_1}\) to each of the \(100\) datasets associated to the alternative hypothesis corresponding to smoothness \(\alpha_1 \wedge \alpha_2\). The power of \(\varphi_{\alpha_1}\) was estimated by averaging the result of \(\varphi_{\alpha_1}\) across those \(100\) simulated datasets. 

\begin{figure}
\centering
\includegraphics{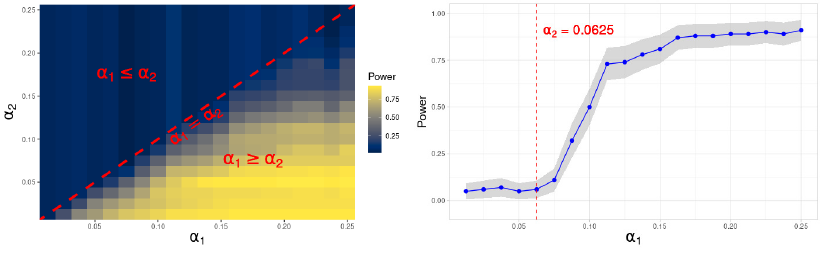}
\caption{(Left) Heatmap of the estimated power of \(\varphi_{\alpha_1}\) across \((\alpha_1, \alpha_2)\). The power is low in the regime \(\alpha_1 < \alpha_2\) where the separation in the alternative is of order \(n^{-2\alpha_1}\), in accordance with Theorem \ref{thm:alpha_adapt}. The power increases in the regime \(\alpha_1 \geq \alpha_2\) where the separation is of order \(n^{-2\alpha_2}\). (Right) Cross-section across \(\alpha_1\) at \(\alpha_2 = 0.0625\). Error bars are pointwise 95\% Wald confidence intervals.}\label{fig:adapt_sim}
\end{figure}

The left panel of Figure \ref{fig:adapt_sim} shows a heatmap of the estimated power of \(\varphi_{\alpha_1}\) and the right panel of Figure \ref{fig:adapt_sim} shows the power at the cross-section of \(\alpha_2 = 0.0625\), with the error bars signifying a pointwise \(95\%\) Wald confidence interval. These figures agree with the expectations outlined earlier. Theorem \ref{thm:alpha_adapt} asserts that when \(\alpha_1 < \alpha_2\), it is impossible to detect heteroskedasticity having separation \(n^{-2\alpha_1}\); indeed, the left panel of Figure \ref{fig:adapt_sim} shows the rate-optimal test \(\varphi_{\alpha_1}\) achieves very low power when \(\alpha_1 < \alpha_2\). Likewise, in the right panel, the test has very low power when \(\alpha_1 < 0.0625\). On the other hand, we expect detection to be possible when \(\alpha_1 \geq \alpha_2\) as explained earlier. Indeed, the power increases in the region \(\alpha_1 \geq \alpha_2\) as seen in the left panel of Figure \ref{fig:adapt_sim}. Examining the right panel, the power increases quite a bit right when \(\alpha_1\) crosses the threshold of \(0.0625\). Hence, the numerical simulation agrees with our expectations and thus illustrates the impossibility of adaptation.
\section{Heteroskedasticity across the design}\label{appendix:heteroskedasticity_across}
\subsection{Known noise distribution}
\begin{proof}[Proof of Theorem \ref{theorem: optimal test discrete known noise}]
When $\beta<\frac{1}{4}$, the result directly follows from Theorem \ref{thm:beta_zero_upper}. When $\beta \geq \frac{1}{4}$, we can use the same logic as in the proof of Theorem \ref{theorem: MSE of T_hat} to derive for any $\frac{1}{4} < \beta < \frac{1}{2}$,
\begin{align*}
 E_{f,V}\left(\left|\hat{T} - T\right|^2\right) 
 &\lesssim n^{-8(\alpha \wedge 1)} + n^{-\frac{8\beta}{4\beta+1}} + n^{-4\beta} + \frac{||W - \bar{W}_n\mathbf{1}_n||_2^2}{n^2}\nonumber\\
 &\lesssim n^{-8(\alpha \wedge 1)} + n^{-\frac{8\beta}{4\beta+1}} + n^{-4\beta} + \frac{\sum_{i=0}^n\left(V\left(\frac{i}{n}\right) - \bar{V}_n\right)^2  }{n^2}.
\end{align*}
Now following the same logic as the proof of Theorem \ref{theorem: optimal test beta>1/4}, we can get the required result.
\end{proof}

\begin{proof}[Proof of Proposition \ref{prop:discrete_known_noise_sqrtn}]
The proof is largely the same as the proof of Theorem \ref{thm:beta_zero_lower} with some slight modifications. Fix \(\eta \in (0, 1)\) and set \(0 < c < c_\eta\) where we take \(c_\eta\) sufficiently small like in the proof of Theorem \ref{thm:beta_zero_lower}. Consequently, we will only point out the changes and omit details for brevity. We will define a prior \(\tilde{\pi}\) which is similar to the \(\tilde{\pi}\) appearing in the proof of Theorem \ref{thm:beta_zero_lower}. A draw \(V \sim \tilde{\pi}\) is obtained by drawing \(R_0,...,R_n \overset{iid}{\sim} \Rademacher\left(\frac{1}{2}\right)\) and setting, for \(0 \leq x \leq 1\), 
\begin{equation*}
V(x) = 1 + \sum_{i=0}^{n} R_i \cdot cn^{-1/4} g\left(x - \frac{i}{n}\right)
\end{equation*}
where \(g(x) = \left(1 - 2n|x|\right)\mathbbm{1}_{\{|x| \leq \frac{1}{2n}\}}\). Note \(V \in \mathcal{H}_\beta\) for \(\beta \leq \frac{1}{4}\) and \(V \geq 0\) for \(n\) large enough. Notice that \(V(x_i) \overset{iid}{\sim} \frac{1}{2}\delta_{1 + cn^{-1/4}} + \frac{1}{2}\delta_{1 - cn^{-1/4}}\), which is precisely the same distribution as \(V_i \overset{iid}{\sim} \tilde{\pi}\) in the proof of Theorem \ref{thm:beta_zero_lower}. Consequently, one can define a prior \(\pi\) from \(\tilde{\pi}\) by conditioning on a high probability event and follow the argument of the proof of Theorem \ref{thm:beta_zero_lower} to obtain the desired result. The proof details are omitted. 
\end{proof}

\subsection{Unknown noise distribution}\label{section:discrete_unknown_noise_proofs}

\begin{proof}[Proof of Theorem \ref{theorem: optimal test discrete unknown noise}]
The proof is essentially the proof of Theorem \ref{theorem: optimal test discrete known noise}. We omit the details.
\end{proof}

\begin{proof}[Proof of Proposition \ref{prop:discrete_unknown_noise_2beta}]
The architecture of the proof is the same as the proofs of Proposition \ref{prop:fourth_moment_info} and Theorem \ref{thm:beta_zero_lower}. Only small details in the prior construction are different. Fix \(\eta \in (0, 1)\) and let \(0 < c < c_\eta\) where we take \(c_\eta = 1\). To define a prior \(\pi\) on \(\mathcal{D}_{1,\beta}(cn^{-\beta})\), a preliminary probability distribution \(\tilde{\pi}\) on \(\mathcal{H}_\beta\) will be constructed first. Define the function \(g : \R \to \R\) with \(g(x) = (1 - 2n|x|)\mathbbm{1}_{\{|x| \leq \frac{1}{2n}\}}\). The distribution \(\tilde{\pi}\) is defined as follows. A draw \(V \sim \tilde{\pi}\) is obtained by first drawing \(R_0,...,R_n \overset{iid}{\sim} \text{Rademacher}\left(\frac{1}{2}\right)\) and, for \(0 \leq x \leq 1\), setting 
\begin{equation*}
V(x) = 1 + \sum_{i=0}^{n} R_i \cdot \sqrt{2}cn^{-\beta} g\left(x - \frac{i}{n}\right).
\end{equation*}
It is straightforward to check \(V \in \mathcal{H}_\beta\) and \(V \geq 0\). Since the design points are \(x_i = \frac{i}{n}\) for \(0 \leq i \leq n\), we have \(V(x_i) = 1 + \sqrt{2} R_i cn^{-\beta}\) and \(\bar{V}_n = \frac{1}{n+1}\sum_{i=0}^{n} V(x_i) = 1 + \sqrt{2}\bar{R}_n cn^{-\beta}\) where \(\bar{R}_n = \frac{1}{n+1}\sum_{i=0}^{n} R_i\). Therefore, \(\frac{1}{n+1}\sum_{i=0}^{n} \left(V(x_i) - \bar{V}_n\right)^2 = 2c^2n^{-2\beta} \cdot \frac{1}{n+1}\sum_{i=0}^{n}\left(R_i - \bar{R}_n\right)^2\). As discussed in the proof of Theorem \ref{thm:beta_zero_lower}, \(\tilde{\pi}\) need not be supported on \(\mathcal{D}_{1,\beta}(cn^{-\beta})\) since \(\frac{1}{n+1}\sum_{i=0}^{n}\left(R_i - \bar{R}_n\right)^2\) may be quite small. However, it is small only a small probability event. Therefore, we will condition on a high probability event to furnish a prior \(\pi\) which is supported on \(\mathcal{D}_{1,\beta}(cn^{-\beta})\). Define the event 
\begin{equation*}
\mathcal{E} = \left\{\frac{1}{n} \sum_{i=1}^{n} \left(V_i - \bar{V}_n\right)^2 > c^2n^{-2\beta}\right\}.
\end{equation*}
By the same argument employed in the proof of Theorem \ref{thm:beta_zero_lower}, it follows \(\tilde{\pi}(\mathcal{E}) \geq 1 - \eta\) since \(n\) is sufficiently large depending only on \(\eta\). Define the prior \(\pi\) to be the distribution such that for any event \(A\), 
\begin{equation*}
\pi(A) = \frac{\tilde{\pi}(A \cap \mathcal{E})}{\tilde{\pi}(\mathcal{E})}.
\end{equation*}
Note \(\pi\) is supported on \(\mathcal{D}_{1,\beta}(cn^{-\beta})\). With \(\pi\) in hand, consider the Bayes testing problem 
\begin{align*}
H_0 &: f \equiv 0, V = \mathbf{1}, \text{ and } \xi_i \overset{iid}{\sim} Q_\xi , \\
H_1 &: f \equiv 0, V \sim \pi, \text{ and } \xi_i \overset{iid}{\sim} N(0, 1),
\end{align*}
where \(Q_\xi = \frac{1}{2}N\left(0, 1 + cn^{-\beta}\right) + \frac{1}{2}N\left(0, 1 - cn^{-\beta}\right)\). It is clear the optimal testing risk for this Bayes testing problem is a lower bound for the minimax testing risk. Specifically, we have 
\begin{align*}
&\inf_{\varphi}\left\{ \sup_{\substack{f \in \mathcal{H}_\alpha, \\ V \in \mathcal{D}_0, \\ P_\xi \in \Xi}} P_{f, V, P_\xi}\left\{\varphi = 1\right\} + \sup_{\substack{f \in \mathcal{H}_\alpha, \\ V \in \mathcal{D}_{1,\beta}(cn^{-\beta}), \\ P_\xi \in \Xi}} P_{f,V,P_\xi}\left\{\varphi = 0\right\} \right\} \\
&\geq 1 - d_{TV}\left(P_{0,\mathbf{1},Q_\xi}, P_{0, \pi, N(0, 1)}\right) \\
&\geq 1 - d_{TV}\left(P_{0, \pi, N(0, 1)}, P_{0, \tilde{\pi}, N(0, 1)}\right) - d_{TV}\left(P_{0,\mathbf{1},Q_\xi}, P_{0, \tilde{\pi}, N(0, 1)}\right)
\end{align*}
where \(P_{0, \pi, N(0, 1)} = \int P_{0, V, N(0, 1)} \pi(dV)\) is the mixture induced by \(\pi\). By definition of \(\pi\), it follows that \(d_{TV}\left(P_{0, \pi, N(0, 1)}, P_{0, \tilde{\pi}, N(0, 1)}\right) \leq \tilde{\pi}(\mathcal{E}^c) \leq \eta\). To bound \(d_{TV}\left(P_{0,\mathbf{1},Q_\xi}, P_{0, \tilde{\pi}, N(0, 1)}\right)\), consider by Neyman-Pearson lemma it is the optimal risk for the hypothesis testing problem 
\begin{align*}
H_0 &: f \equiv 0, V = \mathbf{1}, \text{ and } \xi_i \overset{iid}{\sim} Q_\xi, \\
H_1 &: f \equiv 0, V \sim \tilde{\pi}, \text{ and } \xi_i \overset{iid}{\sim} N(0, 1).
\end{align*}
By definition of \(Q_\xi\), it follows that under the null hypothesis we have \(Y_0,...,Y_n \overset{iid}{\sim} \frac{1}{2}N(0, 1 + cn^{-\beta}) + \frac{1}{2}N(0, 1-cn^{-\beta})\). Under the alternative, by definition of \(\tilde{\pi}\) we have \(V(x_i) \overset{iid}{\sim} \frac{1}{2}\delta_{1 + cn^{-\beta}} + \frac{1}{2}\delta_{1 - cn^{-\beta}}\). Since \(Y_i = V^{1/2}(x_i)\xi_i\), it follows \(Y_0,...,Y_n \overset{iid}{\sim} \frac{1}{2}N(0, 1 + cn^{-\beta}) + \frac{1}{2}N(0, 1-cn^{-\beta})\) under the alternative. Therefore, the data share the same distribution under \(H_0\) and \(H_1\), and so the hypotheses are indistinguishable. Hence, \(d_{TV}\left(P_{0,\mathbf{1},Q_\xi}, P_{0, \tilde{\pi}, N(0, 1)}\right) = 0\) and the proof is complete.
\end{proof}

\section{Auxiliary results}

\begin{lemma}[\(\chi^2\) tensorization \cite{tsybakov_introduction_2009}]\label{lemma:chisquare_tensorization}
We have 
\begin{equation*}
 \chi^2\left(\bigotimes_i P_i, \bigotimes_i Q_i\right) = \left(\prod_{i=1}^{n} (1 + \chi^2(P_i, Q_i))\right) - 1.
\end{equation*}
\end{lemma}

\begin{lemma}[Lemma 1 \cite{wang_effect_2008}]\label{lemma:moment_matching}
For any fixed positive integer \(q\), there exists \(B < \infty\) and a symmetric distribution \(G\) on \([-B, B]\) such that \(G\) and the standard normal distribution have the same first \(q\) moments, that is, 
\begin{equation*}
\int_{-B}^{B} x^j G(dx) = \int_{-\infty}^{\infty} x^j \varphi(x) \,dx 
\end{equation*}
for \(j=1,2,...,q\), where \(\varphi\) denotes the density of the standard normal distribution. 
\end{lemma}

\begin{proposition}[Theorem 3.3.3 \cite{wu_polynomial_2020}]\label{prop:chisquare_moment_matching}
Suppose \(\nu\) and \(\nu'\) are two probability distributions sharing the first \(L\) moments. If \(\nu, \nu'\) are \(\varepsilon\)-subgaussian for \(\varepsilon < 1\), then 
\begin{equation*}
\chi^2(\nu * N(0, 1), \nu' * N(0, 1)) \leq \frac{16}{\sqrt{L}} \frac{\varepsilon^{2L+2}}{1-\varepsilon^2}. 
\end{equation*}
\end{proposition}

\begin{definition}\label{def:subgaussian}
A random variable \(X\) is \(\sigma\)-subguassian if \(E(e^{tX}) \leq e^{\frac{\sigma^2t^2}{2}}\) for all \(t \in \R\). 
\end{definition}

\begin{lemma}\label{lemma:bounded_subgaussian}
If \(X\) is a real-valued random variable such that \(a \leq X \leq b\) with probability one, then \(X\) is \(\frac{b-a}{2}\)-subgaussian. 
\end{lemma}

\begin{lemma}\label{lemma:zygmund_holder_interval_discrete}
For \(h \in \Z\), define the finite difference operator \(D_h\) which sends a function \(g : \Z \to \R\) to the function \(D_h g : \Z \to \R\) with \(D_h g(z) = g(z + h) - g(z)\). Suppose for some \(\alpha < 1\)
\begin{equation*}
||g||_{*} := \sup_{\substack{z \in \Z, \\h\in E_n\setminus \{0\}}} \frac{|D_h^2g(z)|}{|h/n|^{\alpha}} < \infty,
\end{equation*}
where \(E_n = [-n, n] \cap \Z\). Then for all \(z \in E_n\) we have
\begin{equation*}
|g(z) - g(0)| \leq |z/n|^\alpha\left(||g||_* \frac{C(\alpha)}{2} + 2||g||_\infty\right)
\end{equation*}
where \(C(\alpha) > 0\) is a constant depending only on \(\alpha\). 
\end{lemma}

\begin{proof}
The case \(z = 0\) is trivial, so let us fix \(z \in E_n\setminus\{0\}\). Let us first consider the case where \(g(0) = 0\). Consider that \(|D^2_h g(y)| = |g(y+2h) - 2g(y+h) + g(y)|\). By taking \(h = z\) (noting \(|h| \leq n\) since \(|z| \leq n\)) and \(y = 0\), we have by \(g(0) = 0\) that 
\begin{equation*}
 |g(2z) - 2g(z)| = |D_{z}^2g(0)| \leq |z/n|^\alpha ||g||_*. 
\end{equation*}
Let \(m\) be the smallest non-negative integer such that \(2^{-m} \leq |z/n|\). Note that we then have \(|z/n| \leq 2^{-(m-1)}\) and \(|2^kz| \leq n\) for all \(k \leq m-1\). Therefore, iterating and applying the above decomposition successfully by choosing \(h = 2^k z\) at each step (noting \(|h| \leq n\)), observe
\begin{align*}
 |g(2^mz) - 2^mg(z)| &\leq |g(2^mz) - 2g(2^{m-1}z)| + |2g(2^{m-1}z) - 2^2g(2^{m-2}z)| + ... + |2^{m-1}g(2z) - 2^mg(z)| \\
 &= \sum_{k=0}^{m-1} |2^{m-k-1}g(2^{k+1}z) - 2^{m-k}g(2^{k}z)| \\
 &\leq |z/n|^\alpha ||g||_* \sum_{k=0}^{m-1} 2^{m-k-1 + k\alpha} \\
 &\leq |z/n|^\alpha ||g||_* 2^m \sum_{k=0}^{\infty} 2^{-k(1-\alpha) - 1} \\
 &= |z/n|^\alpha ||g||_* 2^m \frac{C(\alpha)}{2}
\end{align*}
for some constant \(C(\alpha)\) depending only on \(\alpha\). Here, we have used \(\alpha < 1\) to conclude \(\sum_{k=0}^{\infty} 2^{-k(1-\alpha)}\) is finite. Therefore, by triangle inequality and \(2^{-m} \leq |z/n|\), we have 
\begin{align*}
 |g(z)| &\leq |z/n|^\alpha ||g||_{*} \frac{C(\alpha)}{2} + \frac{||g||_\infty}{2^m} \\
 &\leq |z/n|^\alpha ||g||_* \frac{C(\alpha)}{2} + |z/n| \cdot ||g||_\infty \\
 &\leq |z/n|^\alpha \left(||g||_* \frac{C(\alpha)}{2} + ||g||_\infty\right)
\end{align*}
where we have used \(|z| \leq n\) and \(\alpha < 1\) to obtain \(|z/n| \leq |z/n|^\alpha\). Since \(z\) was arbitrary, it follows the bound 
\begin{equation*}
  |g(z)| \leq |z/n|^\alpha \left(||g||_* \frac{C(\alpha)}{2} + ||g||_\infty\right)
\end{equation*}
holds for all \(z \in E_n\). 

Let us now consider the general case where \(g(0)\) need not equal zero. Define the function \(\tilde{g}(z) = g(z) - g(0)\). Clearly \(\tilde{g}(0) = 0\), \(||\tilde{g}||_* = ||g||_*\), and \(||\tilde{g}||_\infty \leq 2||g||_\infty\). Repeating the above argument for \(\tilde{g}\) gives \(|g(z) - g(0)| = |\tilde{g}(z)| \leq |z/n|^\alpha \left(||\tilde{g}||_* C(\alpha)/2 + ||\tilde{g}||_\infty \right)= |z/n|^\alpha \left(||g||_* C(\alpha)/2 + 2||g||_\infty\right)\) as desired. 
\end{proof}

\begin{proposition}\label{prop: convolution}
Suppose $\beta \in (0, 1/2)$ and $f_n,g_n:\Z\to\R$ such that $f_n (z)$ and $g_n(z)$ equal $0$ for $z\notin [0,n]\cap \Z$. Suppose, for any $h\in [-n,n] \cap \Z$, both $f_n$ and $g_n$ satisfy
\begin{equation}\label{eqn:convolution_prop_holder}
    |f_n(z+h) - f_n(z)| \leq M |h/n|^\beta
\end{equation}
for all \(z\in [max\{-h,0\}, n-max\{h,0\}] \cap \Z\). Then for all \(z \in [-n, n] \cap \Z\) we have 
\begin{equation*}
    |f_n\ast_D  g_n^-(z) - f_n\ast_D  g_n^-(0)|\leq  n|z/n|^{2\beta} C(\beta, M)
\end{equation*}
where $g_n^-(z) = g_n(-z)$ and $f_n\ast_D  g_n^-$ is the discrete convolution given by $f_n\ast_D  g_n^-(z) = \sum_{k\in \Z} f_n(k)g_n^-(z-k)$ and $C(\beta, M)$ is some constant  depending only on $\beta$ and $M$.
\end{proposition}

\begin{proof}
Define \(D_h\) to be the discrete finite difference operator as in Lemma \ref{lemma:zygmund_holder_interval_discrete}. Consider that,
\begin{align*}
    (D_h(f_n\ast_D  g_n^-))(z)
    &= \sum_{k\in \Z} f_n(k)g_n^-(z+h-k) - \sum_{k\in \Z} f_n(k)g_n^-(z-k)\\
    &= \sum_{k\in \Z} f_n(k)D_hg_n^-(z-k)\\
    &= f_n\ast_D  D_hg_n^-(z).
\end{align*}
By the same argument, we can show $D^2_h(f_n\ast_D  g_n^-) = (D_hf_n)\ast_D(D_hg_n^-)$. First assume $0\leq h\leq n$. Then by Cauchy-Schwarz we have
\begin{align*}
    |D_h^2(f_n\ast_D g_n^-)(z)| 
    &= \left| \sum_{k\in \Z} D_hf_n(k) D_hg_n^-(z-k) \right|\\
    &\leq \sqrt{\sum_{k\in \Z}|D_hf_n(k)|^2}\sqrt{\sum_{k\in \Z}|D_hg_n^-(z-k)|^2}\\
    &= \sqrt{\sum_{k\in \Z}|D_hf_n(k)|^2}\sqrt{\sum_{k\in \Z}|g_n(k-z) - g_n(k-z-h)|^2}\\
    &= \sqrt{\sum_{k\in \Z}|D_hf_n(k)|^2}\sqrt{\sum_{k\in \Z}|g_n(k+h) - g_n(k)|^2}\\
    &= \sqrt{\sum_{k\in \Z}|D_hf_n(k)|^2}\sqrt{\sum_{k\in \Z}|D_hg_n(k)|^2}
\end{align*}

\noindent Note that
\begin{equation*}
    D_hf_n(k)=
   \begin{cases}
     f_n(k+h) &\textit{if } -h\leq k < 0 \\
     f_n(k+h) - f_n(k) &\textit{if } 0\leq k\leq n-h \\
     -f_n(k) &\textit{if } n-h <  k \leq n \\
     0 &\textit{otherwise}.
    \end{cases}
\end{equation*}
A similar expression holds for $D_hg_n$. From (\ref{eqn:convolution_prop_holder}) we have the following bounds,
\begin{equation*}
|D_hf_n(k)|\vee|D_hg_n(k)| \leq
\begin{cases}
    M|h/n|^\beta & \textit{if } 0\leq k\leq n-h\\
    M          & \textit{if } k\in [-h, 0) \cup (n-h, n]\\
    0          &\textit{otherwise}.
\end{cases}
\end{equation*}
Therefore,
\begin{align*}
    |D^2_h(f_n\ast_D  g_n^{-})(z)|
    &\leq \sqrt{\sum_{k\in \Z}|D_hf_n(k)|^2}\sqrt{\sum_{k\in \Z}|D_hg_n(k)|^2}\\
    &\leq M^2|h/n|^{2\beta}(n-h+1)+2M^2|h|\\
    &\leq 2M^2n|h/n|^{2\beta}+2M^2n|h/n|\\
    &\leq 4M^2n |h/n|^{2\beta}
\end{align*}
\noindent where we have used that \(h\) was chosen to satisfy \(0 \leq  h \leq n\). In particular, we have used \(|h/n| \leq 1\) to obtain \(|h/n| + |h/n|^{2\beta}\leq 2|h/n|^{2\beta}\). Using the same argument, we also have, $|D^2_h(f_n\ast_D  g_n^{-})(z)|\leq 4M^2n |h/n|^{2\beta}$ when $-n\leq h < 0$. Since this holds for all \(z \in \Z\), and \(h \in [-n,n] \cap \Z\), and since \(2\beta < 1\), we have by Lemma \ref{lemma:zygmund_holder_interval_discrete} that for all \(z \in [-n, n] \cap \Z\)
\begin{equation*}
\left|(f_n\ast_D  g_n^{-})(z) - (f_n\ast_D  g_n^{ -})(0)\right| \leq n{|z/n|^{2\beta}}C(\beta, M). 
\end{equation*}
In the application of Lemma \ref{lemma:zygmund_holder_interval_discrete}, we have used
$$||f_n\ast_D  g_n^{-}||_\infty = \sup_{z} \left|\sum_{k\in \Z} f_n(k)g_n^{-}(z-k)\right| = 0 \vee \sup_{z} \left|\sum_{k = 0}^n f_n(k)g_n^{-}(z-k) \right| \leq nM^2$$
The proof is complete.
\end{proof}

\end{appendix}

\end{document}